\documentclass[11pt]{amsart}
\usepackage{amsfonts}
\usepackage{yfonts}
\usepackage{amsmath}
\usepackage{amsthm}
\usepackage{amssymb}
\usepackage{mathtools}
\usepackage{graphicx}
\usepackage{float}
\usepackage[caption = false]{subfig}
\usepackage[english]{babel}
\usepackage{verbatim}
\usepackage{enumitem}
\usepackage{graphics}
\usepackage{bbm}
\usepackage{bm}
\usepackage{epstopdf}
\usepackage{marginnote}
\usepackage{wasysym}
\usepackage{pgfkeys}
\usepackage{tikz-cd}
\usepackage{tikz}
\usetikzlibrary{arrows,decorations.pathmorphing,matrix,positioning,calc}
\usepackage[hidelinks]{hyperref}
\usepackage[margin=3.5 cm]{geometry}
\usepackage{marginnote}
\usepackage{extpfeil}
\usepackage{enumitem}

\let\~=\tilde

\newcommand{\R}{\mathbb{R}}
\newcommand{\C}{\mathbb{C}}
\newcommand{\Z}{\mathbb{Z}}

\newcommand{\D}{\mathbb{D}}

\newcommand{\Imm}{\mathrm{Im}}
\newcommand{\hooklongrightarrow}{\lhook\joinrel\longrightarrow}

\newcommand{\wf}{\widehat{\varphi}}
\newcommand{\vf}{\varphi}
\newcommand{\bs}{\boldsymbol}
\newcommand{\Fix}{\mathrm{Fix}}
\newcommand{\Int}{\mathrm{int}}

\newtheorem{Thm}{\textbf{Theorem}}[section]
\newtheorem{Lemma}[Thm]{\textbf{Lemma}}
\newtheorem{Cor}[Thm]{\textbf{Corollary}}
\newtheorem{Prop}[Thm]{\textbf{Proposition}}

\theoremstyle{remark}

\numberwithin{equation}{section}

\newtheorem{Ex}[Thm]{\textbf{Example}}
\newtheorem{Def}[Thm]{\textbf{Definition}}
\newtheorem{Rmk}[Thm]{\textbf{Remark}}

\newtheorem{Cnj}[Thm]{\textbf{Conjecture}}

\title[Knot Floer homology and surface diffeomorphisms]{Knot Floer homology of fibred knots and Floer homology of surface diffeomorphisms}

\author{Paolo Ghiggini}
\address{Laboratoire de Mathématiques Jean Leray, CNRS and Universit\'e de Nantes}
\email{paolo.ghiggini@univ-nantes.fr}

 \author{Gilberto Spano}
 \email{gilbertospano.math@gmail.com}
 
\begin{document}
\setcounter{tocdepth}{4}

\begin{abstract}
We prove that the Knot Floer homology group of a fibred knot of genus $g$ in the Alexander grading $1-g$ is isomorphic to a version of the fixed point Floer homology of an area-preserving representative of the monodromy.
\end{abstract}

\maketitle

\tableofcontents

\section{Introduction}
Knot Floer homology \cite{OS3, Ra} is a a family of abelian groups --- or vector spaces; here
we will work over a field of characteristic two ---  $\widehat{HFK}(Y, K, i)$ 
associated to any oriented null-homologous knot $K$ in an oriented three-manifold $Y$. 
If $g$ is the minimal genus of a Seifert surface of $K$, then $\widehat{HFK}(Y, K, i)=0$
if $|i|>0$ and $\widehat{HFK}(Y, K, i) \ne 0$ for $i=-g,g$ by \cite{OSgenus, Nigenus}. 
Moreover, $\widehat{HFK}(Y, K, -g)$ has rank one if and only if $K$ is fibred by \cite{Ghi, 
Nifibred}.  See also \cite{Ju1}. 

The power of knot Floer homology, which is not at all limited to the results mentioned above,
comes from its connections to many areas of low-dimensional topology, but its topological 
meaning is obscured by the complexity of its definition. In this article we try to shed some light
on it by relating the knot Floer homology of a fibered knot to the dynamics of its 
monodromy. This is a partial result in the direction of
 a relative version of the isomorphism between Heegaard Floer homology and embedded
contact homology; see \cite{CGH2}. 

A fibred knot $K \subset Y$ gives rise to an open book decomposition of $Y$ with binding $K$,
fibre $S$ and monodromy $\vf$; see Section \ref{sec: open books}.
The monodromy is well defined up to isotopy relative to the boundary of $S$. We will fix an
area form on $S$ and assume that $\varphi$ is an area-preserving diffeomorphism. Then we 
can associate to it a Floer homology group $HF^\sharp(\varphi)$, which is an intermediate 
version between the Floer homology groups $HF(\varphi, +)$ and $HF(\varphi, -)$ considered 
by Seidel in \cite{Se4}. In \cite{CGH1}  Colin, Ghiggini and Honda defined an embedded 
contact homology group $ECH^\sharp(T_\varphi)$, which is conjecturally isomorphic to 
$\widehat{HFK}(Y, K)$ when $T_\vf$ is the complement of a tubular neighbourhood of $K$.
The group $HF^\sharp(\varphi)$ is isomorphic to the subgroup of
$ECH^\sharp(T_\varphi)$ generated by the Reeb orbits with algebraic intersection one with 
a fibre.

Let $\overline{K}$ and $\overline{Y}$ denote $K$ and, respectively, $Y$ with the reversed 
orientation. If $Y = S^3$, then $(\overline{Y}, \overline{K})= (Y, m(K))$, where $m(K)$
denotes the mirror of $K$. The main result of this paper is the following:
\begin{Thm} \label{Theorem: Main theorem in introduction} Let $K \subset Y$ be a fibered 
knot with associated monodromy $\vf \colon S \rightarrow S$.
Then there exists an isomorphism of vector spaces over $\Z / 2 \Z$
\[ \widehat{HFK}(\overline{Y},\overline{K};-g+1) \cong  HF^{\sharp}(\vf). \]
\end{Thm}
This isomorphism is induced by an open-close map 
$$\Phi^\sharp_* \colon \widehat{HFK}(\overline{Y},\overline{K};-g+1) \to 
HF^\sharp(\varphi)$$
which is a simplified version of the open-close map $\Phi_* \colon \widehat{HF}(Y) \to
\widehat{ECH}(Y)$ defined in \cite{CGH3}. The proof that $\Phi^\sharp_*$ is an 
ismomorphism is by induction on
the length of a minimal factorisation of $\varphi$ as a product of Dehn twists and, unlike the 
proof of the isomorphism between Heegaard Floer homology and embedded contact homology,
does not require the construction of an inverse map. The initial step of the induction, for
$\varphi=id$, is an explicit computation. The inductive step relies on the comparison between
two exact triangles: Ozsv\'ath and Szab\'o's sergery exact triangle for knot Floer homology 
\cite[Theorem~8.2]{OS3}  and Seidel's exact triangle for Dehn twists in the context of fixed
point Floer homology \cite[Theorem 4.2]{Se4}. We give the first complete proof of the latter
for exact symplectomorphisms of closed surfaces, a result which can be of independent 
interest.

More precisely, an essential circle $L$ embedded in $S$ induces a knot 
in $Y$ via the open book decomposition. We denote by $Y_+$ and $Y_0$ the manifolds 
obtained by $+1$-- and, respectively,  $0$--surgery on $Y$ along $L$ with surgery coefficient 
computed with respect to the framing induced by the page. The knot $K \subset Y$
induces knots $K_+ \subset Y_+$ and $K_0 \subset Y_0$.  Moreover, $K_+$ is the binding of 
an open book decomposition of $Y_+$ with page $S$ and monodromy $\vf \circ \tau_L^{-1}$, 
where $\tau_L$ is a positive Dehn twist around $L$.


Seidel's exact sequence for fixed point Floer homology, in this context, is 
\begin{equation}\label{Triangle: L phi phi-tau without maps}
\begin{tikzpicture}[->,auto,node distance=1.5cm,>=latex',baseline=(current  bounding  box.center)]

  \node (1) {$HF(\vf(L),L)$};
  \node (2) [right = of 1] {$HF^{\sharp}(\vf)$};
  \node (3) [right = of 2] {$HF^{\sharp}(\vf \circ \tau_L^{-1})$};

  \path[every node/.style={font=\sffamily\small}]
    (1) edge node[above] {} (2)
    (2) edge node[above] {} (3)
    (3) edge[bend left=20] node {} (1);
\end{tikzpicture}
\end{equation}
while Ozsv\'ath and Szab\'o's exact sequence for knot Floer homology is
 \begin{equation} \label{Triangle: Exact triangle in HF without maps}
 \begin{tikzpicture}[->,auto,node distance=0.8cm,>=latex',baseline=(current  bounding  box.center)]

  \node (1) {$\widehat{HFK}(\overline{Y_0}; \overline{K_0}; -g+1$)};
  \node (2) [right = of 1] {$\widehat{HFK}(\overline{K},\overline{Y};-g+1)$};
  \node (3) [right = of 2] {$\widehat{HFK}(\overline{K_+},\overline{Y_+};-g+1)$.};

  \path[every node/.style={font=\sffamily\small}]
    (1) edge node[above] {} (2)
    (2) edge node[above] {} (3)
    (3) edge[bend left=20] node {} (1);
 \end{tikzpicture}
 \end{equation}
It is easy to see, using the surface decomposition formula in sutured Floer homology 
\cite[Theorem 1.3]{Ju1},  that $\widehat{HFK}(\overline{Y_0}; \overline{K_0}; -g+1) 
\cong HF(\vf(L), L)$. Then we will prove that the diagram
\begin{equation} \label{Diagram: Main diagram between HF and HF without maps}
\begin{tikzpicture}[->,auto,node distance=1.5cm,>=latex',baseline=(current  bounding  box.center)]
\matrix (m) [matrix of nodes, row sep=1cm,column sep=0.85cm]
{ $\widehat{HFK}(\overline{Y_0}; \overline{K_0}; -g+1)$ & $\widehat{HFK}(\overline{K},\overline{Y};-g+1)$ &  
$\widehat{HFK}(\overline{K}_-,\overline{Y_{1}(L)};-g+1)$ \\
  $HF(\vf(L),L)$ & $HF^{\sharp}(\vf)$ & $HF^{\sharp}(\vf \circ \tau_L^{-1})$. \\ };
 \path[->]
  (m-1-1) edge node[above] {} (m-1-2)
          edge node[right] {} node[anchor=center,rotate=-90,yshift=1.8ex]{$\cong$} (m-2-1);
 \path[->]
  (m-1-2) edge node[right] {$\Phi^\sharp_*$} (m-2-2)
          edge node[above] {} (m-1-3);
 \path[->]
  (m-1-3) edge node[right] {$\Phi^\sharp_*$} (m-2-3)
          edge[bend right=25] node[above] {} (m-1-1);
 \path[->]
  (m-2-1) edge node[below] {} (m-2-2);
 \path[->]
  (m-2-2) edge node[below] {} (m-2-3);
 \path[->]
  (m-2-3) edge[bend left=25] node[below] {} (m-2-1);
\end{tikzpicture}
\end{equation}
commutes. For a suitable choice of $L$, one between $\vf$ and $\vf \circ \tau^{-1}$ has a 
factorisation in Dehn twists which is shorter than the other's, and thus the inductive step
will follow from the Five Lemma.

\color{black}

In the last section we will give some applications of Theorem \eqref{Theorem: Main theorem in introduction}. In particular we show that, under certain hypotheses, the dimension of $\widehat{HFK}(Y,K,-g+1)$ detects the minimal number $F_{min}([\vf])$ of fixed points that an area-preserving non-degenerate representative of $\vf$ may have.
\begin{Thm}
 Let $K \subset Y$ be a genus $g>0$ fibered knot and let $\vf$ be a representative of the monodromy of its complement. Assume that either 
 \begin{enumerate}
  \item $Y$ is a rational homology sphere or
  \item the mapping class $[\vf]$ is irreducible in the sense of the Nielsen--Thurston classification.
 \end{enumerate}
 Then:
  \begin{equation*}
  \dim\left(\widehat{HFK}(Y,K;-g+1)\right) = \left\{ \begin{array}{ll}
                                                                             \mathcal{F}_{min}({[\vf]}) - 1 & \mbox{if } \vf \sim id;\\
                                                                             \mathcal{F}_{min}({[\vf]}) + 1 & \mbox{otherwise}.
                                                                         \end{array} \right. 
 \end{equation*}
\end{Thm}

As a consequence we obtain that if $K$ is a fibered knot satisfying the hypotheses of last theorem then
\[\dim\left(\widehat{HFK}(Y,K;-g+1)\right) \geq 1\]
with equality if and only if the mapping class of monodromy of its complement admits an area-preserving non-degenerate representative with no fix points. This implies the following corollary.
\begin{Cor}
$L$-space knots in $S^3$ admit a representative of the monodromy with no interior fixed points.
\end{Cor} 

Another application of our main theorem gives a link between Heegaard Floer homology and the geometry of three-manifolds. Let $K^n$ be the branched locus of the $n$-th branched cover $Y^n(K)$ of $Y$ over $K$: if $(K,S,\vf)$ is an open book decomposition of $Y$ then $(K^n,S,\vf^n)$ is an open book decomposition of $Y^n(K)$. Combining Theorem \ref{Theorem: Main theorem in introduction} with a result of Fel'shtyn \cite{Fel} we recover the following result, that was already proven in \cite{LOT2} by Lipshitz, Ozsv\'{a}th and Thurston using bordered Floer homology techniques.
\begin{Cor}\label{Corollary: HFK detects the stretching factor intro}
 If $(K,S,\vf)$ is an open book decomposition of $Y$ the growth rate for $n \rightarrow \infty$ of the dimensions of $\widehat{HFK}(Y^n(K),K^n;-g+1)$ coincides with the largest dilatation factor among all the pseudo--Anosov components of the canonical representative of the mapping class $[\vf]$.
\end{Cor}

Another interesting consequence of Theorem \ref{Theorem: Main theorem in introduction} is about algebraic knots and comes from an analogous result of McLean about fix point Floer homology (see \cite[Corollary 1.4]{McLean}).
\begin{Cor}\label{Corollary: HFK detects the multiplicity intro}
 Let $K \subset S^3$ be the $1$--component link of an isolated singularity of an irreducible complex polynomial $f$ with two variables. Then
 \[\min \left\{n>0\ |\ \dim\left(\widehat{HFK}((S^3)^n(K),K^n;-g+1)\right) \neq 1 \right\} = \mathfrak{m}(f) \]
 where $\mathfrak{m}(f)$ is the multiplicity of $f$ in $0$.
\end{Cor}

Similar results have recently been proved by Ni in \cite{str1, str2}.

\section{Preliminaries}
\subsection{Open book decompositions}\label{sec: open books}
Let $K$ be an oriented knot in a closed oriented $3$--manifold
$Y$. Recall that $K$ is fibred if there exists a neighborhood
$U \cong K \times \mathbb{D}^2$ of $K$ and a fibration
$Y\setminus U \rightarrow S^1$ that extends the fibration
$$U\setminus K \cong K \times (\mathbb{D}^2 \setminus \{0 \}) \rightarrow S^1; \quad (x, z) 
\mapsto \frac{z}{|z|}.$$  
This implies that there exists an oriented surface $S$ with boundary and an orientation
preserving diffeomorphism $\vf \colon S \rightarrow S$ such that
$\vf|_{\partial S}$ is the identity and $Y\setminus U$ is diffeomorphic
to the mapping torus
\[T_{\vf} = \frac{\R \times S}{(t+1,x) \sim (t,\vf(x))}.\]
The images of the fibers
$S_t \coloneqq \{t\} \times S$ under this diffeomorphism extend canonically to Seifert
surfaces for $K$. The triple $(K,S,\vf)$ is often called an \emph{open
  book decomposition} of $Y$ with \emph{binding} $K$, \emph{page} $S$
and \emph{monodromy} $\vf$. We will refer to the genus of
$S$ as to the genus of the open book or the genus of the knot $K$.

Let $Y$ be a closed oriented $3$-manifold endowed with an open book decomposition 
$(K,S,\vf)$. We parametrise a collar neighborhood 
$A \cong [0,2] \times S^1$ of $\partial S$ by $\{(y,\vartheta) \in [0,2] \times \frac{[0,
2\pi]}{0 \sim 2\pi}\}$ so that $\partial S$ corresponds to $\{ 2 \} \times  \frac{[0,
2\pi]}{0 \sim 2\pi}$. On $S$ we fix once and for all a Liouville form $\lambda$ such that,
on $A$, it can be written as $y d \theta$ and denote $\omega=d \lambda$. We also assume, 
without loss of generality, that the monodromy satisfies 
\begin{equation}\label{exactness of phi}
\vf^* \lambda - \lambda = d\chi
\end{equation}
for a compactly supported function $\chi \colon S \to \R$ (see \cite[Lemma 9.3.2]{CGH1}) and,
  moreover, for every $(y,\vartheta) \in A$,
\begin{equation} \label{phi near the boundary}
\vf(y,\vartheta) = (y,\vartheta + f(y))
\end{equation}
 for a function $f \in \mathcal{C}^{\infty}([0,2])$ such that 
\begin{itemize}
\item $f(2) =0$, 
\item $f(y) = 3\varepsilon$ for some $\varepsilon \in \left (0, \frac{\pi}{g} \right )$ when $y \in [0,1]$, and
\item  $f'(y) < 0$ when $y\in (1,2)$.
\end{itemize}


\subsection{Heegaard Floer homology via open books} \label{Subsection: 
Heegaard Floer homology for open books}

As shown in \cite{HKM} by Honda, Kazez and Mati\'{c}, to an open book decomposition of 
$Y$ it is possible to associate a special kind of Heegaard diagram for $Y$. Let us recall the slightly 
different construction given by Colin, Honda and the first author in \cite{CGH3}.

If $S$ is a surface of genus $g$, a \emph{basis of arcs of $S$} is a set $\bm{a}=\{a_1,\ldots,
a_{2g}\}$ of smooth properly embedded arcs such that $S \setminus \{a_1,\ldots,a_{2g}\}$ is 
a topological disk. Given the open book decomposition $(K,S,\vf)$, the surface 
$$\Sigma \coloneqq S_{\frac{1}{2}}\sqcup_{\partial_{S_0}} \overline{S}_0$$
is a Heegaard surface for $Y$, where $\overline{S}_0$ denotes $S_0$ with the orientation 
reversed. Fix a basis of arcs $\{a_1,\ldots,a_{2g}\}$ for $S_0$, let $\{\underline{a}_1, \ldots, 
\underline{a}_{2g} \}$ be a copy of the basis in $S_{\frac{1}{2}}$ and define the set of closed 
curves
$\bm{\alpha}= \{\alpha_1,\ldots,\alpha_{2g}\} \subset \Sigma$ by 
$$\alpha_i=\underline{a}_i \sqcup_{\partial a_i}a_i,$$
conveniently smoothed near $\partial_{S_0}$.

Consider now the set of arcs $\{ \underline{b}_1, \ldots, \underline{b}_{2g} \}$ in 
$S_{\frac{1}{2}} \subset \Sigma$ obtained by modifying the $\underline{a}_i$'s by a small 
isotopy relative to the boundary such that, for every $i \in \{1,\ldots,2g\}$:
\begin{itemize}
\item $\underline{a}_i \cap \underline{b}_i = \{ c_i, c_i', c_i'' \}$, where $c_i''$ is in the interior of $S_{\frac{1}{2}}$, $c_i$ and $c_i'$ are on the boundary and all intersections are transverse,
\item if we orient $\underline{a}_i$ and give $\underline{b}_i$ the orientation induced 
from that of $\underline{a}_i$, then an oriented basis of $T_{c_i''}\underline{b}_i$ followed 
by an oriented basis of $T_{c_i''}\underline{a}_i$ yields an oriented basis of $T_{c_i''}\Sigma$,
 \item $\underline{a}_i \cap \underline{b}_j = \emptyset$ for $j\neq i$ and
 \item in a neighborhood of $\partial S_{\frac{1}{2}} \subset \Sigma$, $\underline{b}_i$ is a 
smooth extension to $S_\frac{1}{2}$ of $\vf(a_i) \subset S_0$.
\end{itemize}
Then we define the set of curves $\bm{\beta} = \{\beta_1,\ldots,\beta_{2g}\} \subset \Sigma$ by 
\[\beta_i=\underline{b}_i\sqcup_{\partial \vf(a_i)}\vf(a_i).\]

\begin{figure} [ht!] 
  \begin{center}
  \includegraphics[scale = 1.7]{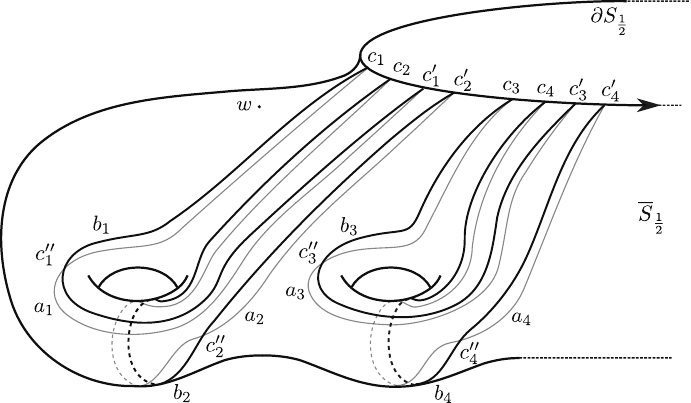}
  \end{center}
  \caption{A basis of arcs in $\overline{S}_{\frac{1}{2}}$  and its perturbation.}
\end{figure}

If we choose a basepoint $w \in S_\frac{1}{2}$ outside of the thin strips given by the
 isotopies from the $\underline{a}_i$'s to the $\underline{b}_i$'s, then $(\Sigma,
\bm{\alpha},\bm{\beta},w)$ is a pointed \emph{Heegaard diagram for $Y$ compatible with 
$(K,S,\vf)$}.
However, we prefer to work with a Heegaard diagram where the orientation of 
the Heegaard surface matches the orientation of $S_0$, and therefore we will consider the
Heegaard diagram $(\overline{\Sigma}, \bm{\alpha}, \bm{\beta}, z)$ for $\overline{Y}$.
This Heegaard diagram is clearly weakly admissible: see \cite{HKM}.



Let $\mathfrak{S}_{k}$ denote the symmetric group with $k$ elements.
We recall that the Heegaard Floer chain complex $\widehat{CF}(\overline{\Sigma}, 
\bm{\alpha}, \bm{\beta}, w)$ is defined as the vector space over 
$\Z /2 \Z$ generated by the $2g$-tuples of intersection points $\mathbf{x} = \{x_1, \ldots,
x_{2g}\}$ such that $x_i \in \alpha_i \cap \beta_{\sigma(i)}$ for some $\sigma \in 
\mathfrak{S}_{2g}$. 


To define the Heegaard Floer differential we will use Lipshitz's cylindrical reformulation from \cite{Li1} with the conventions of \cite{CGH3}. A generator $\mathbf{x} = 
\{x_1,\ldots,x_{2g}\}$ of $\widehat{CF}(\overline{\Sigma}, \bm{\alpha}, \bm{\beta}, w)$ 
can be identified with the set of $2g$ chords $[0,1] \times \{x_i\}_{i=1}^{2g} \subset [0,1] 
\times \Sigma$. 
We endow $\R \times [0,1] \times \Sigma$ with the symplectic form
\[ds \wedge dt + \omega_{\overline{\Sigma}},\]
where $s$ and $t$ are the coordinates of $\R$ and $[0,1]$ respectively and 
$\omega_{\overline{\Sigma}}$ is an area form on $\overline{\Sigma}$ which restricts to the area 
form $\omega$ on $S$, and choose  an \emph{admissible} almost complex structure $J$ 
(see \cite[Definition 4.2.1]{CGH3}). 

For every $i \in \{1,\ldots,2g\}$, call $L_{\alpha_i}$ and $L_{\beta_i}$ the Lagrangian 
submanifolds $\R \times \{1\} \times \alpha_i$ and, respectively, $\R \times \{0\} \times 
\beta_i$ of $\R \times [0,1] \times \overline{\Sigma}$. Define moreover $L_{\bm{\alpha}} = 
\bigsqcup_{i=1}^{2g} L_{\alpha_i}$ and $L_{\bm{\beta}} = \bigsqcup_{i=1}^{2g} L_{\beta_i}$.

Let $(F,j)$ be a compact (possibly disconnected) Riemann surface with two sets of punctures 
$\mathbf{p^+}= \{p^+_1,\ldots,p^+_k\}$ and $\mathbf{p^-}= \{p^-_1,\ldots,p^-_k\}$ on 
$\partial F$ such that 
\begin{itemize}
\item[(i)] every connected component of $F$ has nonempty boundary, and
\item[(ii)] every connected component of $\partial F$ contains at least one element of 
$\mathbf{p^+}$ and one of $\mathbf{p^-}$, and 
\item[(iii)] the elements of $\mathbf{p^+}$ and $\mathbf{p^-}$ alternate along $\partial F$. 
\end{itemize}
Let $\dot{F}$ denote $F$ with the sets of punctures $\mathbf{p^+}$ and $\mathbf{p^-}$
removed.

\begin{Def} \label{Definition: Multisections of R times I times Sigma }
Let $\mathbf{x} = \{x_1,\ldots,x_k\}$ and $\mathbf{y} = \{y_1,\ldots,y_k\}$ be two 
$k$-tuple ($k\leq 2g$) of points in $\Sigma$ with $x_i \in \alpha_i \cap \beta_{\sigma(i)}$ 
and $y_i \in \alpha_i \cap \beta_{\sigma'(i)}$ for some permutations $\sigma, \sigma' \in 
\mathfrak{S}_k$. A \emph{degree-$k$ multisection of $\R \times [0,1] \times \Sigma$
 from $\mathbf{x}$ to $\mathbf{y}$} is a $J$-holomorphic map 
$$u:(\dot{F},j) \longrightarrow (\R \times [0,1] \times \Sigma,J)$$
satisfying the following conditions:
\begin{enumerate}
 \item $(\dot{F},j)$ is a punctured Riemann surface as above;
 \item $u(\partial \dot{F}) \subset L_{\bm{\alpha}} \cup L_{\bm{\beta}}$ and each 
connected component of $\partial \dot{F}$ is mapped to a different $L_{\alpha_i}$ or $L_{\beta_i}$;
 \item $\lim_{w\rightarrow p_i^+} u_{\R}(w) = +\infty$ and $\lim_{w\rightarrow p_i^-} u_{\R}(w) 
= -\infty$ where $u_{\R}$ is the component of $u$ to $\R$;
 \item near $p_i^+$ (respectively $p_i^-$), $u$ is asymptotic to the strip over $[0,1] \times 
\{x_i\}$ (respectively $[0,1] \times \{y_i\}$);
\end{enumerate}
\end{Def}


For a $J$-holomorphic map $u$ as above, we define $n_w(u)$ as the algebraic intersection 
number between the image of $u$ and the $J$-holomorphic section $\R \times [0,1] \times 
\{w\}$. By positivity of intersection $n_w(u) \ge 0$.

We define $\widehat{\mathcal{M}}_1(\mathbf{x},\mathbf{y})$ as the set of equivalence 
classes (modulo reparametrisations and $\R$-translations) of holomorphic multisections
$u :(\dot{F},j) \longrightarrow (\R \times [0,1] \times \Sigma,J)$ of degree $2g$ from $\mathbf{x}$ 
to $\mathbf{y}$ such that:
\begin{enumerate}
 \item $n_w(u) = 0$  and
 \item $u$ has Lipshitz's $ECH$-type index $1$ (see \cite[Section 4]{CGH3} for details).
\end{enumerate}
Note that, for multisections of $\R \times [0,1] \times \Sigma$, having $ECH$-type index 
$1$ is equivalent to being embedded and having Fredholm index $1$.

\color{black}

The Heegaard Floer differential (in the hat version) is then defined by
\[\partial^{HF}(\mathbf{x}) = \sum_{\mathbf{y}} \#_{2} 
\widehat{\mathcal{M}}_1(\mathbf{x},\mathbf{y},J) \mathbf{y}, \]
where the sum is taken over the set of generators of $\widehat{CF}(\overline{\Sigma},
\boldsymbol{\alpha}, \boldsymbol{\beta}, w)$ and $\#_2$ denotes the cardinality modulo 
$2$. The homology of $\widehat{CF}(\overline{\Sigma},
\boldsymbol{\alpha}, \boldsymbol{\beta}, w)$ does not depend on the various choices and
is denoted by $\widehat{HF}(\overline{Y})$. 

In \cite[Section 4]{CGH3} it was shown that $\widehat{HF}(\overline{Y})$ can be computed using 
only the part of the diagram $(\overline{\Sigma},\bm{\alpha},\bm{\beta},w)$
contained in $S_0$. To recall the construction, observe first that if $u$ is a connected degree $k$   multisection with $n_w(u)=0$ and a positive end at a chord associated to $c_i$ or $c_i'$, then $k=1$ and $u$ is a trivial strip over that chord, i.e. a reparametrisation of either $\R \times [0,1] \times \{ c_i \}$ or $\R \times [0,1] \times \{ c_i' \}$.

Let $CF'(S,\bm{a},\vf(\bm{a}))$ be the submodule of $\widehat{CF}(\overline{\Sigma},
\boldsymbol{\alpha}, \boldsymbol{\beta},w)$ generated by the $2g$-tuples of intersection 
points contained in $S_0$ and endow it with the restriction of $\partial$. 
$CF'(S,\bm{a},\vf(\bm{a}))$ is a subcomplex of 
$\big(\widehat{CF}(\overline{\Sigma},\bm{\alpha},\bm{\beta}, w),\partial\big)$. 

Let ${\mathcal R} \subset CF'(S,\bm{a},\vf(\bm{a}))$ be the subspace generated by all 
elements $\mathbf{x}- \mathbf{x}'$ where there exists $i$ such that
$\mathbf{x}=\{x_1, \ldots,c_i, \ldots,x_{2g}\}$ and $\mathbf{x}'=\{x_1, \ldots,c_i', 
 \ldots,x_{2g}\}.$ Then define 
\begin{equation}\label{eq: quotient}
\widehat{CF}(S,\bm{a},\vf(\bm{a})) = CF'(S,\bm{a}, \vf(\bm{a})) / {\mathcal R}.
\end{equation}
In \cite[Subsection 4.9]{CGH3} it was proved, with a slightly different terminology, that 
${\mathcal R}$ is a subcomplex of $CF'(S,\bm{a}, \vf(\bm{a}))$, and therefore the 
quotient $\widehat{CF}(S,\bm{a},\vf(\bm{a}))$ is a chain complex with the induced 
differential.
We will call $\widehat{HF}(S,\bm{a},\vf(\bm{a}))$
its homology.
\begin{Thm}(see \cite[Theorem 4.9.4]{CGH3})\label{Theorem: HF on a page}
 $$\widehat{HF}(S,\bm{a},\vf(\bm{a})) \cong \widehat{HF}(\overline{Y}).$$
\end{Thm}


\subsection{Knot Floer homology via open book decompositions} 
\label{Subsubsection: The knots filtrations}
In this section we show how the hat version of knot Floer homology can be computed using 
the diagram $(S, \mathbf{a}, \vf(\mathbf{a}))$. To obtain this, however, it is necessary to 
be slightly more careful in the construction of the diagram. 
 Let $\varepsilon$ be the same constant as in Equation \ref{phi near the boundary}. We 
assume that, for every $i \in \{1,\ldots,2g\}$,
\begin{enumerate}
 \item $\{c_i,c_i'\} = a_i \cap \partial S \subset \{2\} \times (0,\varepsilon)$, and 
 \item $a_i \cap [1,2] \times \partial S = [1,2] \times\{c_i,c_i'\}$.
 \end{enumerate}
These conditions can be achieved by a Hamiltonian isotopy of the arcs  $a_i$. 

\begin{figure} [ht!] 
  \begin{center}
  \includegraphics[scale = 0.45]{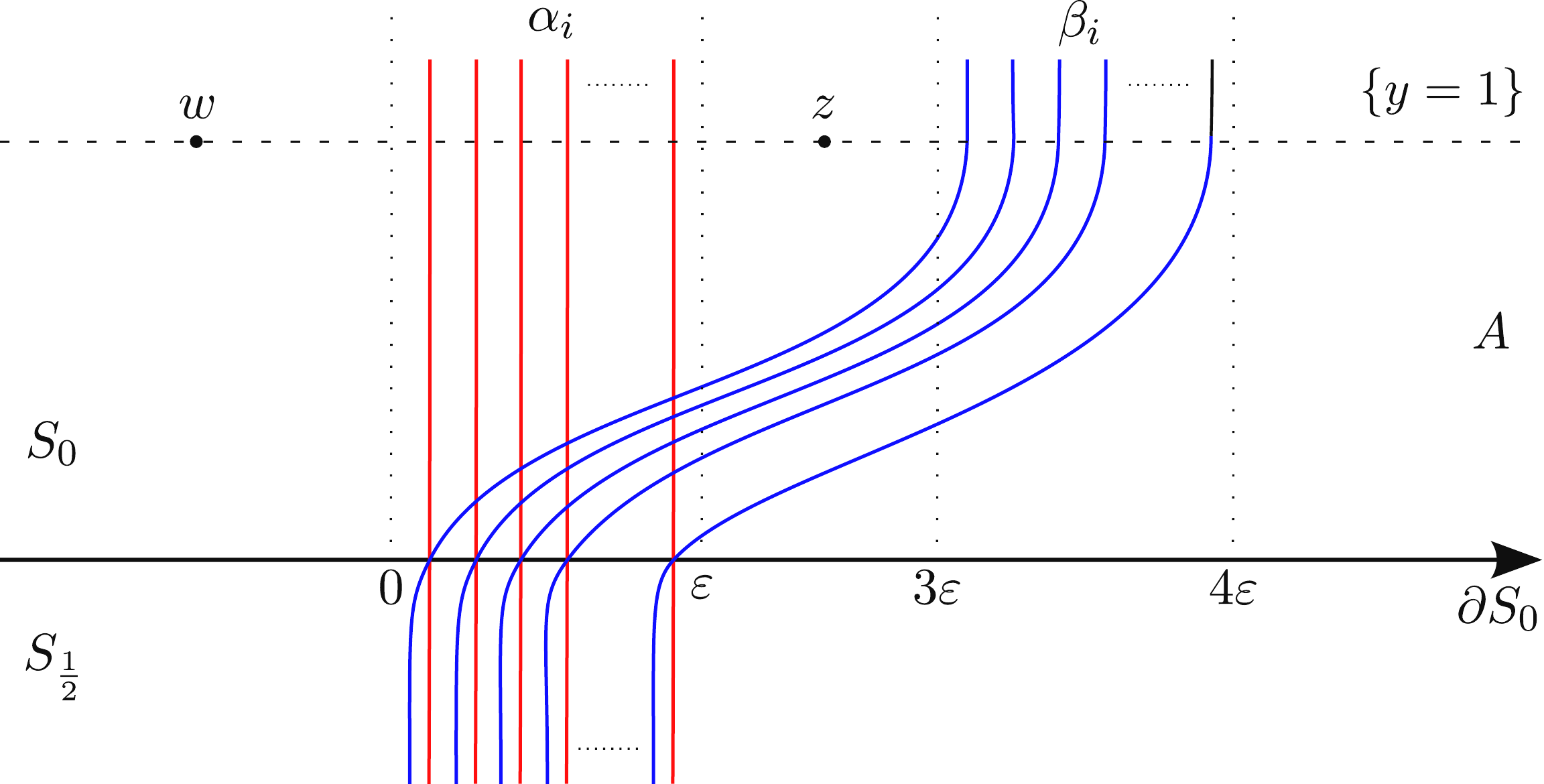}
  \end{center}
  \caption{A diagram compatible with the knot in a neighborhood of $\partial S_0$.}
 \label{fig:trivial region}
\end{figure}


If $z = (1,2\varepsilon)$, then $(\overline{\Sigma},\bm{\alpha},\bm{\beta},z,w)$ is a a 
doubly pointed Heegaard diagram for $(\overline{K},\overline{Y})$.
Let $\overline{Y}_0(K)$ be the zero surgery on $Y$ along $K$. 
Following Ozsv\'{a}th and Szab\'{o} \cite{OS3}, given a generator $\mathbf{x}$ of 
$\widehat{CF}(\overline{\Sigma},\bm{\alpha},\bm{\beta},w)$, its Alexander degree
$\mathcal{A}(\mathbf{x})$ with respect to $S_0$ is the integer
\begin{equation} \label{Equation: Alexander degree in terms of spin^c structures}
 \mathcal{A}(\mathbf{x}) = \frac{1}{2}\langle c_1(\underline{\mathfrak{s}}_w(\mathbf{x})),[\widetilde{S}_0] \rangle 
\end{equation}
where $\underline{\mathfrak{s}}_w(\mathbf{x})$ is the $Spin^c$--structure on $\overline{Y}_0(K)$ determined by $\mathbf{x}$ and $w$ and $\widetilde{S}_0$ is the surface $S_0 \setminus A$ capped off in $\overline{Y}_0(K)$. By standard computations using periodic domains (see for example \cite[Section 7.1]{OS2}, and also \cite[Lemma 6.1]{BVV} for a similar computation) one can check that 
\begin{equation} \label{Equation: Alexander degree in terms of components on S_0}
 \mathcal{A}(\mathbf{x}) = -g + \#\left(\mathbf{x} \cap \big(S_0 \setminus A\right)\big).
\end{equation}
Moreover the positivity of intersection between holomorphic curves in dimension $4$ implies that, if $u$ is a multisection from $\mathbf{x}$ to $\mathbf{y}$, then: 

\[\mathcal{A}(\mathbf{x}) - \mathcal{A}(\mathbf{y}) = n_z(u) \geq 0. \]
For this reason the Alexander grading ${\mathcal A}$ induces a filtration on 
$\widehat{CF}(\overline{\Sigma}, \bm{\alpha}, \bm{\beta}, w)$. The knot Floer homology 
complex of $(\overline{Y}, \overline{K})$ is the associated graded 
complex 
\[\left(\widehat{CFK}(\overline{\Sigma}, \bm{\alpha},\bm{\beta}, z, w),\partial_K \right) 
\coloneqq \bigoplus_{i \in \Z} \left( \widehat{CFK}(\overline{\Sigma}, \bm{\alpha}, 
\bm{\beta}, z, w, i),\partial_i\right), \]
where 
$\widehat{CFK}(\overline{\Sigma}, \bm{\alpha}, 
\bm{\beta}, z, w, i)$ is the subspace of $\widehat{CF}(\overline{\Sigma}, \bm{\alpha}, 
\bm{\beta}, w)$ generated by the $2g$-tuples of intersection points $\mathbf{x}$ with
$\mathcal{A}(\mathbf{x})=i$ and
\[ \partial_i \colon \widehat{CFK}(\Sigma,\bm{\beta},\bm{\alpha},z,w,i) \longrightarrow 
\widehat{CFK}(\overline{\Sigma}, \bm{\alpha},\bm{\beta},z,w,i) \]
is the restriction to $\widehat{CFK}(\Sigma,\bm{\beta},\bm{\alpha},z,w,i)$ of the 
component of $\partial$ that preserves the Alexander degree. The resulting homology
\[\widehat{HFK}(\overline{Y},\overline{K}) = \bigoplus_{i \in \Z}\widehat{HFK}(\overline{Y},
\overline{K};i) = \bigoplus_{i \in \Z}H_*\left(\widehat{CFK}(\overline{\Sigma}, \bm{\alpha},
\bm{\beta},z,w,i),\partial_i\right)\]
is the knot Floer homology of $\overline{K}$ in $\overline{Y}$.


The restriction to generators and holomorphic curves that are contained in the $S_0$ part of 
the diagram and the quotient \eqref{eq: quotient} are compatible with the Alexander grading.

The result is the chain complex
\[\left(\widehat{CFK}(S,\bm{a},\vf(\bm{a}),z),\partial_K \right) = \bigoplus_{i \in \Z} \left(\widehat{CFK}(S,\bm{a},\vf(\bm{a}),z;i),\partial_i\right) \]
whose homology is isomorphic to $\widehat{HFK}(\overline{Y},\overline{K})$. The proof is the same of that of \cite[Theorem 4.9.4]{CGH3}.
The base point $z$ will often be dropped from the notation, as we did for $w$, because it is
placed  in the region $A$ where all our diagrams will have the standard form 
described above.



\subsection{Fixed point Floer homology} \label{Subsection: Review of symplectic 
Floer homology}
Let $M$ be a closed manifold endowed with a symplectic form and let $\psi \colon M \rightarrow M$ be a symplectomorphism. If $\psi$ satisfies suitable properties (see below for some of the details), one can define the fixed point Floer homology $HF(\psi)$ of $\psi$ as the homology of a finite dimensional chain complex whose generators are the fixed points of $\psi$. This homology was defined by Floer (\cite{Fl}) in the case $\psi$ is Hamiltonian isotopic to $id_M$ and by Dostoglou and Salamon (\cite{DS}) in the general case.
In general $HF(\psi)$ is an invariant of the Hamiltonian isotopy class of $\psi$. However, if $\dim(M)=2$, Seidel showed in \cite{Se2} that $HF(\psi)$ is in fact a topological invariant of the mapping class $[\psi]$.

Fix now $(K,S,\vf)$ as in Section \ref{sec: open books}. To define the fixed point Floer homology of $\vf$, in order to get a finitely generated chain complex, one needs first to perturb the infinite family of fixed points given by $\vf|_{\partial S} = id$. One way to do this is to compose $\vf$ with a small rotation along $\partial S$: this gives two versions $HF(\vf,+)$ and $HF(\vf,-)$ of Floer homologies of $\vf$, which correspond to choosing a positive or, respectively, negative rotation along $\partial S$ (with respect to the orientation induces by $S$). See for example \cite{Se4} for details.
In this section we introduce an intermediate version of fixed point Floer homology $HF^\sharp(\vf)$ 


Let $\mathrm{Fix}(\vf)$ be the set of fixed points of $\vf$. There is a natural identification 
between $\mathrm{Fix}(\vf)$ and the set of closed orbits of period one of the vector field 
$\partial_t$ in the mapping torus $T_{\vf}$: to a fixed point $x$ corresponds the 
periodic orbit $\gamma_x$ through $x$.
Up to Hamiltonian isotopy, we can suppose that $\vf$ has only \emph{non--degenerate} 
fixed points in the interior of $S$. We recall that a fixed point $x$ is 
non-degenerate if $\det(\mathbbm{1} - d_x\vf) \neq 0$. This implies, in particular, that 
$\vf$ has finitely many points in the interior of $S$. 
Recall that a non degenerate fixed point $x$, and the corresponding orbit $\gamma_x$, are 
\emph{elliptic} if $d_x \vf$ has complex conjugated eigenvalues and is \emph{positive} 
or, respectively, \emph{negative hyperbolic} if the eigenvalues of $d_x \vf$ are positive or, 
respectively, negative real.

The boundary $\partial S$ is a Morse-Bott circle of degenerate fixed points and, 
correspondingly, $\partial T_\vf$ is a Morse-Bott torus of degenerate orbits. 
We will describe here only the effect of the Morse-Bott perturbation of $\partial T_\vf$,  referring the reader to Bourgeois \cite{Bour} for a more complete discussion of
the subject and to Colin, Ghiggini and Honda \cite{CGH1} for one more adapted to the situation at hand.

The Morse--Bott perturbation takes place in the mapping torus of an
extension of $\vf$ to a larger surface $\widehat{S}$ that
we now describe (cf. \cite[Section 2]{CGH3}). Let
$(y,\vartheta) \in [1,2] \times S^1$ and
$f \in \mathcal{C}^{\infty}([1,2])$ be the system of coordinates and,
respectively, the function introduced in Section
\ref{sec: open books}. We define $\widehat{S} \coloneqq S \cup ([2,4] \times S^1)$ where
$\partial S$ is glued to $\{2\} \times S^1$ and extend the coordinates
$(y,\vartheta)$ to $[2,4] \times S^1$ in the natural way. We extend also
$f$ to a function ${f} \in \mathcal{C}^{\infty}([1,4])$ with
${{f}}'(y)<0$ for all $y \in (1,4]$ and
$\frac{\pi}{g} < {f}(4) < 0$. Then we extend $\vf$ to a
symplectomorphism (still denoted by $\vf$) of $\widehat{S}$ by setting
${\vf}(y,\vartheta) = (y,\vartheta + f(y))$ $\forall y \in [1,4]$. Let
$\widehat{T}_{\vf}$ be the mapping torus of
$(\widehat{S},{\vf})$. Observe that the properties of ${f}$ imply that
${\vf}$ has no periodic orbits of period smaller or equal to $2g$
crossing a page in the region $\{y \in (2,4]\}$. As for the fibers of
$T_{\vf}$, we denote $\widehat{S}_{t} \coloneqq \{t\} \times \widehat{S} \subset
\widehat{T}_{\vf}$.



We also extend $\omega$ to $\widehat{S}$ using the Liouville structure
and let $\omega_v$ denote the closed two-form on $\widehat{T}_\vf$ which
restricts to $\omega$ on every fiber and such that $\iota_{\partial_t} \omega_v=0$. Then
$(\omega_v, dt)$ is a stable Hamiltonian structure on $\widehat{T}_\phi$ with Reeb vector 
field $\partial_t$. The goal 
of the Morse-Bott perturbation is to replace $(\omega_v, \partial_t)$ with a new stable 
Hamiltonian structure $(\widehat{\omega}_v, dt)$ with Reeb vector field $\widehat{R}$ 
such that all periodic orbits of the flow of $R$ of period at most $2g$ are nondegenerate.
The first return map of the flow of $\widehat{R}$ is a symplectomorphism $\wf$ of 
$(S, \omega)$ which is Hamiltonian isotopic to $\vf$.

The Morse--Bott perturbation 
can be made with support in the neighborhood
$\{y \in (1,3)\}$ of $\partial T_{\vf}$. Before the perturbation each
of the boundary parallel tori
$T_{y_0} \coloneqq \{(t,y,\vartheta) \in \widehat{T}_{\vf}\ |\ y =
y_0\}$,
for $y_0 \in [1,4]$, is linearly foliated by orbits of $\partial_t$
and $T_2 = \partial T_{\vf}$ is the only one that is foliated by
orbits with period smaller or equal to $2g$.
After the perturbation the only periodic
orbits with period smaller than or equal to $2g$ and crossing the region
$\{y \in [1,4]\}$ are two period--1 orbits $e$ and $h$ contained in
$\partial T_{\vf}$. The tori $T_{y}$ which are contained in the support of the perturbation are not foliated by trajectories of $\widehat{R}$. On the other hand, a new family of
tori foliated by trajectories of $\widehat{R}$ is created (see Figure
\ref{Figure: Modifica MB}) and each of these tori bounds a solid torus
in $\widehat{T}_{\vf}$ with core curve $e$.
\begin{figure} [ht!] 
  \begin{center}
  \includegraphics{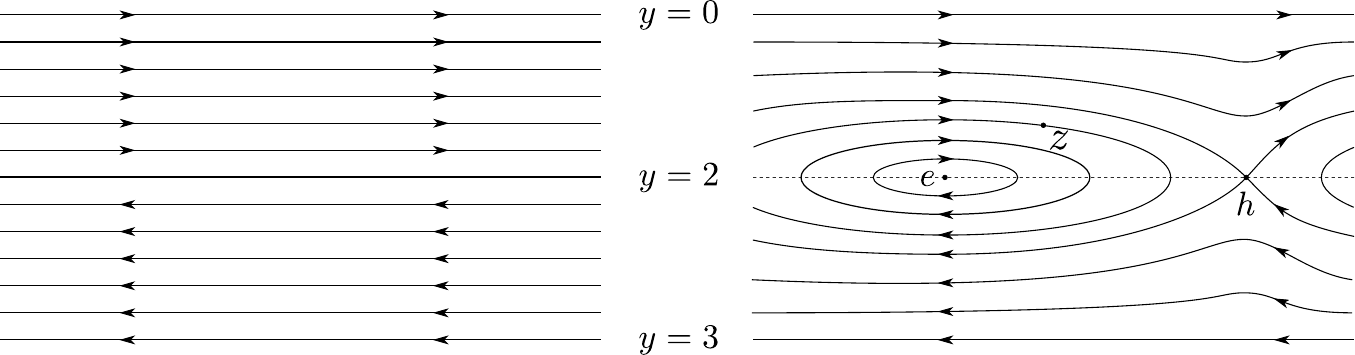}
  \end{center}
  \caption{The dynamics before and after the Morse--Bott perturbation
    of the torus $\partial T_{\vf}$. Both pictures represent the
    annulus $\{y \in [0,3]\}$ in $\widehat{S}_0$ (the left side has to
    be identified with the corresponding right side).  Each flow line
    represents an invariant subset of $\widehat{S}_0$ under the first
    return map of $\partial_t$ or $\widehat{R}$ and the arrows give
    the direction in which any point is mapped. In particular, on both
    annuli, each closed non-singular flow line is the intersection
    between the page and an embedded torus foliated by flow
    trajectories of the corresponding vector field.}
  \label{Figure: Modifica MB}
\end{figure}

\begin{Rmk}
  The reason for which the two orbits above are called $e$ and $h$ is
  that the first is elliptic and the second is (positive)
  hyperbolic. We will denote $p_e$ and $p_h$ the corresponding fixed
  points of $\wf$.
\end{Rmk}
\begin{Rmk} \label{remark: exactness of omega_v} 
The form $\widehat{\omega}_v$ is exact. In fact $\omega_v$ is exact because Equation 
\ref{exactness of phi} allows us to construct a primitive by interpolating between $\beta$ 
and $\vf^* \beta$, and $\widehat{\omega}_v = \omega_v + dH \wedge dt$ for a function 
$H \colon \widehat{T}_\vf \to \R$.
\end{Rmk}
We define $CF^\sharp(\vf)$ as the vector space generated over $\Z/2 \Z$  by 
$\mathrm{Fix}(\wf) \setminus \{ p_e \}$. We endow it with a differential as follows.  
Consider the symplectc fibration $(\R \times \widehat{T}_{\vf},\Omega)$ 
over $\R \times S^1$ with symplectic form $\Omega = ds\wedge dt + \widehat{\omega}_v)$ 
and fiber $\widehat{S}$, and endow it with an $\R$--invariant $\Omega$--tame 
almost complex structure $J$ such that $J(\partial_s) = \widehat{R}$ and $J(T\widehat{S}) 
= T\widehat{S}$. If $x_+, x_-$ are fixed points of $\wf$, let 
$\mathcal{M}(x_+,x_-,J)$ be the moduli space of $J$--holomorphic sections $u\colon \R
\times S^1 \rightarrow \R \times \widehat{T}_{\vf}$ such that
\[ \lim_{s \rightarrow \pm\infty} u(s,\cdot) = \gamma_{x_\pm}.\]

To each $J$--holomorphic cylinder $u$ is associated a Fredholm operator $D_u$ of index 
$\mathrm{ind}(u)$. Call $\mathcal{M}_k(x,y,J)$ the subset of $\mathcal{M}(x,y,J)$ with 
$\mathrm{ind}(u)=k$.
We define the differential $\partial^\sharp$ on $CF^\sharp(\vf)$ by
\[\partial^\sharp (x) = \sum_{y\in \mathrm{Fix}(\wf) \setminus \{ p_e \}} \#_2 
\overline{\mathcal{M}}_1(x,y,J) y \]

where $\#_2$ denotes the cardinality modulo $2$ and $\overline{\mathcal{M}}_1(x,y,J)$ 
is $\mathcal{M}_1(x,y,J)$ quotiented by the $\R$--action given by translations in the 
$\R$-direction. The map $\partial^\sharp$ is well defined because 
exactness of the form $\widehat{\omega}_v$ (see Remark \ref{remark: exactness of 
omega_v}) implies the compactness of the moduli spaces
$\overline{\mathcal{M}}_1(x,y,J)$ and the finiteness of the sum. Moreover it is a differential
because every holomorphic cylinder in $\R \times \widehat{T}_\vf$ with positive end at $e$ is 
a trivial cylinder over $e$; see   \cite[Section 7]{CGH1}. 

The {\em sharp version} of the fixed point Floer homology of $\vf$ is then
\[HF^\sharp(\vf) \coloneqq H_*(CF^\sharp(\vf),\partial^\sharp).\]
The argument of \cite[Section 3]{Se2} can be easily adapted to show that $HF^\sharp(\vf)$ is an
invariant of the mapping class of $\vf$ in the mapping class group of $(S, \partial S)$.

\begin{Rmk}
The definition of $HF^\sharp$ given here is inspired by the definition of $ECH^\sharp$ in 
\cite[Section 7]{CGH1}. An equivalent definition has been given by A. Kotelsky in \cite{Kot}.
\end{Rmk}
\color{black}

\subsection{Periodic Floer homology} \label{Subsection: PFH} 
The construction of symplectic Floer homology can be generalized to
consider orbits of $\vf$ of higher period: the resulting invariant is
the \emph{periodic Floer homology} of $\vf$. We will recall the
definition of two versions of this homology defined when $\vf$ is the
monodromy of a fibered knot $K$ in a three-manifold $Y$: the first, denoted
$PFH_{2g}(T_{\vf})$, is defined in this subsection and is an invariant of $Y$ 
(cf. \cite{CGH4}); the second version, denoted
$PFH^\sharp_i(T_\vf)$, for $i= -g, \ldots, g$, will be defined in the next subsection and
is  an invariant of $(Y, K)$ (cf. \cite[Section 10]{CGH2}). For the details we refer
the reader to \cite[Section 3]{CGH3} and \cite{HS}.

Consider the stable Hamiltonian structure $(\widehat{\omega}_v, dt)$ on 
$\widehat{T}_{\vf}$ with Reeb vector field $\widehat{R}$ introduced in the 
previous section. 
An \emph{orbit set} (or \emph{multiorbit}) in $T_{\vf}$ is a formal finite product 
$\bs{\gamma} = \prod_i \gamma_i^{k_i}$, where $\gamma_i$ is a simple (i.e. embedded)
orbit of the flow of $\widehat{R}$ and $k_i \in \mathbb{N}$ is the \emph{multiplicity} of 
$\gamma_i$ in $\bs{\gamma}$, with $k_i \in \{0,1\}$ whenever $\gamma_i$ is hyperbolic. 

We will denote by $\mathcal{P}$ the set of 
 simple orbits $\gamma$ of $\widehat{R}$ in $T_\vf$ and 
$\mathcal{O}_d$  the set of orbit sets $\bs{\gamma} = \prod_i 
\gamma_i^{k_i}$ with $\gamma_i \in \mathcal{P}$ and 
$$\langle \bs{\gamma}, S_0 \rangle = \sum_i k_i \langle \gamma_i,S_0 \rangle =d.$$


Given $d \in \Z$, we define $PFC_d(T_{\vf})$ to be the vector space over $\mathbb{Z} / 2 \Z$ 
generated by the orbit sets in $\mathcal{O}_d$.
Observe that we are considering orbit sets for the perturbed Reeb vector field $\widehat{R}$, 
so that both orbits $e$ and $h$ can appear in the generators. Remark moreover that, since 
$\langle \delta,S_0 \rangle>0$ for any orbit in $\mathcal{P}$, we have $PFC_d(T_{\vf}) = 0$ 
whenever $d < 0$, and $PFC_0(T_{\vf}) \cong  \Z /2 \Z$ is generated by the empty orbit set.

Let $J$ be an alost complex structure as in the previous section. Given orbit sets 
$\bs{\gamma}_-$ and $\bs{\gamma}_+$ in $\mathcal{O}_d$, let 
${\mathcal M}(\bm{\gamma}_+, \bs{\gamma}_-, J)$ be the moduli space of degree $d$,
possibly disconnected, $J$-holomorphic multisections of $\R \times \widehat{T}_\vf$ which are
positively asymptotic to $\bs{\gamma}_+$ and negatively asymptotic to 
$\bs{\gamma}_-$. Here convergence to $\bs{\gamma}_+$ and $\bs{\gamma}_-$ has to be 
understood in the sense of ECH; see \cite[Definition 1.2]{Hu2}. We call ${\mathcal M}_1(\bs{\gamma}_+, 
\bs{\gamma}_-, J)$ the subset of multisections $u \in {\mathcal M}(\bs{\gamma}_+, 
\bs{\gamma}_-, J)$ with ECH index $I(u)=1$. (see \cite[Section 2]{HS} for details)

For any $d \ge 0$, the group $PFC_d(T_{\vf})$ can be endowed with the differential 
\begin{equation} \label{PFH differential} 
\partial (\bs{\gamma}_+) = \sum_{\bs{\gamma}_- \in \mathcal{O}_d} \#_2 
\overline{\mathcal M}_1(\bs{\gamma}_+, \bs{\gamma}_-, J) \bs{\gamma}_-,
\end{equation}
where $\overline{\mathcal M}_1(\bs{\gamma}_+, \bs{\gamma}_-, J)$ is the quotient of
${\mathcal M}_1(\bs{\gamma}_+, \bs{\gamma}_-, J)$ by the $\R$-action given by translations in the $\R$-direction. The resulting homology 
$${PFH}_d(T_{\vf}) \coloneqq  H_*({PFC}_d(T_{\vf}),\partial)$$
is the periodic Floer homology of $\vf$ (or $T_{\vf}$).

The following theorem is one of the main results of \cite{CGH3, CGH4} and an important step
toward the isomorphism between Heegaard Floer homology and embedded contact homology.
However we will not need it in the present article.
\begin{Thm} \label{Theorem: HF hat iso to PFH_2g} Let $(K,S,\vf)$ be an open book 
decomposition of genus $g\geq 1$ of $Y$. Then there is an isomorphism
\begin{equation*}
 \widehat{HF}(\overline{Y}) \cong {PFH}_{2g}(T_{\vf}),
\end{equation*}
\end{Thm}


\subsection{Periodic Floer homology for fibered knots} 
\label{Subsection: PFH for fibered knots}
In this section define the {\em sharp-version} of periodic Floer homology 
$PFH^\sharp(T_{\vf})$ associated to an open book decomposition. The construction is 
analogous to the construction of $ECH^\sharp$ from \cite[Section 7]{CGH2}.

\begin{Def} \label{Definition: Alexander degree in PFH} For any generator $\bs{\gamma} = 
e^k \prod_j\gamma_j^{k_j}$ of $PFC_d(T_{\vf})$, with $\gamma_j \neq e$ for all $j$,
we define $\underline{\bs{\gamma}}\coloneqq \prod_j\gamma_j^{k_j}$. The 
\emph{Alexander degree} of $\bs{\gamma}$ with respect to $S_0$ is the integer, or half 
integer,
 \[\mathcal{A}(\bs{\gamma}) \coloneqq  \langle \underline{\bs{\gamma}}, S_0\rangle 
- \frac d2 = \frac d2 - k.\]
\end{Def}
The most interesting case is when $d=2g$. If $\bs{\gamma} \in \mathcal{O}_{2g}$, 
 then $\mathcal{A}(\bs{\gamma}) \in \{-g,\ldots,g\}$
and $\mathcal{A}(\bs{\gamma}) = -g$ if and only if $\gamma= e^{2g}$. 
\begin{Lemma}\label{Alexander is a filtration}
if $\bs{\gamma}_+$ and $\bs{\gamma}_-$ are two orbit sets in $\mathcal{O}_d$ such that 
$\mathcal{M}(\bs{\gamma}_+, \bs{\gamma}_-, J) \neq \emptyset$, then 
$\mathcal{A}(\bs{\gamma}_+) \geq \mathcal{A}(\bs{\gamma}_-)$. 
\end{Lemma}
\begin{proof}
Let $u \colon (\dot F, j) \to \R \times \widehat{T}_\vf$ be a $J$-holomorphic curve multisection 
in $\mathcal{M}(\bs{\gamma}_+, \bs{\gamma}_-, J)$   and let $\dot F_0$ be a connected 
component of $F$ with a positive puncture asymptotic to $e$, or a multiple cover of $e$. 
If $u|_{\dot F_0}$ is not a cover of a trivial cylinder on $e$, then the projection to 
$\widehat{T}_\vf$ of the end of $u$ which is asymptotic to $e$ must approach $e$ with 
nonzero winding number with 
respect to the longitude induced by $\partial T_\vf$: this is a consequence of Lemma 5.3.2 of \cite{CGH1} and the fact that  $\partial T_\vf$ is a negative Morse-Bott torus. 
However, for topological reasons, this is not possible for a map 
with value in $\widehat{T}_{\vf}$ and thus $u|_{\dot F_0}$ is a cover of a trivial cylinder on 
$e$. This implies that the total multiplicity of $e$ in $\bs{\gamma}_-$ is at least equal to 
the total multiplicity in $\bs{\gamma}_+$
\end{proof}
The previous lemma shows that the Alexander degree induces a filtration on 
$PFC_{2g}(T_\vf)$ called the {\em Alexander filtration}.

Now we recall the sharp-version of periodic Floer homology, which is similar to the 
sharp-version of embeded contact homology defined in \cite[Section 7]{CGH1} and 
generalises the sharp-version of fixed point Floer homology defined in 
Section~\ref{Subsection: Review of symplectic Floer homology}. 
Given $d \ge 0$, we denote by $\mathcal{O}^\sharp_d$ the subset of $\mathcal{O}_d$
consisting of orbit sets wich do not contain $e$. Then we define $PFC_d^{\sharp}(T_{\vf})$ as 
the vector space generated over $\Z / 2 \Z$ by the orbit sets in ${\mathcal O}_d^\sharp$. 
By Lemma~\ref{Alexander is a filtration} we can identify $PFC_d^{\sharp}(T_{\vf})$ to the
quotent of the complex of $PFC_d(T_{\vf})$ by the subcomplex generated by the orbit sets 
containing $e$, and therefore it carries an induced differential 
$\partial^{\sharp}$. We denote 
$$PFH_d^\sharp(T_\vf)= H_* \left ( PFC_d(T_{\vf}), \partial^\sharp \right ).$$
From Lemma~\ref{Alexander is a filtration} it follows that the graded complex of 
$PFC_{2g}(T_\vf)$ associated to the Alexander filtration is isomorphic to
$$\bigoplus_{i=-g}^g PFC^\sharp_{i+g}(T_\vf).$$

Combining the isomorphism $PFH^\sharp_*(T_\vf) \cong ECH^\sharp_*(T_\vf)$ from 
\cite[Section 3.6]{CGH3}, the isomorphism between $PFH^\sharp_*(T_\vf)$ and the 
sutured contact homology of the knot complement from \cite[Theorem 10.3.2]{CGH1} and 
 \cite[Conjecture 1.5]{CGHH}, we obtain the following conjecture. 
\begin{Cnj}
If $K \subset Y$ is a fibered knot with monodromy $\vf$, then 
$$\widehat{HFK}(\overline{Y}, \overline{K}, i) \cong PFH_{i+g}^\sharp(T_\phi).$$
\end{Cnj}
Note that, dropping the requirement that $K$ is fibered and denoting by $N$ the 
complement of a tubular neighbourhood of $K$, one can still define
$ECH^\sharp_i(N)$ and in that case it is expected that $\widehat{HFK}(\overline{Y}, 
\overline{K}, i)$ is isomorphic to $ECH^\sharp_{i+g}(N)$.
\begin{Rmk} \label{Remark: PFH_1 = HF}
 By Hutchins and Sullivan \cite{HS}, there is an identification of chain complexes 
 \[(PFC_1^\sharp(T_{\vf}),\partial^\sharp) = (CF^\sharp(\vf),\partial^\sharp)\]
 inducing a canonical isomorphism between $PFH_1^\sharp(T_{\vf})$ and $HF^\sharp(\vf)$.
\end{Rmk}

\section{The chain map from HFK to PFK} \label{Section: The isomo from 
HF to ECH}

The aim of this section is to define chain maps
\begin{equation}
 \Phi_i^\sharp \colon \widehat{CFK}(S, \bm{a}, \vf(\bm{a}),z,i) \longrightarrow 
PFC^\sharp_{i+g}(T_{\vf}) 
\end{equation}
for any $i \in \{-g, \ldots, g \}$. In particular, for $i = -g+1$, by Remark \ref{Remark: PFH_1 
= HF}, we get a chain map 
\begin{equation*}
 \Phi^\sharp \colon  \widehat{HFK}(S, \bm{a}, \vf(\bm{a}),z,i) \longrightarrow 
CF^{\sharp}(\vf).
\end{equation*}
In the next section we will prove that $\Phi^\sharp$ induces an isomorphism in homology.
In order to define the maps $\Phi_i^\sharp$ we will first review the chain map
\begin{equation*}
 \Phi \colon \widehat{CF}(S, \bm{a}, \vf(\bm{a}))
\longrightarrow PFC_{2g}(T_{\vf})
\end{equation*}
defined in \cite{CGH3} by Colin, Ghiggini and Honda. 

We will then show that $\Phi$ respects the Alexander filtrations induced by $K$ on both 
sides and define the maps $\Phi_i^\sharp$ as the maps induced by $\Phi$ between 
the homogeneous summands of the graded complexes, taking into account the identification 
of $\bigoplus \limits_{i=-g}^g PFC^\sharp_{i+g}(T_{\vf})$ with the graded complex of 
$PFC_2g(T_\vf)$.

\subsection{Review of the chain map $\Phi$} 
\label{Subsection: Review of the chain map Phi}
In this section we recall the definition of the chain map
$$\Phi \colon \widehat{CF}(S, \mathbf{a}, \vf(\mathbf{a})) \to PFC_{2g}(T_\vf)$$
introduced in \cite{CGH3}.
We refer the reader to sections 5 and 6 of \cite{CGH3} for details. Moreover, only for this
section, we will regard the mapping torus of $(S,\vf)$ as
$$T_{\vf} = \frac{S \times [0,2]}{(x,2)\sim (\vf(x),0)},$$
and analogously for the mapping torus $\widehat{T}_\vf$ of $(\widehat{S}, \vf)$.
We denote $\pi \colon \R \times \widehat{T}_\vf \to S^1 \cong [0,2]/ 0 \sim 2$
the fibration  naturally extending to the 
$\mathbb{R}$-component the fibration $\widehat{T}_{\vf} \rightarrow S^1$ with fibre 
$\widehat{S}$ defined by $(x,t) \mapsto t$. 

Define now the subset $B^c_+ \coloneqq [2,\infty) \times (1,2)$ of $\mathbb{R} 
\times\frac{[0,2]}{0 \sim 2}$ with all the corners
smoothed and let $B_+ = (\mathbb{R} \times S^1) \setminus B_+^c$.
\begin{figure} 
  \begin{center}
   \includegraphics[scale = 1]{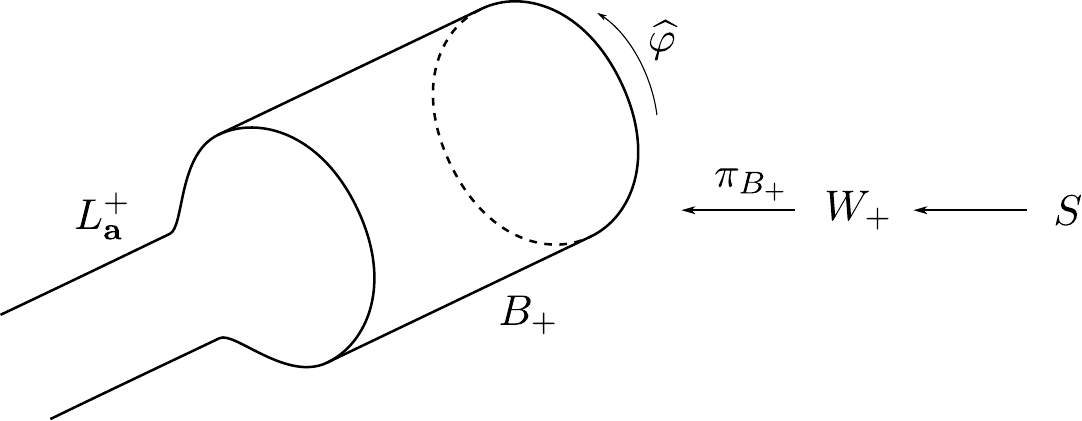}
  \end{center}
  \caption{The fibration $W_+$ over $B_+$ with fiber $\widehat{S}$ and Lagrangian 
boundary condition $L^+_{\bm{a}}$.}
  \label{Figure: B_+}
\end{figure}
Note that $B_+$ is biholomorphic to $\mathbb{D}^2 \setminus \{0, 1 \}$. Then we define 
the cobordism
\begin{equation*}
  W_+ = \pi^{-1}(B_+)
\end{equation*}
and a map $\pi_{B_+} \colon W_+ \to B_+$ by restricting $\pi$ to $W_+$. 
By construction $\pi_{B_+} \colon W_+ \to B_+$ is a fibration with fibre $\widehat{S}$. 
We will continue to indicate by $\pi_{\mathbb{R}}$ the restriction to $W_+$ of the projection 
$\R \times \widehat{T}_\vf \to \R$.

Consider the Morse-Bott perturbed stable Hamiltonian structure 
$(\widehat{\omega}_v,dt)$ on $\widehat{T}_{\vf}$ with Reeb vector field 
$\widehat{R}$ defined in Section \ref{Subsection: Review of symplectic 
Floer homology}. 
The $2$--form $\Omega = ds \wedge dt + \widehat{\omega}_v$ on $\R \times 
\widehat{T}_\vf$ is symplectic  and induces, by restriction, a symplectic form $\Omega_+$ 
on $W_+$. Endowed $W_+$ with this symplectic form, $\pi_{B_+} \colon W_+ \to B_+$  can 
be seen as a symplectic fibration.
The symplectic connection on $W_+$ given by the $\Omega_+$--orthogonal of the tangent 
space of the fibers is spanned by $\partial_s$ and $\widehat{R}$.

Fix a basis of arcs $\bm{a}=\{a_1,\ldots,a_{2g}\}$ in $S$ as in Section \ref{Subsubsection: 
The knots filtrations} and extend each $a_i$ to a segment $\widehat{a}_i$ in $\widehat{S}$
straight until it meets $\partial \widehat{S}$. If $f \in \mathcal{C}^{\infty}([1,4])$ is the function defined in last section, we assume that $|f(4)|$ is small enough to ensure that $\widehat{a}_i \cap \widehat{\vf}(\widehat{a}_j) \cap \{y \in (2,4]\} = \emptyset$ for any $i,j$. The proof of the following lemma is an 
immediate application of the implicit function theorem.
\begin{Lemma}\label{lemma: hat or not hat}
If the Morse-Bott perturbation is sufficiently small, then the intersection points between the 
arcs $\widehat{a}_i$ and $\wf(\widehat{a}_i)$ are in natural bijection with the intersection 
points between $a_i$ and $\vf(a_j)$ for all $i$ and $j$. Moreover this bijection induces an 
isomorphism between the chain complex $\widehat{CF}(S, \mathbf{a}, \vf(\mathbf{a}))$
and the chain complex $\widehat{CF}(\widehat{S}, \widehat{\mathbf{a}}, 
\wf(\widehat{\mathbf{a}}))$.
\end{Lemma}
Given a $k$-tuple of intersection points $\mathbf{x} = (x_1, \ldots, x_k)$ between $\bm{a}$ 
and $\vf(\bm{a})$, we denote $\widehat{\mathbf{x}}= (\widehat{x}_1, \ldots, 
\widehat{x}_k)$ the corresponding $k$-tuple of intersection points between 
$\widehat{\bm{a}}$ and $\wf(\widehat{\bm{a}})$.

Take a copy of $\widehat{\bm{a}}= \{ \widehat{a}_1,\ldots, \widehat{a}_{2g}\}$ in 
$\pi_{B_+}^{-1} (3,1)$ and call $L_{\bm{a}}^+$ the trace of the parallel
transport of $\widehat{\bm{a}}$ along $\partial B_+$ using the symplectic connection; then 
$L_{\bm{a}}^+$ is Lagrangian and
\begin{equation*}
 \begin{array}{lll}
  L^+_{\bm{a}} \cap \{s\geq 3,\ t=0\} & = & \{s\geq 3\} \times \{t=0\} \times 
\wf(\widehat{\bm{a}}); \\
  L^+_{\bm{a}} \cap \{s\geq 3,\ t=1\} & = & \{s\geq 3\} \times \{t=1\} \times 
\widehat{\bm{a}}. 
 \end{array}
\end{equation*}
Note that $L_{\bm{a}}^+$ has $2g$ connected components $L_{a_i}^+$, one for each 
component $\widehat{a}_i$ of $\widehat{\bm{a}}$.


The map $\Phi$ is given as a chain map 
\[ \Phi \colon \widehat{CF}(S, \bm{a}, \vf(\bm{a})) 
\longrightarrow PFC_{2g}(T_{\vf}) \]
defined by counting certain \emph{multisections} of the symplectic fibration 
$(W_+,\pi_{B_+})$ with Lagrangian boundary condition  $L^+_{\bm{a}}$. 

Let $J_+$ be a suitable almost complex structure on $W_+$ which is compatible with 
$\Omega_+$ and cylindrical for $s>3$ and $s<1$. (See \cite[section 5]{CGH3}, and in 
particular Remark 5.3.10 for the Morse-Bott perturbation.) Let $(F,j)$ be a compact (possibly
disconnected) Riemann surface with two sets of punctures $\mathbf{p}= \{p_1,\ldots,p_l\}$ 
in the interior and $\mathbf{q}= \{q_1,\ldots,q_k\}$ in the boundary of $F$ such that 
\begin{itemize}
\item[(i)] every connected component of $F$ contains at least an element of $\mathbf{p}$ 
and a connected component of $\partial F$, and 
\item[(ii)] every connected component of $\partial F$ contains at least an element of 
$\mathbf{q}$. 
\end{itemize}
We will set $\dot{F} = F \setminus \{\mathbf{p} \cup \mathbf{q}\}$.

\begin{Def} \label{Definition: Multisections of W_+}
Let $\widehat{\mathbf{x}} = \{\widehat{x}_1,\ldots, \widehat{x}_k\}$ be a $k$-tuple 
($k\leq 2g$) of $\widehat{\bm{a}} \cap \wf(\widehat{\bm{a}})$ and $\bs{\gamma} = 
\prod_j \gamma^{m_j} \in \mathcal{O}_k$. A degree $k$ multisection of $(W_+,J_+)$ from 
$\widehat{\mathbf{x}}$ to 
$\bs{\gamma}$ is  a holomorphic map 
 $$u\colon(\dot{F},j) \longrightarrow (W_+,J_+)$$
 where $(F,j)$ is a Riemann surface as above, and $u$ is such that
 \begin{enumerate}
  \item $u(\partial \dot{F}) \subset L_{\bm{a}}^+$ and maps each connected component 
of $\partial \dot{F}$ to a different $L_{a_i}^+$;
  \item $\lim_{w\rightarrow q_i} \pi_{\mathbb{R}}\circ u(w) = +\infty$ and $\lim_{w\rightarrow p_i} \
pi_{\mathbb{R}}\circ u(w) = -\infty$,
  \item near $q_i$, $u$ is asymptotic to a strip over $[0,1] \times \{\widehat{x}_i\}$;
  \item near each $p_i$, $u$ is asymptotic to a cylinder over a multiple of some $\gamma_j$ 
so that the total multiplicity of $\gamma_j$ over all the   $p_i$ is $m_i$.
 \end{enumerate}
 \end{Def}

In practice, holomorphic multisections in $W_+$ interpolate between multisections in $\R 
\times [0,1] \times \widehat{S}$ and $\R \times \widehat{T}_\vf$. Moreover, in 
\cite[Section 5]{CGH3} the authors define an \emph{$ECH$ index for holomorphic 
multisections of $W_+$} which interpolates between the ECH-type index for holomorphic 
curves in $\R \times [0,1] \times \widehat{S}$ and the $ECH$ index for holomorphic curves 
in $\R \times \widehat{T}_\phi$. As in \cite{CGH3}, we will refer to the image of the 
connected components of $\dot F$ as to the \emph{irreducible components of $u$}.
We say that $u$ is {\em irreducible} if $\dot F$ is connected.

For $\mathbf{x}=(x_1, \ldots, x_{2g})$ intersection points between $\bm{a}$ and
$\vf(\bm{a})$, and $\bs{\gamma} \in \mathcal{O}_{2g}$, we denote by
$$\mathcal{M}_0(\widehat{\mathbf{x}}, \bs{\gamma}, J_+)$$
the space of ECH index $0$, degree $2g$ multisections $(W_+,J_+)$ from 
$\widehat{\mathbf{x}}$ to $\bs{\gamma}$. 

Recall the chain complex $CF'(S, \bm{a}, \vf(\bm{a}))$ defined in Section 
\ref{Subsection: Heegaard Floer homology for open books} as a precursor 
of  $\widehat{CF}(S, \bm{a},\vf(\bm{a}))$.
We define a map $\Phi' \colon CF'(S, \bm{a},\vf(\bm{a})) \to PFC_{2g}(T_{\vf})$ as 
\begin{equation} \label{eqn: definition of Phi}
\Phi' (\mathbf{x})= \sum_{\bs{\gamma} \in \mathcal{O}} \#_2 
\mathcal{M}_0(\widehat{\mathbf{x}}, \bs{\gamma}, J_+) \bs{\gamma}.
\end{equation}



In the following theorem we summarize some of the results about $\Phi$ proved in 
\cite{CGH3}. 

\begin{Thm} \label{Theorem: Summary of results about the properties of Phi}
The following hold:
\begin{enumerate}
 \item $\Phi'$ is a chain map, and 
 \item $\Phi'$ maps the subspace $\mathcal{R}$ defined in Section \ref{Subsection: 
Heegaard Floer homology for open books} to zero.
\end{enumerate}
\label{Theorem: main results on Phi} 
\end{Thm}
The first statement is \cite[Proposition 6.2.2]{CGH3} together with Lemma \ref{lemma: 
hat or not hat}. The second statement is \cite[Proposition 6.2.2]{CGH3}. 
\begin{Cor}
The map $\Phi'$ induces a well defined chain map
$$\Phi \colon \widehat{CF}(S, \bm{a},\vf(\bm{a})) \to PFC_{2g}(T_{\vf}).$$
\end{Cor}
By \cite[Theorem 1.0.1]{CGH4}, the map $\Phi$ induces an isomorphism in homology. We 
will not need this result.


\subsection{$\Phi$ respects the filtrations} \label{Subsection: Phi is filtered}
The aim of this section is prove the following theorem.
\begin{Thm} \label{Theorem: Phi is filtered}
 $\Phi$ is filtered with respect to the Alexander filtrations induced by $\overline{K}$ on  
$\widehat{CF}(S, \bm{a},\vf(\bm{a}))$ and by $K$ on 
$PFC_{2g}(T_{\vf})$.
\end{Thm}
Before proving the theorem we introduce some notation. Let ${\mathcal D} \subset \widehat{S}$ be the disc which is invariant under the monodromy $\widehat{\vf}$ and passes through the base point $z$. Then ${\mathcal D}$  contains the fixes point $p_e$ and,
if $\varepsilon$ in the definition of $\vf$ is small enough, $\widehat{\bm{a}} \cap \widehat{\varphi}(\widehat{\bm{a}}) \cap {\mathcal D}=  \widehat{\bm{a}} \cap \widehat{\varphi}(\widehat{\bm{a}}) \cap \{ y \in [1,3] \}$; i.e.\ ${\mathcal D}$ contains all the intersection points in Figure \ref{fig:trivial region}. Let ${\mathcal C}= \partial \mathcal D$ with the induced boundary orientation. 
Since ${\mathcal C}$ is invariant under the monodromy, it gives rise to a torus ${\mathcal T}
\subset \widehat{T}_{\vf}$ which is foliated by trajectories of the flow of $\widehat{R}$. Fianlly we denote $\Sigma = (\R \times {\mathcal T}) \cap W_+ \cong B_+ \times {\mathcal C}$. If we project $L_{\bm{a}} \cap \Sigma$ to ${\mathcal T}$ we obtain $2g$ segments
$\bm{\sigma}=\{ \sigma_1, \ldots, \sigma_{2g} \}$ which are tangent to $\widehat{R}$.

\begin{Lemma} \label{Lemma: Filtration of Phi looking to the intersection with G_z} 
 Let $u\colon (\dot{F},j)\rightarrow (W_+, J_+)$ be a multisection of $(W_+, J_+)$ from 
$\widehat{\mathbf{x}}$ to $\bs{\gamma}$. Then
 \[\mathcal{A}(\mathbf{x}) - \mathcal{A}(\bs{\gamma}) = \langle u(\dot{F}),\R \times 
\{t_0\} \times {\mathcal C} \rangle\]
 for any $t_0 \in [0,1]$.
\end{Lemma}
\begin{proof}
We recall that all intersection points from Figure \ref{fig:trivial region} are contained in 
${\mathcal D}$ and therefore, by the reinterpretation of the Alexander grading given in Equation \eqref{Equation: Alexander degree in terms of components on S_0}, we have
$${\mathcal A} (\mathbf{x})= -g + \langle [0,1] \times \widehat{\mathbf{x}},  \{ t_0 \} \times (\widehat{S} \setminus {\mathcal D}) \rangle$$
for any $t_0 \in [0,1]$. Also, since $e$ is the only periodic orbit of $\widehat{R}$ with period at most $2g$ intersecting $\{t_0 \} \times {\mathcal D}$, we have
$${\mathcal A}(\bm{\gamma})= -g+\langle \bm{\gamma},  \{ t_0 \} \times (\widehat{S} \setminus {\mathcal D}) \rangle.$$

Let $\check{u} \colon \check{F} \to \check{W}_+$ be the continuous extension of the holomorphic multisection $u \colon \dot F \to W_+$ where $\check F$ is obtained by performing a real blow up of $\dot F$ at the punctures and $\check{W}_+$ is obtained, roughly speaking, by adding $\{ - \infty \} \times \widehat{T}_\vf$ and 
$\{ + \infty \} \times [0,1] \times \widehat{S}$ to $W_+$ (see \cite[Section 5.4.2]{CGH3}). 

Consider the two surfaces
\[M_{+\infty} = \{+\infty\} \times \{t_0\} \times (\widehat{S} \setminus {\mathcal D}) \ \mbox{ and }\ M_{-\infty} = \{-\infty\} \times \{t_0\} \times (\widehat{S} \setminus {\mathcal D}),\]
and define the closed surface
\begin{align*} 
  M  \coloneqq M_{-\infty} \cup \big([-\infty,+\infty] \times \{t_0\} \times ({\mathcal C} \sqcup \partial \widehat{S}) \big) \cup \overline{M}_{+\infty}.
\end{align*}
Since $0 = [M] \in H_2(\check{W}_+;\Z)$, we have:
\begin{equation*}
\begin{split}
  0 & =  \langle \check{u}(\check{F}) , M \rangle  = \\
    & =  - \langle \check{u}(\check{F}) , M_{+\infty} \rangle + \langle \check{u}(\check{F}),  
\R \times \{t_0\} \times ({\mathcal C} \sqcup \partial \widehat{S}) \rangle + \langle 
\check{u}(\dot{F}), M_{-\infty} \rangle = \\
    & =  -(\mathcal{A}(\mathbf{x}) + g) + \langle u(\dot{F}), \R \times \{t_0\} \times 
{\mathcal C} \rangle + (\mathcal{A}(\bs{\gamma}) + g), 
\end{split}
\end{equation*}
where we used the fact that $u(\dot{F})$ is disjoint from $B_+ \times \partial \widehat{S} 
\subset \partial W_+$ and has no ends intersecting ${\mathcal C}$. 
\end{proof}
\begin{proof}[Proof of Theorem \ref{Theorem: Phi is filtered}.]
Let $u \colon \dot F \to W_+$ be a multisection of $(W_+, J_+)$ from $\mathbf{x}$ to 
$\bm{\gamma}$. By Lemma \ref{Lemma: Filtration of Phi looking to the intersection with 
G_z} it remains to prove that
\begin{equation} \label{Equation: Intersection between u and lambda is positive}
 \langle {u}(\dot{F}), \R \times \{t_0\} \times {\mathcal C} \rangle \geq 0.
\end{equation}
We can assume without loss of generality that the intersection of the projection of $u(\dot F)$
to $\widehat{T}_\vf$ with ${\mathcal T}$ consists of a finite collections of closed curves and 
arcs with endpoints in $\bm{\sigma}$ because, by construction, a neighbourhood of 
${\mathcal T}$ in $\widehat{T}_\vf$ is foliated by invariant tori for the flow of 
$\widehat{R}$ and replacing ${\mathcal T}$ with a nearby one does not change the 
arguments.

Let $\mathbf{c}_u$ be the the intersection of the projection of $u(\dot F)$ to 
$\widehat{T}_\vf$ with ${\mathcal T}$ closed by adding segments in $\bm{\sigma}$. We orient $\mathbf{c}_u$ by requiring that, at the intersection points between the projection of $u(\dot F)$ with ${\mathcal T}$, the coorientation of ${\mathcal T}$ followed by the orientation of $\mathbf{c}_u$ gives the orientation of the projection of $u(\dot F)$. Then 
one can verify that 
$$\langle {u}(\dot{F}), \R \times \{t_0\} \times {\mathcal C} \rangle_{W_+}=
\langle \mathbf{c}_u, \{t_0\} \times {\mathcal C} \rangle_{\mathcal T}.$$

The homology group $H_1({\mathcal T}; \Z)$ is freely generated by classes $\lambda$, 
represented by $\{t_0 \} \times {\mathcal C}$, and $\mu$, represeted by a curve isotopic to 
$e$ and intersecting each $\widehat{S}_t$ in $\widehat{S}_t \setminus S_t$, such that 
$\langle \mu, \lambda \rangle =1$. Then 
$[\mathbf{c}_u]= k \mu$ because $\mathbf{c}_u$ is disjoint from $\widehat{S}_t \setminus 
S_t$, and therefore $\langle \mathbf{c}_u, \mu \rangle =0$.

We form a curve $\mathfrak{t} \subset {\mathcal T}$ by connecting the endpoints of a 
sufficiently long trajectory of the flow of $\widehat{R}$ with a short segment disjoint from 
$\mathbf{c}_u$. Then, by the positivity of intersection between the projection of $u(\dot F)$
and $\widehat{R}$, we have $\langle \mathbf{c}_u,   \mathbf{t} \rangle \ge 0$. Since
$\mathbf{t}= l \lambda + m \mu$ with $l,m >0$, we obtain $\langle \mathbf{c}_u,
 \mathbf{t} \rangle = kl$, from which we deduce that $k \ge 0$.
\end{proof}
We observe that the chain map $\Phi^\sharp$ can be definied explicitly as
$$\Phi^\sharp (\mathbf{x}) = \sum \limits_{\gamma \in \mathcal{P} \setminus  \{e \}}
\#_2 {\mathcal M}_0(\widehat{\mathbf{x}}, \gamma e^{2g-1}, J_+) \gamma$$
for all $\mathbf{x}$ with $\mathcal{A}(\mathbf{x})=1-g$.

\begin{Lemma}
  Let $\mathbf{a}$ and $\mathbf{b}$ be bases of arcs for $S$. Then there is a commutative
  diagram
\begin{equation*}
 \begin{tikzpicture}[->,auto,node distance=1.5cm,>=latex',baseline=(current  bounding  box.center)]
 \matrix (m) [matrix of nodes, row sep=2cm,column sep=2cm]
{$\widehat{HFK}(S, \bm{a}, \varphi(\bm{a});-g+1)$ & $HF^\sharp(\varphi)$ \\
$\widehat{HFK}(S, \bm{b}, \varphi(\bm{b});-g+1)$ & $HF^\sharp(\varphi)$ \\ };
  \path[->]
  (m-1-1) edge node[above] {$\Phi^\sharp_*$} (m-1-2)
      edge node[left] { $\mathfrak{s}_*$} (m-2-1);
 \path[->]
  (m-2-1) edge node[below] {$\Phi^\sharp_*$} (m-2-2);
 \path[->]
  (m-1-2) edge node[right] {$=$} (m-2-2);
 \end{tikzpicture}
\end{equation*}
where $\mathfrak{s}_*$ is the isomorphism of knot Floer homology for changing of basis.
\end{Lemma}
\begin{proof}
  Since two bases of arcs are always related by a sequence of arc slides, we assume without loss of generality that $\bm{b}$ is obtained from $\bm{a}$ by sliding $a_1$ over $a_2$. This means that $\bm{b}$ is as in the picture.
 \begin{figure} [ht!] 
  \begin{center}
   \includegraphics[scale = 0.4]{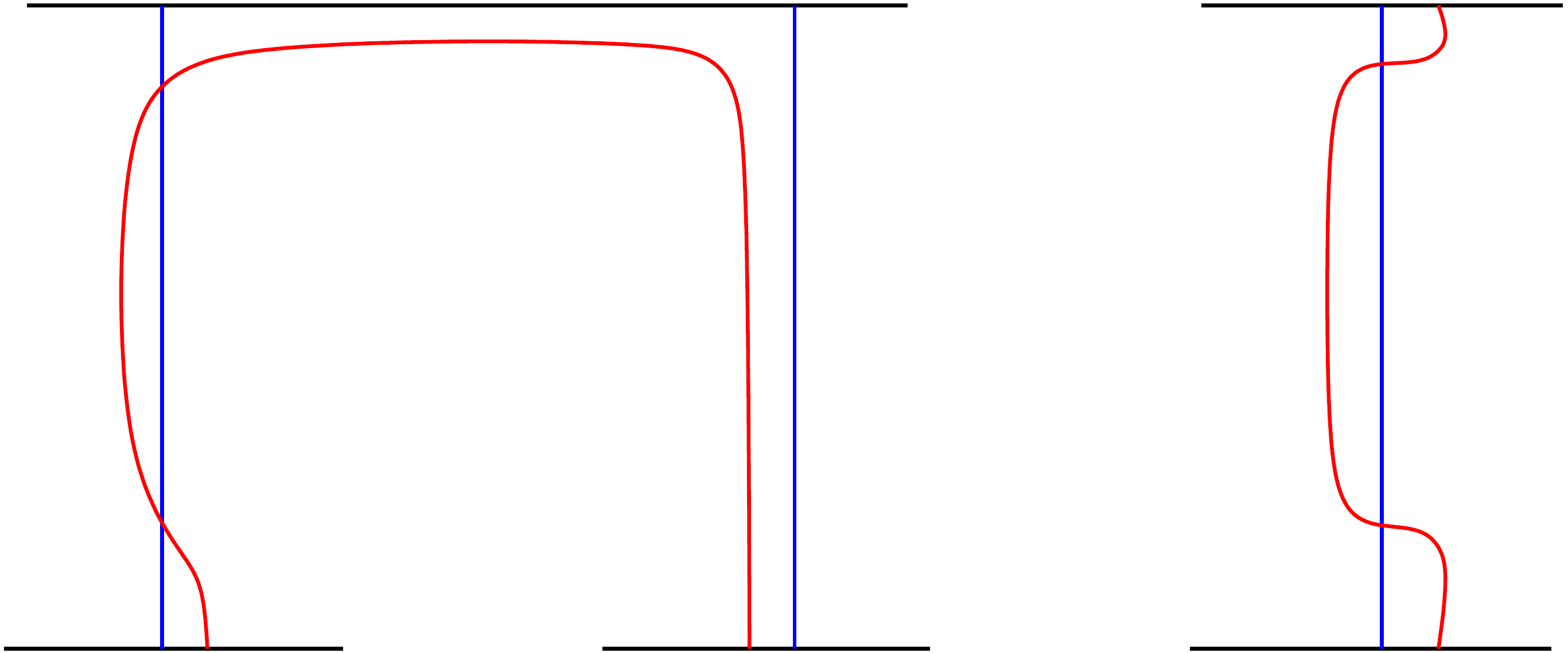}
  \end{center}
  \caption{The basis $\bm{a}$ is in blue and the basis $\bm{b}$ is in red. On the left: $b_1$ is obtained by sliding $a_1$ over $a_2$; on the right: $b_i$, for $i>1$, is obtained as a small deformation of $a_i$. We call $\theta_i^+$ the intersection point on the top and $\theta_i^-$ the intersection point on the bottom.}
 \end{figure}
It is important that we choose $\theta^\pm_i$ close enough to the boundary so that every intersection point on $a_i$, $b_i$, $\varphi(a_i)$ and $\varphi(b_i)$ is between $\theta_i^+$ and $\theta_i^-$ and not between $\theta_i^\pm$ and the boundary.
 
  We define $B_{\mathfrak{s}}= D^2 \setminus \{\pm 1, \pm i \}$ and $E_{\mathfrak{s}}=B_{\mathfrak{s}} \times S$. Let $\ell_1, \ldots, \ell_4$ be the connected components of $\partial B_{\mathfrak{s}}$ numbered in counterclockwise order starting from the arc between $1$ and $i$. On $\partial E_{\mathfrak{s}}$ we consider the Lagrangian submanifold $L= \ell_1 \times \bm{a} \sqcup \ell_2 \times \bm{\mathbf{b}} \sqcup \ell_3 \times \phi(\bm{b}) \sqcup \ell_4 \times \phi(\bm{a})$. 
  The map $\mathfrak{s} \colon \widehat{CFK}(S, \bm{a}, \varphi(\bm{a});-g+1) \to \widehat{CFK}(S, \bm{b}, \varphi(\bm{b});-g+1)$ is defined by counting index zero, degree 2g embedded multisections of $E_{\mathfrak{s}}$ with boundary on $L$ which are asymptotic to $\Theta^+=(\theta_1^+, \ldots, \theta_{2g}^+)$ at $i$ and to $\varphi(\Theta^-)=(\varphi(\theta_1^-), \ldots, \varphi(\theta_{2g}^-))$ at $-i$.

  This map induces an isomorphism in homology because for every $2g$-tuple of intersection points $\mathbf{y}_a$ between $\bm{a}$ and $\varphi(\bm{a})$ there is a unique closest $2g$-tuple of intersection points $\mathbf{y}_b$ and a unique small area, index zero multisection of $E_\mathfrak{s}$ which is asymptotic to $\mathbf{y}_a$, $\Theta^+$, $\mathbf{y}_b$ and $\varphi(\Theta^-)$. On the fibre this multisection projects to a $2g$-tuple of small fish-tail shaped quadrilaterals.

  Next, we define $B_{\mathfrak{h}, z}= D^2 \setminus \{0, 1, z, \bar{z} \}$ for $z \in \partial D^2$ with $\Im(z)>0$ and $E_{\mathfrak{h}, z}$ as the total space of a fibration over $B_{\mathfrak{h}, z}$ with fibre $S$ and monodromy $\varphi$.  Let $\ell_1, \ell_2, \ell_3$ the connected component of $\partial B_{\mathfrak{h}, z}$ numbered counterclockwise starting from $1$ and define the Lagrangian submanifold $L= L_1 \sqcup L_2 \sqcup L_3$ where $L_1 = \ell_1 \times \bm{a}$, $L_3= \ell_3 \times \varphi(\bm{a})$ and $L_2$ is the parallel transport of $\bm{b}$ over $\ell_3$.

  We define the chain homotopy $\mathfrak{h} \colon \widehat{CFK}(S, \bm{a}, \varphi(\bm{a});-g+1) \to CF^\sharp(\varphi)$ by counting index $-1$ embedded degree $2g$ multisections in the family of fibrations $E_{\mathfrak{h}, z} \to B_{\mathfrak{h}, z}$, for $z \in \partial D^2 \cap \{ \Im(z)>0 \}$, which have boundary on $L$ and are asymptotic to $\Theta^+$ at $z$ and $\varphi(\Theta^-)$ at $\bar{z}$.

  The moduli space of index zero multisections of $E_{\mathfrak{h}, z} \to B_{\mathfrak{h},z}$ (with varying $z$) is a one dimensional moduli space which has boundary degenerations which correspond to $\partial \circ \mathfrak{h} + \mathfrak{h} \circ \partial$, degenerations for $z \to 1$ which correspond to $\Phi^\sharp \circ \mathfrak{s}$ and degenerations for $z \to -1$ which, we claim, correspond to $\Phi^\sharp$.

  As $z \to -1$, the surface $B_{\mathfrak{h},z}$ degenerates towards a nodal surface $B_{\mathfrak{h},-1}= B_{\mathfrak{h},-1}' \cup B_{\mathfrak{h},-1}''$ where $B_{\mathfrak{h},-1}' = D^2\setminus \{ 0,1 \}$ and $B_{\mathfrak{h},-1}'' = D^2\setminus
  \{\pm i \}$. The two sides of the node are $-1 \in B_{\mathfrak{h},-1}'$ and $1 \in  B_{\mathfrak{h},-1}''$. The fibration also splits as $E_{\mathfrak{h},-1}= E_{\mathfrak{h},-1}' \cup E_{\mathfrak{h},-1}''$, where $B_{\mathfrak{h},-1}'$ is the total space of a fibration with fibre $S$ and monodromy $\varphi$ over $B_{\mathfrak{h},-1}'$ and $E_{\mathfrak{h},-1}''= B_{\mathfrak{h},-1}'' \times S$. On $\partial E_{\mathfrak{h},-1}'$ there is the Lagrangian submanifold $L'$ obtained by parallel transport of $\bm{a}$ over $\partial B_{\mathfrak{h},-1}'$. We denote $\ell_+= \partial D^2 \cap \{ \Re >0 \}$ and  $\ell_-= \partial D^2 \cap \{ \Re <0 \}$. On $\partial E_{\mathfrak{h},-1}''$ there is the Lagrangian submanifold $L''= \ell_+ \times \bm{a} \sqcup \ell_- \times \bm{b}$. The count of multisections of $E_{\mathfrak{h},-1}'$ with boundary on $L'$ gives $\Phi^\sharp \colon \widehat{CFK}(\varphi(\bm{a}), \bm{a}; 1-g) \to CF^\sharp(\varphi)$. Each multisection counted in $CF^\sharp(\varphi)$ intersect the the fibre over $-1$ into a $2g$-tuple $\mathfrak{z}$ of points in $L$ which are not too close to the boundary of the fibre. On the other hand, it is easy to check on the diagram $(S, \bm{a}, \bm{b})$ that for any $2g$-tuple of points $\mathfrak{z}$ on $\bm{a}$ which are not too close to the boundary there is a unique degree $2g$ multisection of   $E_{\mathfrak{h},-1}''$ with boundary in $L''$ and asymptotic to $\Theta^+$ at $i$ and $\Theta^-$ at $-i$. This proves the last claim.
  \end{proof}
\section{Proof of the isomorphism} \label{Section: Proof of the equivalence}
To prove Theorem \ref{Theorem: Main theorem in introduction} we will show 
that the chain map
\begin{equation} \label{Equation: the chain map from CFK to HF}
 \Phi^\sharp \colon \widehat{CFK}(S, \bm{a}, \vf(\bm{a});
-g+1) \longrightarrow CF^{\sharp}(\vf)
\end{equation}
given by Theorem \ref{Theorem: Phi is filtered} is an isomorphism. 

The section is organized as follows. We will first  recall some algebraic tools about exact 
sequences that will be used later. We will then prove the existence of the exact triangles in 
fixed point Floer homology (already known by Seidel, cf. \cite[Theorem 4.2]{Se4}) and in 
knot Floer homology (cf. \cite[Theorem~8.2]{OS3}). We will then prove the 
commutativity of the diagram \eqref{Diagram: Main diagram between HF and HF without 
maps} and, finally, finish the proof of the fact that $\Phi^\sharp_*$ is an isomorphism by 
proving the initial step of the induction.

\subsection{Some algebra}
In this subsection we recall some results that will be useful to show the existence of the two exact sequences \eqref{Triangle: L phi phi-tau without maps} and \eqref{Triangle: Exact triangle in HF without maps} and the commutativity of the diagram \eqref{Diagram: Main diagram between HF and HF without maps}. 

 \begin{Lemma} \label{Lemma: Algebraic lemma for the exactness}
  Let $(A, \partial_A)$, $(B, \partial_B)$ and $(C, \partial_C)$, be chain complexes, $f\colon A 
\rightarrow B$ and $g \colon B \rightarrow C$ chain maps and $H:A \rightarrow C$ a chain 
homotopy from $g \circ f$ to $0$. If
  \begin{equation} \label{Equation: Homology of the general cone is 0}
    H_*\left(A\oplus B \oplus C,\left(\begin{array}{ccc} 
   \partial_A     &     0     & 0\\
    f &     \partial_B     & 0\\
    H & g & \partial_C
  \end{array}\right)\right) = 0
  \end{equation}
 then there exists a linear map $d : H_*(C) \rightarrow H_*(A)$ such the triangle 
 \begin{equation*} 
 \begin{tikzpicture}[->,auto,node distance=1.5cm,>=latex',baseline=(current  bounding  box.center)]

  \node (1) {$H_*(A)$};
  \node (2) [right = of 1] {$H_*(B)$};
  \node (3) [right = of 2] {$H_*(C)$};

  \path[every node/.style={font=\sffamily\small}]
    (1) edge node[above] {$f_*$} (2)
    (2) edge node[above] {$g_*$} (3)
    (3) edge[bend left=30] node {$d$} (1);
 \end{tikzpicture}
 \end{equation*}
 is exact. 
 \end{Lemma}
 \begin{proof} The homology of the cone
  \[Cone(f) \coloneqq \left(A \oplus B,\left(\begin{array}{cc}
                                                   \partial   & 0\\
                                                    f    &     \partial  
                                                  \end{array}\right)\right)\]
  fits in a long exact sequence
 \begin{equation} \label{Triangle: Exact triangle with the general cone}
 \begin{tikzpicture}[->,auto,node distance=1.5cm,>=latex',baseline=(current  bounding  box.center)]

  \node (1) {$H_*(A)$};
  \node (2) [right = of 1] {$H_*(B)$};
  \node (3) [right = of 2] {$H_*(Cone(f))$};

  \path[every node/.style={font=\sffamily\small}]
    (1) edge node[above] {$f_*$} (2)
    (2) edge node[above] {$(i_2)_*$} (3)
    (3) edge[bend left=20] node {$(p_1)_*$} (1);
 \end{tikzpicture}
 \end{equation}
 where $i_2$ and $p_1$ are the inclusion and, respectively, the projection of the corresponding summand. Moreover condition \eqref{Equation: Homology of the general cone is 0} implies that
 \begin{equation}\label{Equation: Cone of iota cong to CF of phi circ tau}
Cone(f) \xrightarrow{(H,g)} C
\end{equation}
induces an isomorphism in homology. Putting everything together, we obtain a commutative 
diagram
\begin{equation}
\begin{tikzpicture}[->,auto,node distance=1.5cm,>=latex',baseline=(current  bounding  box.center)]
\matrix (m) [matrix of nodes, row sep=1cm,column sep=1.1cm]
{ $ $ &      $ $       & $H_*(Cone(f))$ &  $ $ & $ $ \\
  \phantom{'}$\ldots $ & $H_*(B)$ & $H_*(C)$ & $H_*(A)$ & $ \ldots $ \phantom{'}\\ };
 \path[->]
  (m-1-3) edge node[anchor=center,rotate=-90,yshift=1.8ex]{$\cong$} (m-2-3)
          edge[bend left=20] node[above] {$\quad (p_1)_*$} (m-2-4);
 \path[->]
  (m-2-1)[xshift=2ex] edge  node[below] {$f_*$} (m-2-2);
 \path[->]
  (m-2-2) edge node[below] {$g_*$} (m-2-3)
          edge[bend left=20] node[above] {$(i_2)_*\quad $} (m-1-3);
 \path[->]
  (m-2-3) edge node[below] {$d$} (m-2-4);
  \path[->]
  (m-2-4) edge node[below] {$f_*$} (m-2-5);
\end{tikzpicture}
\end{equation}
and the exactness of \eqref{Triangle: Exact triangle with the general cone} implies that also the bottom line is exact.
\end{proof}




\begin{Lemma}\label{Lemma: Algebraic lemma for the third square}
Let $A$, $B$, $C$, $A'$, $B'$ and $C'$ be chain complexes fitting in the diagram
\begin{equation} \label{Diagram: Big diagram in the algebraic lemma}
\begin{tikzpicture}[->,auto,node distance=1.5cm,>=latex',baseline=(current  bounding  box.center)]
\matrix (m) [matrix of nodes, row sep=2cm,column sep=4cm]
{  $A'$ & $B'$ & $C'$ \\
    $A$ & $B$ &  $C$ \\};
 \path[->]
  (m-1-1) edge node[above] {$f'$} (m-1-2)
          edge node[right] {$a$} (m-2-1)
          edge node[below] {$R$} (m-2-2)
          edge[bend left=20] node[above] {$H'$} (m-1-3)
          edge[dashed] node[above,yshift=2ex,xshift=-7ex] {$T$} (m-2-3);
 \path[->]
  (m-1-2) edge node[right,yshift=3.1ex] {$b$} (m-2-2)
          edge node[above] {$g'$} (m-1-3)
          edge node[above] {$S$} (m-2-3);
 \path[->]
  (m-1-3) edge node[right] {$c$} (m-2-3);
 \path[->]
  (m-2-1) edge node[below] {$f$} (m-2-2)
  edge[bend right=20] node[below] {$H$} (m-2-3);
 \path[->]
  (m-2-2) edge node[below] {$g$} (m-2-3);
 \end{tikzpicture}
\end{equation}
where:
\begin{enumerate}[leftmargin=*]
 \item $H$ and $H'$ are chain homotopies from $g \circ f$ and, respectively, $g' \circ f'$ to $0$; 
 \item $R$ is a chain homotopy from $f \circ a$ to $b \circ f'$;
 \item $S$ is a chain homotopy from $g \circ b$ to $c \circ g'$;
 \item $T$ is a map such that $(c\circ H' +S \circ f') + (H\circ a + g \circ R) = T \circ \partial' + 
\partial \circ T$;
 \item the homologies of the two iterated cones
 \[\left(A\oplus B \oplus C,\left(\begin{array}{ccc} 
   \partial     &     0     & 0\\
    f &     \partial     & 0\\
    H & g & \partial
  \end{array}\right)\right)\quad \mbox{and}\quad 
  \left(A'\oplus B' \oplus C', \left(\begin{array}{ccc} 
   \partial'     &     0     & 0\\
    f' &     \partial'     & 0\\
    H' & g' & \partial'
  \end{array}\right)\right)\]
 are trivial. 
\end{enumerate}
Then there exist linear maps $d$ and $d'$ such that the following diagram 
\begin{equation}\label{commutativity}
\begin{tikzpicture}[->,auto,node distance=1.5cm ,>=latex',baseline=(current  bounding  box.center)]
\matrix (m) [matrix of nodes, row sep=1cm,column sep=1.5cm]
{ \phantom{f}$\ldots$ & $H_*(A')$ & $H_*(B')$ &  $H_*(C')$  & $\ldots$\phantom{f}\\
  \phantom{f}$\ldots$ & $H_*(A)$ & $H_*(B)$ & $H_*(C)$ & $\ldots$\phantom{f}\\ };
 \path[->]
  (m-1-1) edge node[above] {$d'$}   (m-1-2);
 \path[->]
  (m-1-2) edge node[above] {$f'_*$} (m-1-3)
          edge node[right] {$a_*$} (m-2-2);
 \path[->]
  (m-1-3) edge node[right] {$b_*$} (m-2-3)
          edge node[above] {$g'_*$} (m-1-4);
 \path[->]
  (m-1-4) edge node[right] {$c_*$} (m-2-4)
          edge node[above] {$d'$}   (m-1-5);
 \path[->]
  (m-2-1) edge node[below] {$d$} node[above,yshift=2.75ex] {\huge$\circlearrowleft$}  (m-2-2);
 \path[->]
  (m-2-2) edge node[below] {$f_*$} node[above,yshift=2.75ex] {\huge$\circlearrowleft$} (m-2-3);
 \path[->]
  (m-2-3) edge node[below] {$g_*$} node[above,yshift=2.75ex] {\huge$\circlearrowleft$} (m-2-4);
 \path[->]
  (m-2-4) edge node[below] {$d$} node[above,yshift=2.75ex] {\huge$\circlearrowleft$}  (m-2-5);
 \end{tikzpicture}
\end{equation}
commutes and the two rows are long exact sequences.
\end{Lemma}
\begin{proof}
 Observe first that the commutativity of the two squares in Diagram \eqref{commutativity} comes from the assumptions (2) and (3). The existence of the linear maps $d$ and $d'$ and the exactness of the rows follows from  Lemma \ref{Lemma: Algebraic lemma for the exactness}. It remains to show the commutativity of the third square.
 Consider the quasi-isomorphisms
 \[(H,g)\colon  Cone(f) \longrightarrow C \quad \mbox{and}\quad (H',g')\colon Cone(f') 
\longrightarrow C' \]
 and the  chain map
 \[ \left(\begin{array}{cc}
           a & 0 \\
           R & b
          \end{array}
 \right) \colon Cone(f') \longrightarrow Cone(f).\]

Naturality of mapping cones implies that the following diagram of long exact 
sequences 
\begin{equation*}
\begin{tikzpicture}[->,auto,node distance=1.5cm ,>=latex',baseline=(current  bounding  box.center)]
\matrix (m) [matrix of nodes, row sep=2cm,column sep=2cm]
{ \phantom{f}$\ldots$ & $H_*(A')$ & $H_*(B')$ &  $H_*(Cone(f'))$  & $\ldots$\phantom{f}\\
  \phantom{f}$\ldots$ & $H_*(A)$ & $H_*(B)$ & $H_*(Cone(f))$ & $\ldots$\phantom{f}\\ };
 \path[->]
  (m-1-1) edge node[above] {$(p_1')_*$}   (m-1-2);
 \path[->]
  (m-1-2) edge node[above] {$f'_*$} (m-1-3)
          edge node[right] {$a_*$} (m-2-2);
 \path[->]
  (m-1-3) edge node[right] {$b_*$} (m-2-3)
          edge node[above] {$(i_2')_*$} (m-1-4);
 \path[->]
  (m-1-4) edge node[right] {$\left(\begin{array}{cc}
           a & 0 \\
           R & b
          \end{array} \right)_*$} (m-2-4)
          edge node[above] {$(p_1')_*$}   (m-1-5);
 \path[->]
  (m-2-1) edge node[below] {$(p_1)_*$}  (m-2-2); 
 \path[->]
  (m-2-2) edge node[below] {$f_*$} (m-2-3);
 \path[->]
  (m-2-3) edge node[below] {$(i_2)_*$} (m-2-4);
 \path[->]
  (m-2-4) edge node[below] {$(p_1)_*$}  (m-2-5);
\end{tikzpicture}
\end{equation*}
commutes. Moreover the diagram
 \begin{equation*}
 \begin{tikzpicture}[->,auto,node distance=1.5cm,>=latex',baseline=(current  bounding  box.center)]
 \matrix (m) [matrix of nodes, row sep=2cm,column sep=2cm]
{$H_*(Cone(f'))$ & $H_*(C',\partial')$ \\
$H_*(Cone(f))$ & $H_*(C,\partial)$ \\ };
  \path[->]
  (m-1-1) edge node[above] {$(H',g')_*$} (m-1-2)
      edge node[left] { $\left(\begin{array}{cc} a & 0 \\ R & b \end{array} \right)_*$} (m-2-1);
 \path[->]
  (m-2-1) edge node[below] {$(H,g)_*$} (m-2-2);
 \path[->]
  (m-1-2) edge node[right] {$c_*$} (m-2-2);
 \end{tikzpicture}
 \end{equation*}
 commutes because  $(T, S) \colon Cone(f') \longrightarrow C$ is a chain homotopy between 
$c \circ (H',g')$ and $ (H,g) \circ \left(\begin{array}{cc} a & 0 \\ R & b \end{array} \right)$. 
Hence the lemma follows.
\end{proof}

\begin{Lemma}\label{Come mettere a posto le cose}
Suppose we have a diagram as in Lemma \ref{Lemma: Algebraic lemma for the third square}.
We assume moreover that $A=A_+ \oplus A_-$ and decompose $a=(a_+, a_-)$, $f=f_++ f_-$ 
and $H=H_+ + H_-$. If there is a map $\alpha_- \colon A' \to A_-$ such that $a_-= \partial \circ
 \alpha_- + \alpha_- \circ \partial$, then there exist maps $\widetilde{R}$ and $\widetilde{T}$
 such that the diagram 
\begin{equation} \label{Diagram: modification of big diagram}
\begin{tikzpicture}[->,auto,node distance=1.5cm,>=latex',baseline=(current  bounding  box.center)]
\matrix (m) [matrix of nodes, row sep=2cm,column sep=4cm]
{  $A'$ & $B'$ & $C'$ \\
    $A_+$ & $B$ &  $C$ \\};
 \path[->]
  (m-1-1) edge node[above] {$f'$} (m-1-2)
          edge node[right] {$a_+$} (m-2-1)
          edge node[below] {$\widetilde{R}$} (m-2-2)
          edge[bend left=20] node[above] {$H'$} (m-1-3)
          edge[dashed] node[above,yshift=2ex,xshift=-7ex] {$\widetilde{T}$} (m-2-3);
 \path[->]
  (m-1-2) edge node[right,yshift=3.1ex] {$b$} (m-2-2)
          edge node[above] {$g'$} (m-1-3)
          edge node[above] {$S$} (m-2-3);
 \path[->]
  (m-1-3) edge node[right] {$c$} (m-2-3);
 \path[->]
  (m-2-1) edge node[below] {$f_+$} (m-2-2)
  edge[bend right=20] node[below] {$H_+$} (m-2-3);
 \path[->]
  (m-2-2) edge node[below] {$g$} (m-2-3);
 \end{tikzpicture}
\end{equation}
satisfies the properties of Diagram \eqref{Diagram: Big diagram in the algebraic lemma}.
\end{Lemma}
\begin{proof}
It is evident, by restriction, that $H_+$ is a homotopy between $g \circ f_+$ and $0$.
We define $\widetilde{R}=R+ f_- \circ \alpha_-$ and $\widetilde{T}= T+H_- \circ \alpha_-$.
Then we compute 
\begin{align*}
\partial \circ \widetilde{R} + \widetilde{R} \circ \partial' & = \partial \circ R + R \circ \partial' 
+ \partial \circ f_- \circ \alpha_- + f_- \circ \alpha_- \circ \partial \\ & = \partial \circ R + R \circ 
\partial' + f_- \circ (\partial \circ \alpha_- + \alpha_- \circ \partial') \\
&= f \circ a + b \circ f' + f_- \circ a_- = f_+ \circ a_+ + b \circ f'
\end{align*}
and
\begin{align*}
\partial \circ \widetilde{T} - \widetilde{T} \circ \partial & = \partial \circ T - T \circ \partial
+ \partial \circ H_- \circ \alpha_- + H_- \circ \alpha_- \circ \partial' \\
&= \partial \circ T - T \circ \partial + (\partial \circ H_- + H_- \circ \partial) \circ \alpha_- +
H_- \circ (\partial \circ \alpha_- + \alpha_- \circ \partial')  \\
& = c \circ H' + S \circ f' + g \circ R + H \circ a + g \circ f_- \circ \alpha_- + H_- \circ a_- \\
& =  c \circ H' + S \circ f' + g \circ \widetilde{R} + H_+ \circ a_+.
\end{align*}
\end{proof}

\subsection{The exact triangle in fixed point Floer cohomology} 
\label{Subsection: The exact triangle in symplectic homology}
In \cite{Se4} Seidel sketches the existence of an exact triangle that encodes the behaviour of fixed point Floer cohomology (in the $\pm$ versions) under the composition by a Dehn twist. In this section we prove the exactness of an analogous exact triangle for $HF^{\sharp}$. The proof is inspired by that for Seidel's exact triangle in Lagrangian Floer cohomology from \cite{Se5} and works also for the $\pm$ versions.

Fix a compact Liouville surface $(S,\lambda)$ with genus $g\geq 1$ and boundary $\partial S \cong S^1$ and an exact symplectomorphism $\vf: S \rightarrow S$. We denote $\omega= d \lambda$. If $L \subset S\setminus \partial S$ is an exact Lagrangian submanifold, i.e. a closed essential embedded curve such that $[\lambda|_L]=0 \in H^1_{dR}(L)$, we denote by $\tau_L \colon S \rightarrow S$ the positive Dehn twist along $L$, which is an exact symplectomorphism. If $L'$ is another exact Lagrangian submanifold, we denote by $HF(L,L')$ the \emph{Lagrangian Floer cohomology of $(L,L')$}.

Suppose (without loss of generality) that $L$ and $L'$ intersect transversely. The Floer chain complex $CF(L,L')$ is the $\Z/2\Z$-vector space freely generated by the intersection point $x \in L \cap L'$. To define the boundary, we consider the symplectic fibration 
\[(\R \times [0,1] \times S,ds\wedge dt + \omega) \stackrel{\pi}{\longrightarrow} (\R \times [0,1],ds\wedge dt)\]
endowed with a compatible almost complex structure $J$ such that $J(\partial_s) = \partial_t$ and $J(TS) = TS$. Given $x_+,x_- \in L\cap L'$, let $\mathcal{M}(x_+,x_-,J)$ be the moduli space of $J$-holomorphic sections $u\colon \R\times [0,1] \rightarrow \R \times [0,1] \times S$ of $\pi$ such that
\begin{enumerate}
\item $\lim_{s \rightarrow \pm\infty} \pi \circ u(s,t) = x_\pm$,
\item $u(s, 0) \in L$ for all $s \in \R$, and
\item $u(s,1) \in L'$ for all $s \in \R$.
\end{enumerate}

As in the fixed point case, to each $J$--holomorphic section $u$ is associated a Fredholm operator $D_u$ of index $\mathrm{ind}(u)$. Call $\mathcal{M}_k(x,y,J)$ the subset of maps in $\mathcal{M}(x,y,J)$ with $\mathrm{ind}(u)=k$. Define then a differential on $CF(L,L')$ by
\[\partial x_+ = \sum_{y\in L\cap L'} \#_2 (\overline{\mathcal{M}}_1(x_+,x_-,J)) x_- \]
where again $\overline{\mathcal{M}}_1(x_+,x_-,J)$ denotes the quotient of $\mathcal{M}_1(x_+,x_-,J)$ by the $\R$-action given by translations in the $\R$-direction. The resulting homology 
\[HF(L,L') \coloneqq H_*(CF(L,L'),\partial)\]
is the Lagrangian Floer cohomology of $L$ and $L'$.
The edfinition of Lagrangian Floer cohomology using sections is (tautologically) equivalent to the usual definition using time-dependent almost complex structures. 

\begin{Thm} \label{Theorem: Exact triangle in SH} There is an exact triangle
\begin{equation} \label{Triangle: L phi phi-tau}
\begin{tikzpicture}[->,auto,node distance=1.5cm,>=latex',baseline=(current  bounding  box.center)] 

  \node (1) {$HF(\vf(L),L)$};
  \node (2) [right = of 1] {$HF^{\sharp}(\vf)$};
  \node (3) [right = of 2] {$HF^{\sharp}(\vf \circ \tau_L^{-1})$.};

  \path[every node/.style={font=\sffamily\small}]
    (1) edge node[above] {$\iota_*$} (2)
    (2) edge node[above] {$\lambda_*$} (3)
    (3) edge[bend left=20] node {$\delta$} (1);
\end{tikzpicture}
\end{equation}
\end{Thm}
In comparing this result with Theorem 4.2 of \cite{Se4}, the reader should keep in mind 
the different conventions. 

Before giving the proof of the Theorem we describe the maps $\iota$ and $\lambda$, which are defined in terms of \emph{Gromov invariants associated to exact Lefschetz fibrations} over surfaces. We refer the reader to Seidel's book \cite{Se1} for the general theory of these invariants and to \cite{Se4} and \cite{Se5} for further details about the construction of $\iota$ and $\lambda$ and other closely related maps.

Let $B$ be an oriented surface (possibly with boundary) with a set $\Gamma$ of marked points and consider a Lefschetz fibration $E \stackrel{\pi}{\twoheadrightarrow} B\setminus \Gamma$ with smooth fiber a surface $S$ and without critical points above $\partial B$. Fix a symplectic form $\Omega$ on $E$ of the form $\Omega = \theta + \pi^*\kappa$ with $\kappa$ a symplectic form on $B$ and $\theta$ an exact two-form on $E$ which is positive on each fiber. Assume that $\Omega$ is exact and let $\Upsilon$ be a primitive 
of $\Omega$ which restricts to a Liouville form on every fibre.

If $B$ is without boundary, one can endow $B$ and $E$ with suitable almost complex structures and define the \emph{Gromov invariant} $G(E,\pi)$ by counting (Fredholm) index--$0$ pseudo-holomorphic sections of $\pi$. If $B$ has non-empty boundary we fix an \emph{exact Lagrangian boundary condition} on $\partial E$: this consists in a Lagrangian submanifold $Q \subset \pi^{-1}(\partial B)$  such that $\pi|_Q \colon Q \to \partial B$ is a localy trivial fibration and $[\Upsilon|_{Q}] = 0 \in H_{dR}^1(Q)$. Observe that there is a canonical symplectic connection on $E$ given by the $\Omega$-orthogonal of the tangent space of the fibers, and an exact Lagrangian submanifold $L$ on a fiber above $\partial B \setminus \Gamma$ induces, by parallel transport, an exact Lagrangian boundary condition on the corresponding connected component of $\partial B \setminus \Gamma$. 

To this data it is possible to associate a \emph{relative Gromov invariant} $G_{rel}(E,\pi,Q)$, which counts index--$0$ pseudo-holomorphic sections of $\pi$ with boundary in $Q$. 

\begin{Rmk} In the literature (and in particular in Seidel's work) the Gromov invariants $G(E,\pi)$ and $G_{rel}(E,\pi,Q)$ are usually denoted by $\Phi(E,\pi)$ and $\Phi_{rel}(E,\pi,Q)$ respectively. We prefer to change the notation here to avoid possible confusions with the map $\Phi$ from Heegaard Floer homology to periodic homology.
\end{Rmk}

\begin{enumerate}[leftmargin=*]
 \item Let $\D$ be the unit disc in $\C$. To define $\iota \colon CF(\vf(L),L) \longrightarrow CF^{\sharp}(\vf)$, consider the surface $B_{\iota}= \D \setminus \{0,1\}$ and the symplectic fibration $\pi_{\iota} \colon E_{\iota}^{\vf} \to B_{\iota}$ with fiber $S$ and monodromy $\vf$. On $B_\iota$ we fix a positive strip-like and at $1$ and a negative cylindrical end at $0$, i.e.\ holomorphic identifications of a punctured neighbourhood of $1$ in $B_\iota$ with $(0, + \infty) \times [0,1]$ and of a punctured neighbourhood of $0$ with $(- \infty, 0) \times S^1$. We fix also a trivialisation of $\pi_{\iota}$ over the strip-like end, which means an identification of the preimage of the strip-like end with $S \times \R \times [0,1]$ equipped with a product symplectic form.  We equip the boundary of $E_{\iota}^{\vf}$ with an exact Lagrangian boundary condition $Q_L$ which restricts to $(L \times (0, + \infty) \times \{ 1 \}) \cup (\vf(L) \times (0, + \infty) \times \{ 0 \})$ on the strip-like end. Such exact Lagrangian boundary condition is obtained as the trace of the parallel transport of $L \times \{ c \} \times \{ 1 \}$ for some $c>0$.

Let ${\mathcal M}(x, \gamma)$ be the moduli space of $J$-holomorphic sections $u \colon
B_\iota \to E_\iota^\vf$ with boundary on $Q_L$, positively asymptotic to $x \in \vf(L) \cap L$ and
negatively asymptotic to a periodic orbit $\gamma$ corresponding to a fixed point of $\vf$.
We denote by  ${\mathcal M}_k(x, \gamma)$ the subset of ${\mathcal M}(x, \gamma)$ consisting of sectons of Fredholm index $k$. For a generic choice of almost complex structre,
${\mathcal M}_0(x, \gamma)$ is a transversely cut out compact manifold of dimension zero, and we define 
$$\iota(x)= \sum_{\gamma \ne e} \# {\mathcal M}_0(x, \gamma) \gamma.$$ 

\begin{figure} [ht!] 
  \begin{center}
   \includegraphics[scale = 1.0]{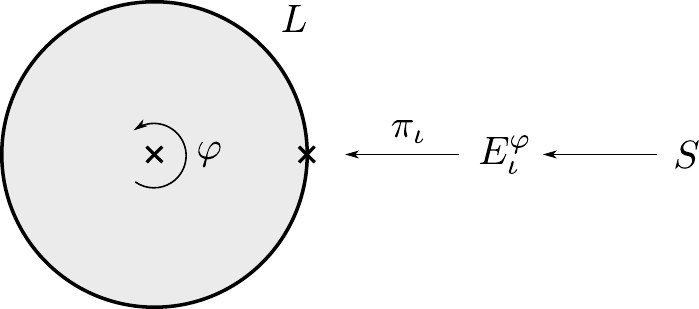}
  \end{center}
  \caption{The fibration $(E_{\iota}^{\vf},\pi_{\iota})$ with Lagrangian boundary condition $Q_L$.}
\end{figure}

\item To define $\lambda \colon CF^{\sharp}(\vf) \longrightarrow CF^{\sharp}(\vf \circ \tau_L^{-1})$, consider a Lefschetz fibration $(E_{\lambda}^{\vf},\pi_{\lambda})$ over $B_{\lambda} \coloneqq \C^*$  with fiber $S$ and only one critical point $\mathfrak{x}$ with associated vanishing cycle $L$ and such that the symplectic monodromy around  $0$ is $\vf \circ \tau_L^{-1}$. Then the symplectic monodromy around $\mathfrak{z}= \pi_{\lambda}(\mathfrak{x})$ is isotopic to $\tau_L$ and around a circle of radius bigger than $|\mathfrak{z}|$ it is $\vf$.

Let ${\mathcal M}(\gamma_+, \gamma_-)$ the moduli space of $J$-holomorphic sections
$u \colon B_\lambda \to E_\lambda^\vf$ whigh are positively asymptotic at $\infty$ to a periodic orbit $\gamma_+$ corresponding to a fixed point of $\phi$ and negatively asymptotic at $0$ to a periodic orbit $\gamma_-$ corresponding to a fixed point of $\vf \circ \tau_L^{-1}$. 
We denote by  ${\mathcal M}_k(\gamma_+, \gamma_-)$ the subset of   ${\mathcal M}(\gamma_+, \gamma_-)$ consisting of sections of Fredholm index $k$.  For a generic choice of almost complex structure, ${\mathcal M}_0(\gamma_+, \gamma_-)$ is a transversely cut out compact manifold of dimension zero, and for $\gamma_+ \ne e$ we define 
$$\lambda(\gamma_+)= \sum_{\gamma_- \ne e} \# {\mathcal M}_0(\gamma_+, \gamma_-) \gamma_-.$$
\begin{figure} [ht!] 
  \begin{center}
  \includegraphics[scale = 1.0]{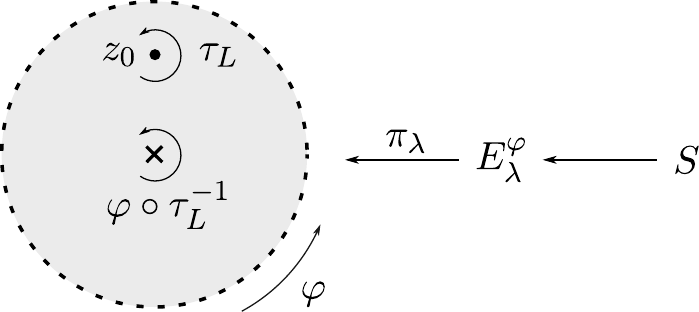}
  \end{center}
  \caption{The fibration $(E_{\lambda}^{\vf},\pi_{\lambda})$. Here and in the future, the points marked with a cross are removed and those marked with a point represent critical values of the corresponding Lefschetz fibration. The point $z_0$ should be labelled $\mathfrak{z}$ instead.}
\end{figure}

\begin{Rmk}
 Observe that $B_{\lambda}\cong S^2 \setminus \{N,S\}$, where punctured neighborhoods of $N$ and $S$ (the north and the south pole respectively) are identified with the cylindrical--like ends $0$ and $\infty$ respectively. In the rest of the paper we will often and implicitly consider $E^{\vf}_{\lambda}$ as a fibration over $S^2 \setminus \{N,S\}$.
\end{Rmk}

\item  We want describe also another Gromov invariant that will be useful later. Consider the Lefschetz fibration $\overline{\pi}_L \colon \overline{E}_L \to \overline{B}= \D \setminus \{-1\}$ with a unique singular fiber over $0$ with vanishing cycle $L$. If $L'$ is another exact Lagrangian submanifold $S$, we can define an exact Lagrangian boundary condition $Q_{L'}$ in $\partial \overline{E}_L$ by parallel transport of $L'$ such that the count of index zero 
$J$-holomorphic sections $u \colon \overline{B} \to \overline{E}_L$ defines an element $\overline{\mathcal{G}}_L(L') \in CF(L', \tau_L (L'))$. 
\end{enumerate}

\begin{Lemma}
The maps $\iota$ and $\lambda$ are chain maps and $\overline{\mathcal{G}}_L(L')$ defines an element in $HF(L', \tau_L (L'))$.
\end{Lemma}
\begin{proof}
The lemma follows from standard degeneration arguments for one-dimensional moduli spaces and an argument similar to the proof of Theorem \ref{Theorem: Phi is filtered}.
\end{proof}

\begin{Ex} \label{Example: G of tau L with condition L is 0} If $L' = L$ then 
 \[0 = \overline{\mathcal{G}}_L(L) \in HF(L, \tau_L (L)) \cong HF(L,L).\] 
 See \cite[ Example 3.1]{Se3}.
\end{Ex}


\begin{Lemma} \label{lambda composed iota is trivial} 
There is a chain homotopy 
\[\mathcal{H} \colon CF(\vf(L), L) \to HF^\sharp(\vf \circ \tau_L^{-1}) \] from $\lambda \circ \iota$ to the $0$-map. 
\end{Lemma}
\begin{proof}
For every $t \in (0,1)$, fix the marked point $\mathfrak{z}_t=(\frac{\sqrt{2}t}{2},\frac{\sqrt{2}t}{2})  \in B_\iota$ and consider the 
family of Lefschetz fibrations $(E_t^{\mathcal H}, \pi_t^{\mathcal H})$ over $B_\iota$ with monodromy $\wf \circ \tau_L^{-1}$ around $0$ and a unique critical point over $\mathfrak{z}_t$ with vanishing cycle $L$. On the preimages of $\partial B_\iota$ we consider the Lagrangian submanifolds $Q_t^{\mathcal H}$ obtained by parallel transport of $L$. The homotopy ${\mathcal H}$ is defined by counting pairs $(t_0, u)$, where $t_0 \in (0,1)$ and $u$ is an index $-1$ $J$-holomorphic section of the Lefshetz fibration $(E_{t_0}^{\mathcal H}, \pi_{t_0}^{\mathcal H})$ with boundary on $Q_{t_0}^{\mathcal H}$ for a generic almost complex structure $J$.

The family of Riemann surfaces with marked points $(B_\iota, \mathfrak{z}_t)$ can be compactified by adding two-level buildings 
\[ B_{-}^{\mathcal H}=  B_{\iota} \sqcup  (S^2 \setminus (\{\mathrm{N}\}, \{\mathrm{S}\}))\]
with a marked point $\mathfrak{z}_- \in S^2 \setminus (\{\mathrm{N}\}, \{\mathrm{S}\})$ and
\[ B_+^{\mathcal H} = \left(B_{\iota} \setminus \{(\frac{\sqrt{2}}{2},\frac{\sqrt{2}}{2})\} \right) \sqcup \left(\D\setminus\{(-\frac{\sqrt{2}}{2},-\frac{\sqrt{2}}{2})\}\right)\] 
with the marked point $\mathfrak{z}_+ =0 \in \D \setminus \left ( \{(-\frac{\sqrt{2}}{2},-\frac{\sqrt{2}}{2})\} \right )$. To the degenerations of $(B_\iota, \mathfrak{z}_t)$ correspond, by pull back, degenerations of $E_t^{\mathcal H}$. 


 \begin{figure} [ht!] 
  \begin{center}
   \includegraphics[scale = 1]{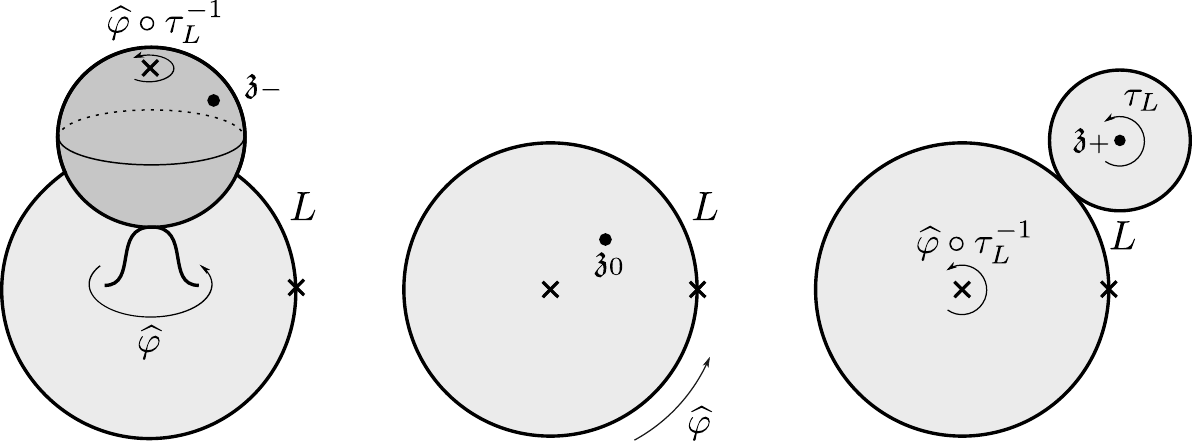}
  \end{center}
  \caption{The homotopy from $B^{\mathcal{H}}_-$ (left) to $B^{\mathcal{H}}_+$ (right). $(B^{\mathcal{H}}_\iota, \mathfrak{z}_0)$ is pictured in the middle. 
  }
  \label{Figure: Degeneration of H}
 \end{figure}
 
 
Thus $(E_-^{\mathcal H}, \pi_-^{\mathcal H}) = (E^\vf_\iota \sqcup E^\vf_\lambda, \pi_\iota \sqcup \pi_\lambda)$, and therefore  counting $J$-holomorphic sections of  $(E_-^{\mathcal H}, \pi_-^{\mathcal H})$ gives $\lambda \circ \iota$. On the other hand, $(E_+^{\mathcal H}, \pi_+^{\mathcal H}) \cong (\overline{E}_L \sqcup E_{\iota}^{\vf\circ \tau_L^{{-1}}}, \overline{\pi}_L \sqcup \pi_\iota)$. Since the Lagrangian boundary condition on the component $\overline{E}_L$ is induced by $L$, Example \ref{Example: G of tau L with condition L is 0} implies the relative Gromov invariant associated to $(E_+^{\mathcal H}, \pi_+^{\mathcal H})$ is trivial, and therefore the total count of $J$-holomorphic sections of $(E_+^{\mathcal H}, \pi_+^{\mathcal H})$ is zero.
\end{proof}



We assume, without loss of generality, that $L$ is disjoint from all fixed point of $\vf \circ 
\tau_L^{-1}$ and intersects $\vf(L)$ transversely. We fix an orientation on $L$ and orient 
$\vf(L)$ accordingly. We say that an intersection point $x \in L \cap \vf(L)$ is positive if the 
orientation of $T_xL$, followed by the orientation of $T_x\vf(L)$, gives the orientation of 
$T_xS$. We say that a fixed point of $\vf$ is {\em positive hyperbolic} if the eigenvalues of 
$d_x \vf$ are  real and positive, and {\em negative hyperbolic} if they are real and negative. 

\begin{Lemma}\label{creation}
We can choose a Hamiltonian isotopy representative of $\vf$ so that for 
every intersection 
point $x \in L \cap \vf(\tau_L^{-1}(x))$ there is a neighbourhood ${\mathcal U}_x$ of $x$ in 
$S$ such that:
\begin{enumerate}
\item no fixed point of $\vf \circ \tau_L^{-1}$ is contained in ${\mathcal U}_x$,
\item exactly one fixed point of $\vf$ is contained in ${\mathcal U}_x$, and it is hyperbolic 
with the same sign as $x$, and
\item if $x$ and $x'$ are distinct intersection points, then ${\mathcal U}_x$ and 
${\mathcal U}_{x'}$ are disjoint.
\end{enumerate} 
\end{Lemma}
\begin{proof}
For the sake of the proof, we denote $\vf \circ \tau_L^{-1}= \psi$. We also denote $L'= \psi(L)$.
Let $x \in L \cap L'$ be an intersection point and denote $y=\psi^{-1}(x)$. We fix Weinstein 
tubular neighbourhoods ${\mathcal N}$  of $L$ and ${\mathcal N}'$ of $L'$ such that 
${\mathcal N}' = \psi({\mathcal N})$. We choose ${\mathcal N}$ small enough that  no fixed
point of $\psi$ is contained in ${\mathcal N}$ or ${\mathcal N}$.

We fix open connected neighbourhoods ${\mathcal V}$ of $y$ and ${\mathcal V}'$ of $x$ in 
$L$ and choose symplectic identifications ${\mathcal N} \cong L \times (- \varepsilon, 
\varepsilon)$ and ${\mathcal N}'\cong L' \times (-\varepsilon, \varepsilon)$ such that
$\psi({\mathcal V} \times (-\varepsilon, \varepsilon))= {\mathcal V}' \times (- \varepsilon, 
\varepsilon)$.  We define ${\mathcal U}_x={\mathcal V}' \times (- \varepsilon, \varepsilon)$. We parametrise $L$ as $[0, 2 \pi]/ 0 \sim 2 \pi$ so that $0 < x <
y< 2 \pi$, ${\mathcal V} = (y - \varepsilon, y+ \varepsilon)$, 
and  ${\mathcal V}' = (x - \varepsilon, x + \varepsilon)$.

We have two possibilities for $\psi|_{{\mathcal V} \times (- \varepsilon, \varepsilon)}$. If $\theta$ is 
the coordinate on $L$ and $t$ is the coordinate on $(- \varepsilon, \varepsilon)$, they are:
\begin{enumerate}
\item $\psi(\theta, t)= (x- t, \theta-y)$, and \label{case 1}
\item $\psi(\theta, t)=(x+t, y- \theta)$. \label{case 2}
\end{enumerate}
The first case corresponds to a positive intersection point between $L$ and $\psi(L)$ and
the second case corresponds to a negative intersection point.

We represent the positive Dehn twist around $L$ by the map
\begin{equation} \label{Dehn twist}
\tau_L(\theta, t)= (\theta + f(t), t),
\end{equation}
where of course $\theta + f(t)$ is to be understood modulo $2 \pi$. If we want the proof to 
work for all intersection points between $L$ and $L'$ at the same time, we we need that the 
parametrisations of $L$ associated to different intersection points differ only by a 
translation, so that the Dehn twists is represented by Equation \eqref{Dehn twist} in the 
neighbourhood of every intersection point.

First we consider case \eqref{case 1}. We are looking for $(\theta_0, t_0) \in (x - \varepsilon, 
x + \varepsilon) \times (- \varepsilon, \varepsilon)$ such that $\tau_L(\theta_0, t_0) \in 
(y - \varepsilon, y + \varepsilon) \times (- \varepsilon, \varepsilon)$ and 
$\psi(\tau_L(\theta_0, t_0))= (\theta_0, t_0)$. We have $\psi(\tau_L(\theta_0, t_0))= (x-t_0, 
\theta_0 + f(t_0)-y)$ and therefore we have to solve the system
$$\begin{cases}
\theta_0 = x-t_0, \\ t_0= \theta_0 + f(t_0)- y \\ | \theta_0 - y + f(t_0)| < \varepsilon
\end{cases}$$

The last inequality is automatic from the second equation, provided that the solution satisfies
$t_0 \in (- \varepsilon, \varepsilon)$.  

From the two equations we obtain 
$$f(t_0)= y-x+2t_0.$$
We define the function $g(t)= y-x+2t$ and observe that, if we choose $\varepsilon$ small 
enough,
\begin{itemize}
\item $f(- \varepsilon) = 0$ and $g(- \varepsilon) = y-x-2\varepsilon >0$,
\item $f(\varepsilon)= 2 \pi$ and $g(\varepsilon) = y-x + 2 \varepsilon < 2 \pi$.
\end{itemize}
Then the graphs of $f$ and $g$ must cross, and since we can arrange $f$ to be linear on the 
preimage of $(y-x-2\varepsilon, y-x + 2 \varepsilon)$, we can assume that they cross at a 
single point $t_0$. We set $\theta_0= x-t_0$, and therefore $(\theta_0, t_0)$ is the unique 
fixed point of $\psi \circ \tau_L$ in ${\mathcal U}' \times (- \varepsilon, \varepsilon)$. 
The linearisation of $\psi \circ \tau_l$ at $(\theta_0, t_0)$ is $\left ( \begin{matrix} 0 & -1 \\
1 & f'(t_0) \end{matrix} \right )$. The determinant is $1$ and the trace is $f'(t_0)>2$, and 
therefore the eigenvalues are positive real. Then $(\theta_0, t_0)$ is a positive hyperbolic 
fixed point.

Now we consider case \eqref{case 2}. Since $\psi(\tau_L(\theta, t))=(x+t, y-\theta-f(t))$, we 
need to solve the system
$$\begin{cases}
\theta=x+t \\ t=y- \theta - f(t) \\ |x- \theta- f(t)| < \varepsilon.
\end{cases}$$
As in case \eqref{case 1}, solving this system is equivalent to solving the equation 
$f(t)=y-x-2t$. We define $g(t)= y-x-2t$ and observe that
\begin{itemize}
\item $f(- \varepsilon)=0$ and $g(- \varepsilon)>0$,
\item $f(\varepsilon)= 2 \pi$ and $g(\varepsilon) < 2 \pi$.
\end{itemize}
This gives a fixed point $(\theta_0, t_0)$ as before. The linarisation of $\psi \circ \tau_L$ at 
$(\theta_0, t_0)$ is $\left ( \begin{matrix} 0 & 1 \\ -1 & -f(t_0) \end{matrix} \right )$. This is 
the negative of the matrix form the previous case, and therefore its eigenvalues are negative 
real. Then  $(\theta_0, t_0)$  is a negative hyperbolic fixed point.
\end{proof}

The next step of the proof of theorem \ref{Theorem: Exact triangle in SH} involves a quite standard argument in Floer homologies (see for example \cite{Se5} and \cite[Section 9]{OS2}). Consider the \emph{iterated cone} $(C,\partial^C)$ where $C \coloneqq \left(CF(\vf(L), L) \oplus CF^{\sharp}(\vf) \oplus CF^{\sharp}(\vf \circ \tau_L^{-1})\right)$ and 
 \begin{equation*}
  \partial^C \coloneqq \left(\begin{array}{ccc} 
   \partial^{CF(\vf(L), L)}     &     0     & 0\\
   \iota &     \partial^{CF^{\sharp}(\vf)}     & 0\\
    \mathcal{H}     & \lambda & \partial^{CF^{\sharp}(\vf \circ \tau_L^{-1})}
  \end{array}\right).
 \end{equation*}
 If we show that $H_*(C,\partial^C)=\{0\}$ then Theorem \ref{Theorem: Exact triangle in SH} comes then from Lemma \ref{Lemma: Algebraic lemma for the exactness}.


In the symplectic fibrations $W_* \xrightarrow{\pi_*} B_*$ we fix a fibrewise symplectic form $\Omega_*$, which induces a splitting $TW_*= T^vW_* \oplus T^hW_*$. The fibrations $W_* \xrightarrow{\pi_*} B_*$ is {\em non-negatively curved} if $\Omega_*|_{T^hW_*} \ge 0$ with respect to the orientation of $T^hW_*$ induced by the orientation of $B_*$. The {\em energy} of a section $u \colon B_* \to W_*$ is $E(u)= \int_{S_*} u^*\Omega_*$. 
A section $u \colon B_* \to W_*$ is called {\em horizontal} if $du(TB_*)=T^hW_*$. An almost complex structure $J$ is {\em horizontal} if it preserves both distributions.  Horizontal sections are always $J$-holomorphic for a horizontal almost complex structure $J$. If $W_* \xrightarrow{\pi_*} B_*$ is non-negatively curved and $u$ is a $J$-holomorphic section, then $E(u) \ge 0$, and if $E(u)=0$ implies that $u$ is horizontal. The  symplectic fibrations
$E^\vf_\iota \to B_\iota$ and $E^\vf_\lambda \to B_\lambda$ are non-negatively curved,  and a section $u$ is horizontal if and only if $E(u)=0$. This is immediate for  $E^\vf_\iota \to B_\iota$, and follows from \cite[Subsection 3.3]{Se5} for  $E^\vf_\lambda \to B_\lambda$. By \cite[Lemma 2.27]{Se5} (see also \cite[Lemma 3.2]{Se-lectures}), a horizontal section of Fredholm index zero is regular (i.e.\ the linearised Cauchy-Riemann operator is surjective) if and only if the corresponding fixed point of $\vf$ is nondegenerate.

 Given $\delta \geq 0$, we denote by
 \begin{equation*}
  \partial_{\delta}^C = \left(\begin{array}{ccc} 
   \partial_{\delta}^{CF(\vf(L), L)}     &     0     & 0\\
   \iota_{\delta} &     \partial_{\delta}^{CF^{\sharp}(\vf)}     & 0\\
    \mathcal{H}_{\delta}     & \lambda_{\delta} & \partial_{\delta}^{CF^{\sharp}(\vf \circ \tau_L^{-1})}
  \end{array}\right).
 \end{equation*}
the component of $\partial^C$ that counts only those holomorphic sections $u$ counted by $\partial^C$ that have energy $E(u)<\delta$.

 If we show that there exists $\delta$ such that $\partial^C - \partial^C_\delta$ counts only holomorphic sections $u$ with energy $E(u) \ge 2 \delta$ and  $H_*(C,\partial_{\delta}^C)=\{0\}$, then by Lemma 2.31 of \cite{Se5} it follows that $H_*(C,\partial^C)=\{0\}$. To show this, we analyse the low-energy contributions to the components of $\partial^C$.

If $\delta$ is sufficiently small, $\partial_{\delta}^{CF(\vf(L), L)}=\partial_{\delta}^{CF^{\sharp}(\vf)}
= \partial_{\delta}^{CF^{\sharp}(\vf \circ \tau_L^{-1})}$ and $\mathcal{H}_{\delta}=0$ because no horizontal section contributes to that maps. For $\delta$ small enough, $\lambda_\delta$ counts horizontal sections of $W^\vf_\lambda \to B_\lambda$, which correspond to the fixed points of $\vf \circ \tau_L^{-1}$, and therefore $\lambda_\delta= \lambda_0$. 

The map $\iota_\delta$ is more involved, because it doesn't necessarily count horizontal sections. In fact horizontal sections of $W^\vf_\iota \to B_\iota$ with boundary on $Q_L$ correspond to fixed points of $\vf$ which are also intersection points between $L$ and $\vf(L)$.
If $x \in L \cap \vf(L)$, we denote by $\gamma_x$ the orbit corresponding to the fixed point $p_x$ of $\vf$ close to $x$ which was constructed in Lemma \ref{creation}. 
\begin{Lemma}\label{afa}
If $\delta$ and ${\mathcal N}$ are small enough and $u \colon B_\iota \to E^\vf_\iota$ is a $J$-holomorphic section with boundary on $Q_L$ and energy $E(u)< \delta$, then $u \in {\mathcal M}(x, \gamma_x)$ for some $x \in L \cap \vf(L)$.  
\end{Lemma}
\begin{proof}
The proof is based on the two following two facts:

\begin{itemize}
\item[(i)] up to Hamiltonian isotopy, all orbits of $\vf$ and all intersection points in $L \cap \vf(L)$ have distinct action, and
\item[(ii)] for every intersection point $x \in L \cap \vf(L)$, the action of $\gamma_x$ can be made arbitrarily close to the action of $x$ by a suitable choice of the parameters in the definition of the Dehn twist.
\end{itemize}
(i) If $p$ is a fixed point of $p$ and $H \colon S \to \R$ is a function with $d_pH=0$ and Hamiltonian flow $\vf^H_t$, then $p$ is a fixed point of $\vf_t^H \circ \vf$. If
$\mathfrak{a}_\vf(x)$ and $\mathfrak{a}_{vf_t^H \circ \vf}(p)$ denote the action of $p$ as a fixed point of $\vf$ and of $vf_t^H \circ \vf$ respectvely, then
$\mathfrak{a}_{vf_t^H \circ \vf}(p)=\mathfrak{a}_\vf(x)-tH(p)$.

(ii) By exactness, it is enough to compute the energy of one (smooth) section between $x$ and 
$\gamma_x$. If we trivialise the fibration $W_\iota^\vf \to B_\iota=D^2 \setminus \{0,1\}$ over 
$D^2 \setminus [0,1]$ and project to the fibre $S$, we obtain a correspondence between smooth sections from $x$ to $\gamma_x$ and boundary on $Q_L$ with maps from a triangle
to $S$ with vertices on $x$, $\phi^{-1}(x)$  and $p_x$ (in counterclockwise order) such that the edge between  $x$ and $\phi^{-1}(x)$ is in $L$, and the edge between $x$ and $p_x$ is 
the image under $\vf$ of the edge between $\phi^{-1}(x)$ and $p_x$. Moreover, the energy of a section is equal to the area covered by the triangle. 

We can find a triangle as above which is embedded and contained in ${\mathcal N}$ if the edge $x$ and $p_x$ is contained in ${\mathcal N} \cap \vf({\mathcal N})$. Then the energy of
any section $u \in {\mathcal M}(x, \gamma_x)$ is not larger than the area of ${\mathcal N}$, which can be made arbitrarily small by reducing $\varepsilon$.
\end{proof}
\begin{Lemma}\label{caldo}
If $x \in L \cap \vf(L)$ is also a hyperbolic fixed point of $\vf$ and the signs as intersection point and as fixed point are the same, then then the horizontal section $u_x$ over $x$ has Fredholm index $\operatorname{ind}(u_x)=0$.
\end{Lemma}
\begin{proof}
By Proposition 11.13 in \cite{Se1}, Theorem 9 in \cite{Dr} and the additivity of the index, we have $\operatorname{ind}(u_x)=\mu(x)- \mu(\gamma_x)$, where $\mu(x)$ is the Conley-Zehnder index of the intersection point $x$, $\mu(\gamma_x)$  is the Conley-Zehnder index of the orbit $\gamma_x$ corresponding to $x$, and they are computed with respect to trivialisations which extend to a trivialisation of $u_x^*T^vW_\iota^\vf$.

Let $\xi_-$ and $\xi_+$ the stable and unstable directions of $\vf$ at $x$, respectively. They give sub-bundles $\Xi_\pm$ of $u_x^*T^vW_\iota^\vf$. The Conley-Zehnder index  $\mu(\gamma_x)$ is equal to the Maslov index of $\Xi_\pm$. The sub-bundle $T^vQ_L=TQ_L \cap T^vW_\iota^\vf \subset T^vW_\iota^\vf|_{\partial B_\iota}$ is always transverse to $\Xi_\pm$. On the strip-like end we close it by the shortest clockwise path from $\vf(L)$ to $L$. This path does not intersect $\xi_\pm$ because the sign of $x$ as intersection point is the same as fixed point. The Conley Zehnder index $\mu(x)$ is equal to the Maslov index of the closure of $T^vQ_L$, which is homotopic to $\Xi_\pm$. Then $\mu(x)=\mu(\gamma_x)$.
\end{proof}
\begin{Prop}
If $\delta$ is small enough, $\iota_\delta(x)=\gamma_x$.
\end{Prop}
\begin{proof}
By Lemma \ref{afa}, $\iota_\delta(x)= \# {\mathcal M}_0(x, \gamma_x) \gamma_x$. For the moment, let us assume that $x$ is also a fixed point of $\vf$. Then ${\mathcal M}(x, \gamma_x)$ contains only the horizontal section $u_x$ because $E(u_x)=0$ and by exactness all holomorphic sections in ${\mathcal M}(x, \gamma_x)$ have the same energy. By Lemma \ref{caldo} $\operatorname{ind}(u_x)=0$ and moreover $u_x$ is regular (i.e.\ its linearised Cauchy-Riemann operator is surjective) by Lemma 2.27 of \cite{Se5}.  Then $\# {\mathcal M}_0(x, \gamma_x) =1$ when $x$ is a fixed point of $\phi$, which happens when $L$ is parametrised so that $x$ and $\phi(\tau_L^{-1}(x))$ are antipodal. There are obstructions to obtain this for every intersection point in $L \cap \vf(L)$, and therefore we use a deformation argument as in the proof of Proposition 3.4 in \cite{Se5}. For every point $x \in L \cap \vf(L)$
we define a family of symplectomorphisms $\vf_t$ by changing the parametrisation of $L$ such that $\vf_1=\vf$ and $\vf_0(x)=x$. To this family of symplectomorphisms we associate a 
family of symplectic fibrations $W_\iota^{\vf_i} \to B_\iota$. Let $\gamma_x^{(i)}$ be the periodic orbit corresponding the fixed point of $\vf_i$ close to $x$ constructed in Lemma \ref{creation}; in particular $\gamma_x^{(1)}=\gamma_x$. We define the parametric moduli space ${\mathcal M}^{par}(x, \gamma_x)$ consisting of pairs $(u,t)$ where $t \in [0,1]$ and $u \colon B_\iota \to W_\iota^{\vf_i}$ asymptotic to $x$ and $\gamma_x^{(i)}$. For a generic choice of almost complex structures 
${\mathcal M}^{par}(x, \gamma_x)$ is a $1$-dimensional manifold. Moreover a sequence $(u_n, t_n) \in {\mathcal M}^{par}(x, \gamma_x)$ with $t_n \to t_\infty \in [0,1]$ cannot break into a multi-level building by action reasons. Then $\partial {\mathcal M}^{par}(x, \gamma_x)= {\mathcal M}(x, \gamma_x^{(1)})- {\mathcal M}(x, \gamma_x^{(0)})$ and therefore $\#{\mathcal M}(x, \gamma_x)=1$.
\end{proof}
We have obtained that $\partial^C_\delta = \left (  \begin{matrix} 0 & 0& 0 \\
\iota_0 & 0 & 0 \\ 0 & \lambda_0 & \end{matrix} \right ),$ and therefore $H_*(C, \partial^C_\delta)=0$. From this it follows that $H_*(C, \partial^C)$, and therefore we have proved Theorem \ref{Theorem: Exact triangle in SH}.
\subsection{The exact triangle in Heegaard Floer homology} \label{Subsection: The surgery triangle}
The aim of this subsection is to prove the exact sequence \ref{Triangle: Exact triangle in HF without maps}. Although we have no claim of originality on that exact sequence, we need to recast its proof in the language of symplectic fibrations.
\begin{Prop} \label{Prop: Exact triangle in HFK}
 Given a genus $g$ open book decomposition $(K,S,\vf)$ of a $3$--manifold $Y$ and a closed exact Lagrangian $L \subset S$, there exists an exact triangle

 \begin{equation} \label{Triangle: Exact triangle in HF}
 \begin{tikzpicture}[->,auto,node distance=0.8cm,>=latex',baseline=(current  bounding  box.center)]

  \node (1) {$\widehat{HFK}(\overline{Y_0}, \overline{K_0}; -g+1)$};
  \node (2) [right = of 1] {$\widehat{HFK}(\overline{Y},\overline{K};-g+1)$};
  \node (3) [right = of 2] {$\widehat{HFK}(\overline{Y_+},\overline{K_+};-g+1),$};

  \path[every node/.style={font=\sffamily\small}]
    (1) edge node[above] {$i_*$} (2)
    (2) edge node[above] {$l_*$} (3)
    (3) edge[bend left=20] node {$d$} (1);
 \end{tikzpicture}
 \end{equation}
\end{Prop}

\subsubsection{Heegaard diagrams} \label{Subsection: Heegaard Diagrams}

If $(K, S, \vf)$ is an open book decomposition for $Y$, then $(K_+, S, \vf \circ \tau_L^{-1})$ is an open book decomposition for $Y_+$.
\begin{Def} \label{Definition: Basis of arcs compatible with a H curve} If $L \subset S$ is a 
nonseparating curve, a basis of arcs $\bm{a}=\{a_1,\ldots,a_{2g}\}$ for $S$ is compatible with $L$ if:
\begin{enumerate}
 \item $a_1$ intersects transversely $L$ in a single point;
 \item $a_i \cap L =\emptyset$ for $i >1$;
 \item $a_2$ is homotopic to $\gamma * L * \gamma^{-1}$, where $\gamma$ is a connected component of $a_1 \setminus L$.
\end{enumerate}
\end{Def}
\begin{figure} 
   \begin{center}
\includegraphics[scale = 1]{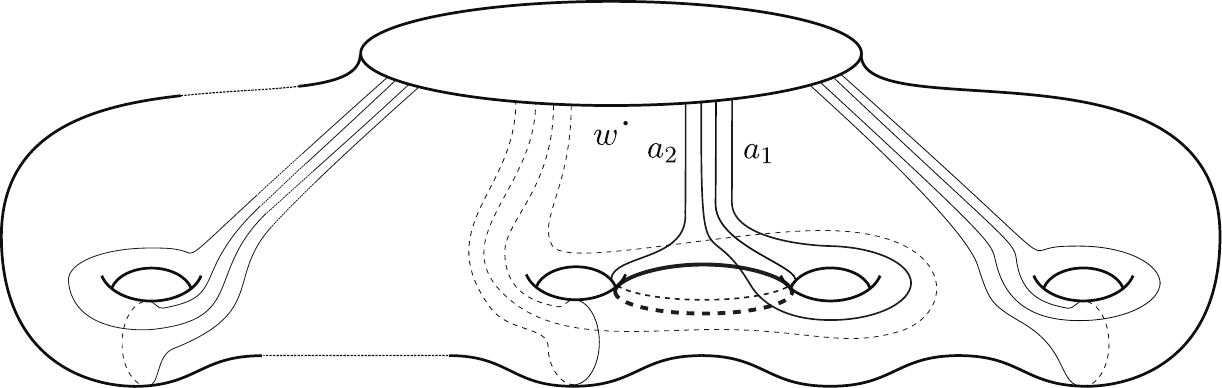}
\label{fig: compatible basis with L}
\caption{An example of basis of arcs compatible with $L$.}
\end{center}
\end{figure}
It is easy to check that any nonseparating curve $L \subset S$ admits a compatible basis of arcs for $S$. Fix $L$ and a compatible basis of arcs $\bm{a}$. 

Let $\tau_L$ be the composition of a Dehn twist along $L$ with support in a thin neighborhood $\mathcal{N}(L)$ of $L$ and of a small Hamiltonian diffeomorphism $\mathfrak{h}$ of $S$. Consider the set of arcs $\bm{a'} \coloneqq \{a'_1,\ldots,a'_{2g}\}$ where $a'_i = \tau^{-1}_L(a_i)$  for every $i$. Obviously $\bm{a'}$ is a basis of arcs for $S$ and, assuming that $\mathcal{N}(L)$ is thin enough, $a'_i$ is Hamiltonian isotopic to $a_i$ for $i>1$. We choose $\mathfrak{h}$ so that $\tau^{-1}_L|_{\partial S}$ is a small positive rotation and $a'_i \cap a_i$ consists of a single point for $i>1$ and in two points, contained in $\mathcal{N}(L)$, for $i=1$ (see Figure \ref{Figure: Curve schematizzate}). We define also a set of arcs $\widetilde{\bm{a}} \coloneqq \{\widetilde{a}_1,\ldots,\widetilde{a}_{2g}\}$ where $\widetilde{a}_1\coloneqq L$ and, for $i =2,\ldots,2g$, $\widetilde{a}_i$ is the result of a perturbation of $a'_i$ under a small Hamiltonian isotopy that rotates $\partial S$ by a small positive angle and such that $\widetilde{a}_i \cap a'_i$ consists of a single point $\Theta_{a'_i,\widetilde{a}_i}$. We also define $\Theta_{a'_1,\widetilde{a}_1}$ to be the only intersection point of $a'_1$ with $\widetilde{a}_1$.

\begin{figure} 
   \begin{center}
\includegraphics[scale = 1]{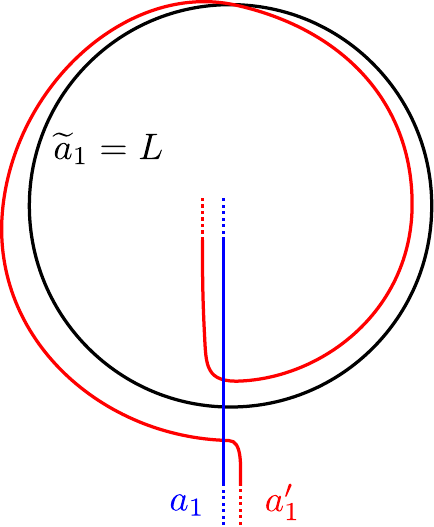}
\caption{An abstract representation of the curves $a_1$ and $a'_1$ near $\widetilde{a}_1=L$. In the rest of the section we will assume that the intersections among the curves are as in the picture.}
\label{Figure: Curve schematizzate}
\end{center}
\end{figure}

With slight abuse of notation we will call $\vf(\bm{a'})$ the set of curves image of $\bm{a'}$ under $\vf \circ \mathfrak{h}^{-1}$, so tha $a_i \cap \vf(a'_i) \cap \partial S$ consists of two points  for each $i \in \{1,\ldots,2g\}$,

We have then three diagrams $(S, \bm{a'}, \vf(\bm{a'}), z)$, $(S, \bm{a}, \vf(\bm{a'}), z)$ and $(S, \widetilde{\bm{a}}, \vf(\bm{a'}), z)$. The first diagram is diffeomorphic to  $(S, \bm{a}, \vf(\bm{a}), z)$, and therefore represents $(\overline{Y}, \overline{K})$. The second diagram is diffeomorphic to $(S, \bm{a}, (\vf \circ \tau_{L}^{-1})(\bm{a}), z)$, and therefore represents $(\overline{Y_+}, \overline{K_+})$.
The third diagram represents $(\overline{Y_0},\overline{K_0})$. 

\begin{Rmk} \label{Remark: The generators in HFK and their components in A}
 If $A$ is the neighborhood of $\partial S$ defined in Subsection \ref{Subsubsection: The knots filtrations}, we can assume that $L \cup \vf(L) \subset S \setminus A$. It follows that the only component of any generator $\mathbf{x}$ of $\widehat{CFK}(S, \widetilde{\bm{a}}, \vf(\bm{a'}), z, -g+1)$ that lies in $S \setminus A$ has to belong to $\widetilde{a}_1 =L$. Similarly, if $\mathbf{x}$ is a generator of $\widehat{CFK}(S, \widetilde{\bm{a}}, \vf(\widetilde{\bm{a}}), z, -g+1)$, the only component of $\mathbf{x}$ lying in $S \setminus A$ has to belong to $\widetilde{a}_1 \cap \vf(\widetilde{a}_1)$.
\end{Rmk}

\subsubsection{The chain map i}

 Let $(E_i,\pi_i)$ be the trivial symplectic fibration over
 \[B_i \coloneqq \D \setminus \{1,-1,i\}\]
 with fiber $S$. 
 \begin{figure} [ht!] 
  \begin{center}
   \includegraphics[scale = 1]{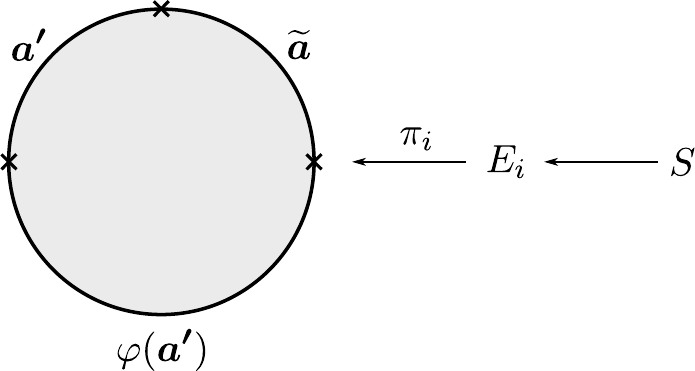}
  \end{center}
  \caption{The fibration $(E_{i},\pi_i)$ with its Lagrangian boundary condition.}
  \label{Figure: The fibration i}
\end{figure}
 Identify $\partial \D$ with $[0,2\pi]/(0 \sim 2\pi)$ in the standard way. Endow $(E_i,\pi_i)$ with the Lagrangian boundary condition $Q_i$ given by the symplectic parallel transport along $(0,\frac{\pi}{2})$, $(\frac{\pi}{2},\pi)$, $(\pi,2\pi)$ of copies of $\widetilde{\bm{a}}$, $\bm{a'}$ and $\vf(\bm{a'})$ respectively (cf. Figure \ref{Figure: The fibration i}).
 The chain map 
\[ i \colon \widehat{CFK}(S,\widetilde{\bm{a}},\vf(\bm{a'}), z, -g+1)  \longrightarrow  \widehat{CFK}(S,\bm{a'},\vf(\bm{a'}), z, -g+1) \]
is defined on the generators $\mathbf{x}$ by
 \[  i(\mathbf{x}) =  \sum_{\mathbf{y}}\#_2\mathcal{M}_0^i(\mathbf{x},\mathbf{y}) \mathbf{y} \]
 where the sum is taken over all the generators $\mathbf{y}$ of $\widehat{CF}(S,\bm{a}',\vf(\bm{a}')$, and $\mathcal{M}_0^i(\mathbf{x},\mathbf{y})$ is the moduli space of index $0$, embedded, degree $2g$ holomorphic multisections $u$ of $(E_{i},\pi_i)$ with Lagrangian boundary condition $Q_i$, asymptotic to $\mathbf{x}$ at $1$, to $\bm{\Theta}_{\bm{a}',\widetilde{\bm{a}}} = \{ \Theta_{a_1', \widetilde{a}_1'}, \ldots, \Theta_{a_{2g}', \widetilde{a}_{2g}'} \}$ at $i$ and to $\mathbf{y}$ at $-1$, and such that $\Imm(u) \cap (B_i \times \{z \}) = \emptyset$.

Observe that the triviality of the intersection with $B_i \times \{ z \}$ implies that $\mathcal{A}(\mathbf{y}) = -g+1$, and therefore $\mathbf{y}$ is a generator of $\widehat{CFK}(S,\bm{a'},\vf(\bm{a'}), z, -g+1)$, if $\mathcal{M}_0^i(\mathbf{x},\mathbf{y}) \neq \emptyset$.

\subsubsection{The chain map l}
 Let $(E_l,\pi_l)$ be a Lefschetz fibration over 
 \[B_l \coloneqq \D \setminus \{1,-1\}\]
 with fiber $S$ and only one critical point $x_0$ over $\mathfrak{z} = 0$ with vanishing cycle $L$.
 \begin{figure} [ht!] 
  \begin{center}
   \includegraphics[scale = 1]{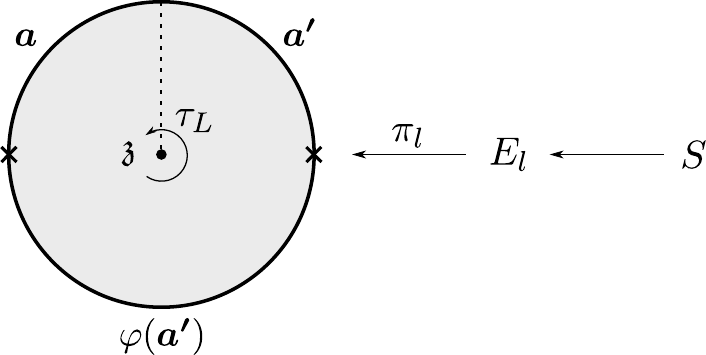}
  \end{center}
  \caption{The fibration $(E_{l},\pi_l)$ with its Lagrangian boundary condition. Crossing the dotted ray applies $\tau_L$ to the fibers.}
 \end{figure}
 We trivialize $\pi_l$ over the ray $\{Re(\zeta) = 0,Im(\zeta)>0\}$ so that we can assume that the monodromy $\tau_L$ around $\mathfrak{z}$ acts on the fibers when crossing that ray from the first standard quadrant of $\C$ to the second. As at the beginning of this section, $\tau_L$ denotes the composition of the positive Dehn twist with support in a thin neighborhood of $L$ with a small Hamiltonian perturbation that maps $a'_i$ to $a_i$ for every $i>1$.
 
 Endow $(E_l,\pi_l)$ with the Lagrangian boundary condition $Q_l$ given by parallel transport of a copy of $\vf(\bm{a}')$ along $(\pi,2\pi)$ and of a copy of $\bm{a'}$ along $(0,\pi)$. By our choice for the trivialization, $Q_l$ can be identified with a copy of $\bm{a'}$ inside each fiber over $(0,\frac{\pi}{2})$ and a copy of $\bm{a}$ inside each fiber over $(\frac{\pi}{2},\pi)$. 
 
 The chain map 
 \[ l \colon \widehat{CFK}(S,\bm{a'},\vf(\bm{a'}), z, -g+1) \longrightarrow  \widehat{CFK}(S,\bm{a},\vf(\bm{a'}), z, -g+1)\]
 is defined on the generators by
 \[ l(\mathbf{x})= \sum_{\mathbf{y}}\#_2\mathcal{M}_0^l(\mathbf{x},\mathbf{y})  \mathbf{y}, \]
 where the sum is taken over all the generators $\mathbf{y}$ of $\widehat{CF}(S,\bm{a},\vf(\bm{a'}))$ and $\mathcal{M}_0^l(\mathbf{x},\mathbf{y})$ is the moduli space of embedded index $0$ degree $2g$ holomorphic multisections $u$ of $(E_{l},\pi_l)$ with Lagrangian boundary condition $Q_l$, asymptotic to $\mathbf{x}$
at $1$ and to $\mathbf{y}$ at $-1$, and such that $\Imm(u) \cap (B_l \times \{z\}) = \emptyset$. The triviality of the intersection with $B_l \times \{ z \}$ implies that $\mathcal{A}(\mathbf{y}) = -g+1$, and therefore $\mathbf{y}$ is a generator of $\widehat{CFK}(S,\bm{a},\vf(\bm{a'}), z, -g+1)$, if $\mathcal{M}_0^i(\mathbf{x},\mathbf{y}) \neq \emptyset$.

\subsubsection{Sketch of the proof of the exact triangle}

The exactness of the triangle in \eqref{Triangle: Exact triangle in HF} can be proved using the same argument as in the proof of the exact surgery triangle in Heegaard Floer homology. Here we recast the proof in the language of symplectic fibrations.
\begin{Lemma} \label{Lemma: l circ i chain homotopic to 0}
There is a homotopy 
\[\mathcal{H}' \colon \widehat{CFK}(S,\widetilde{\bm{a}},\vf(\bm{a'}), z, -g+1)  \to \widehat{CFK}(S, \bm{a}, \vf(\bm{a}'), z, -g+1)\]
between $l \circ i$ and zero.
\end{Lemma}
\begin{proof}
Consider the path $\{\mathfrak{z}_t\}_{t \in (-1,1)}$ in $B_i$ defined by
\[\mathfrak{z}_t = \left\{\begin{array}{ll}
                           (t,0) & \mbox{for } t \in (-1,0]\\
                           (\frac{\sqrt{2}t}{2},\frac{\sqrt{2}t}{2}) & \mbox{for } t \in [0,1)
                          \end{array}
 \right.\]
and smoothed near $0$.
Let $(E_t^{\mathcal{H}'}, \pi_t^{\mathcal{H}'})$ be the Lefschetz fibration over $B_i$ with a unique critical point over $\mathfrak{z}_t$ with associated vanishing cycle $L$. The one-parameter family of punctured Riemann surfaces $B^{{\mathcal H}'}_t= (B_i, \mathfrak{z}_t)$ can be compactified by adding $(B_+, \mathfrak{z}_+)$, where $B_+= (B_i \setminus \{(\frac{\sqrt{2}}{2},\frac{\sqrt{2}}{2})\}) \sqcup (\mathbb{D} \setminus \{(-\frac{\sqrt{2}}{2},-\frac{\sqrt{2}}{2})\})\}$ and $\mathfrak{z}_+ = (0,0) \in int(\mathbb{D} \setminus \{(-\frac{\sqrt{2}}{2},-\frac{\sqrt{2}}{2})\})$, as $t \to + 1$ and $(B_-, \mathfrak{z}_-)$, where $B_-= B_l \sqcup B_i$ and $\mathfrak{z}_- = (0,0) \in int(B_l)$, as $t \to -1$.

We endow these Lefschetz fibrations with Lagrangian boundary conditions $Q_t^{\mathcal{H}'}$ induced by the one represented in the picture in the middle of Figure \ref{Figure: Degenerazione di H primo} (we assume that the the fibrations are trivialised in the complement of the dotted half-line so that  $\tau_L$ acts when one crosses it.) 
\begin{figure} [ht!] 
  \begin{center}
   \includegraphics[scale = 1]{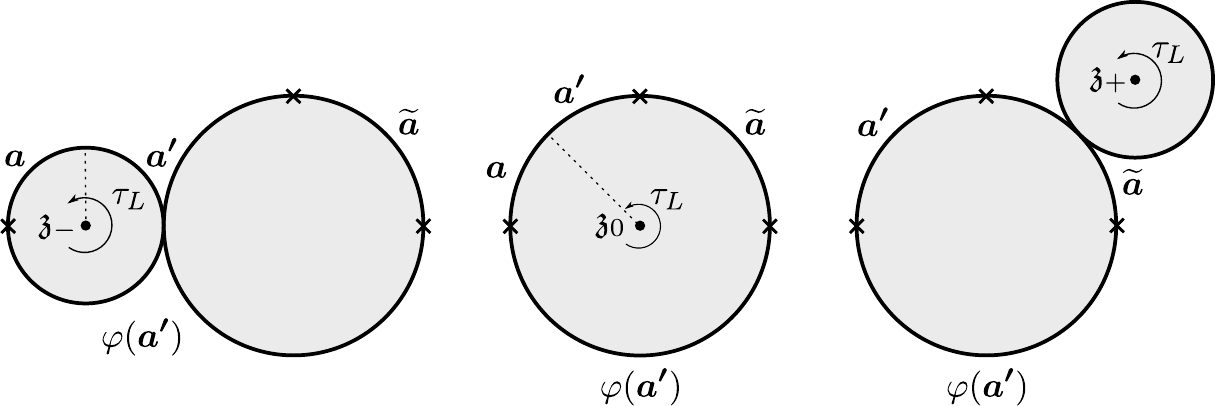}
  \end{center}
  \caption{The homotopy from $B^{\mathcal{H}'}_{-1}$ (left) to $B^{\mathcal{H}'}_{-1}$ (right), with $B^{\mathcal{H}'}_{0}$ pictured in the middle. Again the crosses denote punctures and the marked points the critical values of the Lefschetz fibrations. Each boundary component is labeled with the set of curves in the corresponding Lagrangian boundary condition $Q^{\mathcal{H}'}_t$.}
 \label{Figure: Degenerazione di H primo}
 \end{figure}
Then we define $\mathcal{H}'$ by counting pairs $(t, u)$ where $t \in (-1,1)$ and $u$ is a 
pseudoholomorphic multisection of degree $2g$ of $(E_t^{\mathcal{H}'}, \pi_t^{\mathcal{H}'})$ of index $-1$ with boundary on $Q_t^{\mathcal{H}'}$.

The degenerations of $(B_i, \mathfrak{z}_t)$ correspond, by pull back, to degenerations $(E_\pm^{\mathcal{H}'}, \pi_\pm^{\mathcal{H}'})$ of $(E_t^{\mathcal{H}'}, \pi_t^{\mathcal{H}'})$. 
The relative Gromov invariant of $(E_-^{\mathcal{H}'}, \pi_-^{\mathcal{H}'})$ is $l \circ i$, while the relative Gromov invariant of $(E_+^{\mathcal{H}'}, \pi_+^{\mathcal{H}'})$ is trivial by Example \ref{Example: G of tau L with condition L is 0}.
\end{proof}

 \begin{proof}[Proof of Proposition \ref{Prop: Exact triangle in HFK}]
 Since the end of the proof goes pretty much as in the proofs of Theorem \ref{Theorem: Exact triangle in SH} and of the exact triangle in Lagrangian Floer homology (\cite[Section 3]{Se5}), we will leave some details to the reader. The key point is to study the small energy components of the maps $i$ and $l$.

 If $a'_1$ is close enough to $a_1 \cup \widetilde{a}_1$ and, for $i>1$, $a_i'$ and $a_i$ are close enough to $\widetilde{a}_i$, for any $j \in \{1,\ldots,2g\}$ we have evident bijections
 \begin{equation}
 \begin{array}{ccccc}
  (a'_1 \cap \vf(a'_j)) & \stackrel{f}{\longrightarrow} & (\widetilde{a}_1 \cap \vf(a'_j)) & \sqcup & (a_1 \cap \vf(a'_j)) \\
       x_1'             & \longmapsto                  & \widetilde{x}_1 & \mbox{or} & x_1                       
 \end{array}
 \end{equation}
 (where $\widetilde{x}_1 \in (\widetilde{a}_1 \cap \vf(a'_j))$ and $x_1 \in (a_1 \cap \vf(a'_j))$) and
 \begin{equation}
 \begin{array}{ccccc}
  (\widetilde{a}_i \cap \vf(a'_j)) & \longrightarrow & (a'_i \cap \vf(a'_j)) & \longrightarrow & (a_i \cap \vf(a'_j))\\
     \widetilde{x}_i               & \longmapsto     & x'_i                  &   \longmapsto   &   x_i
 \end{array}
 \end{equation}
 where the image of any point is the closest among all the elements of the corresponding codomain. These induce an injection
 \begin{equation*}  \begin{array}{ccc}
   i_0 \colon  \widehat{CFK}(S,\widetilde{\bm{a}},\vf(\bm{a'}),z,-g+1) & \hooklongrightarrow & \widehat{CFK}(S,\bm{a'},\vf(\bm{a'}),z, -g+1)\\
                (\widetilde{x}_1,\widetilde{x}_2,\ldots,\widetilde{x}_{2g})  & \longmapsto   &  (x'_1,x'_2,\ldots,x'_{2g}). 
  \end{array}
 \end{equation*}
 and a quotient 
 \begin{equation*}
 \begin{array}{ccc}
  l_0 \colon  \widehat{CFK}(S,\bm{a'},\vf(\bm{a'});-g+1) & \xtwoheadrightarrow{} & \widehat{CFK}(S,\bm{a},\vf(\bm{a'});-g+1)\\
                (x'_1,x'_2,\ldots,x'_{2g})               & \longmapsto           &\left\{\begin{array}{cl}
                                                                                           (x_1,x_2,\ldots,x_{2g}) & \mbox{if } f(x'_1) \in (a_1 \cap \vf(a'_j))\\
                                                                                           0 & \mbox{if } f(x'_1) \in (\widetilde{a}_1 \cap \vf(a'_j)).
                                                                                          \end{array}\right.

 \end{array}
 \end{equation*}
 Reasoning as in \cite[Subsection 3.2]{Se5}, one can check that $i_0$ can be expressed by counting holomorphic degree $2g$ multisections of $(E_i,\pi_i,Q_i)$ and that, taking $\widetilde{a}_i$ close enough to $a'_i$ outside a small neighborhood of $a_1$, the energy of these multisections can be made arbitrarily small. We have then a decomposition $i= i_0 + i'$ where $i'$ counts higher energy sections.
 
 Similarly, if $a_i$ is close enough to $a_i'$ outside a small neighborhood of $\widetilde{a}_1$, $l_0$ coincides with the lower energy component of $l$, giving a decomposition $l=l_0+l'$ where $l'$ counts higher energy sections (cf. Section 3.3 of \cite{Se5} and, in particular, Lemma 3.8 for a description of the lower energy sections that appear in the definition of $l$).
 Since $l_0 \circ i_0 =0$, again Lemma 2.31 of \cite{Se5} and Lemma \ref{Lemma: Algebraic lemma for the exactness} above imply Proposition \ref{Prop: Exact triangle in HFK}.
\end{proof}

\subsection{Comparing the double cones}

Applying the main result of Section \ref{Section: The isomo from HF to ECH} we obtain a chain map $\Phi^\sharp $ that allows us to compare two of the three terms of the two exact triangles in Heegaard Floer and symplectic homologies as in Diagram \eqref{Diagram: Main diagram between HF and HF without maps}. To proceed with our strategy for the proof of Theorem \ref{Theorem: Main theorem in introduction}, we first need to define a chain map inducing the isomorphism on the first column of \eqref{Diagram: Main diagram between HF and HF without maps} that behaves well with respect to the Lefschetz fibrations framework. The main difficulty is that we defined the chain maps inducing the exact sequence 
\eqref{Triangle: Exact triangle in HF} by counting degree $2g$ holomorphic multi-sections, and those inducing the exact sequence \eqref{Triangle: L phi phi-tau} by counting holomorphic sections. As a first stem, me need to put both exact sequences on equal footing.

\subsubsection{Seidel's exact sequence revisited}
Let us consider the diagram $(S,\widetilde{\bm{a}},\vf(\widetilde{\bm{a}}),z)$. 
Although it is not a diagram of the type considered in Section \ref{Subsection: Heegaard Floer homology for open books}, we can still define chain complex $\widehat{CFK}(S, \widetilde{\bm{a}}, \vf(\widetilde{\bm{a}}), z, 1-g)$ in the same way. Note, however, that we should be careful to choose for the definition a Liouville form on $S$ for which $L$ is an exact Lagrangian submanifold; this is always possible, as long as $L$ is nonseparating.
 
If $A$ is as in Section \ref{Subsubsection: The knots filtrations} then $(\widetilde{a}_2 \cap \vf(\widetilde{a}_2)) \cap A = \{c_2,c_2',d\}$ where $c_2,c_2' \subset \partial S$. Up to moving the base point $w$ and, possibly, changing the compatible basis of arcs, we can assume that
\begin{equation} \label{Condition: a_2 is the first at the left}
\begin{split}
 &a_2 \cap A \mbox{ is the first arc that one encounters when moving from }w \mbox{ in the} \\ &\mbox{direction of } \partial S.
\end{split}
\end{equation}
This assumption  implies that $\widetilde{a}_2 \cap A$ intersects $\vf(\widetilde{a}_i)$ if and only if $i=2$. Moreover, if $\mathbf{x}=(x_1, \ldots, x_{2g}) \in \widehat{CFK}(S,\widetilde{\bm{a}},\vf(\widetilde{\bm{a}}),z,-g+1)$, then $x_1 \in \widetilde{a}_1 \cap \vf(\widetilde{a}_1)$ because only one intersection point is outside of $A$. Thus we have a splitting of vector spaces
\begin{equation} \label{Equation: Split of CFK in terms of C and overline C}
 \widehat{CFK}(S,\widetilde{\bm{a}},\vf(\widetilde{\bm{a}}),z,-g+1) = \widehat{CFK}_-(S,\widetilde{\bm{a}},\vf(\widetilde{\bm{a}}),z,-g+1) \oplus \widehat{CFK}_+(S,\widetilde{\bm{a}},\vf(\widetilde{\bm{a}}),z,-g+1) 
\end{equation}
where $\widehat{CFK}_-(S,\widetilde{\bm{a}},\vf(\widetilde{\bm{a}}),z,-g+1)$ is the subspace generated by the $2g$-tuples of intersection points $(x_1,x_2\ldots,x_{2g})$ with  $x_2 \in \{ c_2, c_2' \}$  and $\widehat{CFK}_+(S,\widetilde{\bm{a}},\vf(\widetilde{\bm{a}}),z,-g+1)$ is the subspace generated by the $2g$-tuples of intersection points $(x_1,x_2\ldots,x_{2g})$ with $x_2 = d$, both with the usual identifications $c_i \sim c_i'$. 

\begin{Lemma} \label{Lemma: The subcomplexes C and overline C}
 The following hold:
 \begin{enumerate}
  \item  $\widehat{CFK}_-(S,\widetilde{\bm{a}},\vf(\widetilde{\bm{a}}),z,-g+1)$ and  $\widehat{CFK}_+(S,\widetilde{\bm{a}},\vf(\widetilde{\bm{a}}),z,-g+1)$ are subcomplexes of $\widehat{CFK}(S,\widetilde{\bm{a}},\vf(\widetilde{\bm{a}}),z,-g+1)$, and
  \item the chain maps
  \begin{equation*}
   j_- \colon CF(\vf(\widetilde{a}_1),\widetilde{a}_1) \to \widehat{CFK}_-(S,\widetilde{\bm{a}},\vf(\widetilde{\bm{a}}),z,-g+1), \quad j_-(x)=(x, c_2, c_3, \ldots, c_{2g}) 
  \end{equation*}
   and
  \begin{equation*}
   j_+ \colon CF(\vf(\widetilde{a}_1),\widetilde{a}_1) \to \widehat{CFK}_+(S,\widetilde{\bm{a}},\vf(\widetilde{\bm{a}}),z,-g+1), \quad j_+(x)=(x, d, c_3, \ldots, c_{2g})
  \end{equation*}
  induce isomorphisms in homology.
 
   \end{enumerate}
 \end{Lemma}
 \begin{proof}
 The fact that  $\widehat{CFK}_-(S,\widetilde{\bm{a}},\vf(\widetilde{\bm{a}}),z,-g+1)$ is a subcomplex of $\widehat{CFK}(S,\widetilde{\bm{a}},\vf(\widetilde{\bm{a}}), \newline z,-g+1)$ is a direct consequence of the fact that if $u$ is an irreducible component of a holomorphic curve counted in the definition of the Heegaard Floer differential that has a positive end at some $c_i$ or $c'_i$ then $u$ is a trivial strip. To see that  $\widehat{CFK}_+(S,\widetilde{\bm{a}},\vf(\widetilde{\bm{a}}),z,-g+1)$ is also a subcomplex, we observe first that if $\mathbf{x}=(x_1,d,x_3,\ldots,x_{2g})$ is a generator of $\widehat{CFK}_+(S,\widetilde{\bm{a}},\vf(\widetilde{\bm{a}}),z,-g+1)$, then there are two index $1$ holomorphic curves in $\R \times [0,1] \times S$ with
  positive ends at $(\widetilde{x}_1,d)$ and negative ends at $(x_1,c_2)$ or $(x_1,c'_2)$. These two curves project over $S$ to the two shaded annuli in Figure \ref{Figure: The two disks for the tilde curves} (cf. \cite[Lemma 9.4]{OS1} for a description of similar holomorphic curves). Moreover it is not difficult to see from Condition \eqref{Condition: a_2 is the first at the left} that these are the only two holomorphic curves that appear in the
expression for $\partial\mathbf{x}$ that have negative limit not contained in  $\widehat{CFK}_+(S,\widetilde{\bm{a}},\vf(\widetilde{\bm{a}}),z,-g+1)$.
  The identification $c_2 \sim c'_2$ and the fact that we work with $\Z / 2$--coefficients imply then that $\partial \mathbf{x} \in \widehat{CFK}_+(S,\widetilde{\bm{a}},\vf(\widetilde{\bm{a}}),z,-g+1)$. 
  
  \begin{figure} [ht!] 
  \begin{center}
   \includegraphics[scale = 1]{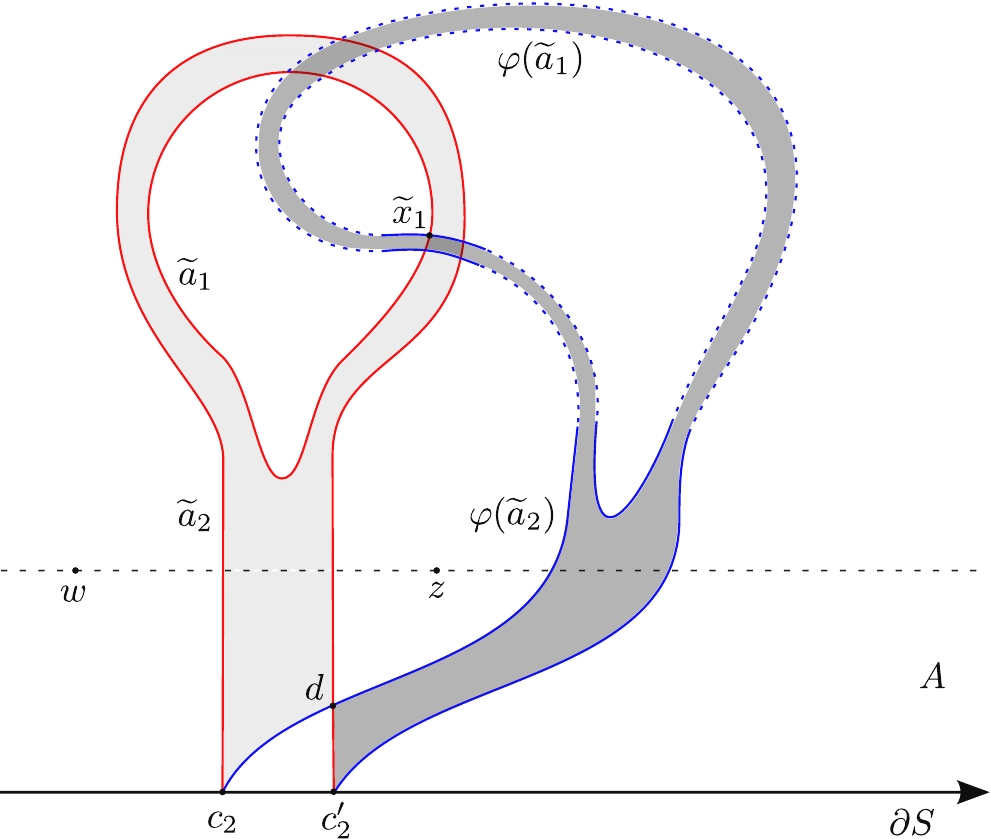}
  \end{center}
  \caption{The projections to $S$ of the two index $1$ holomorphic curves from $(x_1,d)$ to $(x_1,c_2)$ or $(x_1,c'_2)$. For simplicity we avoided to draw the curves $\widetilde{a}_i$ and $\vf(\widetilde{a}_i)$ for $i>2$.}
 \label{Figure: The two disks for the tilde curves}
 \end{figure}
  
A part from the pair of canceling holomorphic curves described above, the projection to $S$ of the other holomorphic curves that appear in the definition of the differential $\partial$ of $\widehat{CFK}(S,\widetilde{\bm{a}},\vf(\widetilde{\bm{a}}),z,-g+1)$ do not cross $\partial A$. Then we have then a splitting of the differential of  $\widehat{CFK}(S,\widetilde{\bm{a}},\vf(\widetilde{\bm{a}}),z,-g+1)$ as $\partial = \partial_0 + \partial_1$, where 
  \begin{itemize}
  \item $\partial_0$ is defined counting pseudo-holomorphic multisections of the form $u_0 \sqcup u'_0$ where $u_0$ is an index $1$ pseudo-holomorphic section counted in the definition of the differential of $CF(\vf(\widetilde{a}_1),\widetilde{a}_1)$ and $u'_0$ is a $2g-1$-tuple of trivial sections over the components of the generators contained in $A$;
  \item $\partial_1$ is defined counting pseudo-holomorphic multisections of the form $u_1 \sqcup u'_1$ where $u_1$ is an index $1$ degree $2g-1$ pseudo-holomorphic multisection that projects to $A$ and $u'_1$ is a trivial section over a point in $\widetilde{a}_1 \cap \vf(\widetilde{a}_1)$. 
 \end{itemize}
  
Since the point in $\widetilde{a}_2 \cap \vf(\widetilde{a}_2)$ cannot indteract with anything else (besides the two cancelling curves described above), the splitting of the differential gives isomorphisms of chain complexes
\begin{align*}
& \widehat{CFK}_\pm(S,\widetilde{\bm{a}},\vf(\widetilde{\bm{a}}),z,-g+1) \cong \\
& CF(\vf(L), L) \otimes \widehat{CFK}(S, \{\widetilde{a}_3, \ldots, \widetilde{a}_{2g} \},  \{\vf(\widetilde{a}_3), \ldots, \vf(\widetilde{a}_{2g}) \}, z, 1-g).
\end{align*}

Now, observe that the $\widehat{CFK}(S, \{\widetilde{a}_3, \ldots, \widetilde{a}_{2g} \},  \{\vf(\widetilde{a}_3), \ldots, \vf(\widetilde{a}_{2g}) \}, z, 1-g)$ is isomorphic to the knot Floer complex of a fibred knot in the bottom Alexander degree, and therefore its homology is generated by the class of $(c_3, \ldots, c_{2g})$. Then the chain maps $j_\pm$ induce isomorphisms in homology by  K\"unneth's formula.

\end{proof}

We define a map 
$$\widetilde{\iota} \colon \widehat{CFK}(S, \widetilde{\bm{a}}, \vf(\widetilde{\bm{a}}), z, 1-g) \to CF^\sharp(\vf)$$
by counting embedded holomorphic multisections of $E_\iota^\vf \to B_\iota$ with boundary on the Lagrangian boundary condition $Q_{\widetilde{\bm{a}}}$ obtained by parallel transport of $\widetilde{\bm{a}}$ and which are asymptotic to a generator $\mathbf{x}$ of $\widehat{CFK}(S, \widetilde{\bm{a}}, \vf(\widetilde{\bm{a}}), z, 1-g)$ at the boundary puncture and to the multiorbit $\gamma e^{2g-1}$ for a generator $\gamma$ of $CF^\sharp(\vf)$
at the interior puncture. The reason why $\widetilde{\iota}$ is a chain map is that it is defined in essentially the same way of $\Phi^{\sharp}$, and therefore the arguments given in Section \ref{Subsection: Phi is filtered} apply also here. Let
$$\widetilde{\iota}_\pm \colon \widehat{CFK}_\pm(S, \widetilde{\bm{a}}, \vf(\widetilde{\bm{a}}), z, 1-g) \to CF^\sharp(\vf)$$
be the restrictions of $\widetilde{\iota}$. 
\begin{Lemma}\label{tilde or not tilde: iota}
The diagram 
\begin{equation} 
  \begin{tikzpicture}[->,auto,node distance=1.5cm,>=latex',baseline=(current  bounding  box.center)]
  \matrix (m) [matrix of nodes, row sep=1.5cm,column sep=1cm] 
  {$\widehat{CFK}_-(S, \widetilde{\bm{a}}, \vf(\widetilde{\bm{a}}), z, 1-g)$ & $CF^{\sharp}(\vf)$\\
  $CF(\vf(L),L)$ & \\};
  \path[->]
   (m-1-1) edge node[above] {$\widetilde{\iota}_-$} (m-1-2);
  \path[right hook-latex]
   (m-2-1) edge node[left] {$j_-$} (m-1-1);
  \path[->]
   (m-2-1) edge node[below,yshift=1ex,xshift=5ex]{$\iota$} (m-1-2);
  \end{tikzpicture}
 \end{equation}
commutes.
\end{Lemma}
\begin{proof}
As in Lemma 6.2.3 of \cite{CGH3}, if $u$ is an irreducible component of a multisection of $(E_{\iota}^{\vf},\pi_{\iota})$ with positive end to some $c_i$, then $u$ is a holomorphic section with negative end at $e$.
\end{proof}
We define also a map $\widetilde{\lambda} \colon CF^\sharp(\vf) \to CF^\sharp(\vf \circ \tau_L^{-1})$ by counting holomorphic multisections of degree $2g$ of $E_\lambda^\vf \to B_\lambda$ which are asymptotic to $\gamma_+e^{2g-1}$ and $\gamma_-e^{2g-1}$ for 
$\gamma_+$ and $\gamma_-$ generators of $CF^\sharp(\vf)$ and $CF^\sharp(\vf \circ \tau_L^{-1})$ respectively.
\begin{Lemma} \label{tilde or not tilde: lambda} $\widetilde{\lambda}=\lambda$.
\end{Lemma}
\begin{proof}
By lemma 5.3.2 of \cite{CGH1}, the holomorphic curves used to define $\widetilde{\lambda}$  consist of the union of a holomorphic section of $E_\lambda^\vf \to B_\lambda$ between $\gamma_+$ and $\gamma_-$ with $2g-1$ copies of the trivial cylinder over $e$.
\end{proof}
With obvious modifications to Lemma \ref{lambda composed iota is trivial} one can define a chain homotopy 
$$\widetilde{\mathcal H} \colon \widehat{CFK}(S, \widetilde{\bm{a}}, \vf(\widetilde{\bm{a}}), z, 1-g) \to CF^\sharp(\vf \circ \tau_L^{-1})$$ 
between $\widetilde{\lambda} \circ \widetilde{\iota}$ and the zero map. Let $\widetilde{\mathcal H}_\pm$ denote the restriction of $\widetilde{\mathcal H}$ to the subcomplexes  $\widehat{CFK}_{\pm}(S, \widetilde{\bm{a}}, \vf(\widetilde{\bm{a}}), z, 1-g)$. The proof of the next lemma is the same as the proof of Lemma \ref{tilde or not tilde: iota}.
\begin{Lemma}\label{tilde or not tilde: H}
The diagram
\begin{equation} 
  \begin{tikzpicture}[->,auto,node distance=1.5cm,>=latex',baseline=(current  bounding  box.center)]
  \matrix (m) [matrix of nodes, row sep=1.5cm,column sep=1cm] 
  {$\widehat{CFK}_-(S, \widetilde{\bm{a}}, \vf(\widetilde{\bm{a}}), z, 1-g)$ & $CF^{\sharp}(\vf\circ \tau_L^{-1})$\\
  $CF(\vf(L),L)$ & \\};
  \path[->]
   (m-1-1) edge node[above] {$\widetilde{\mathcal H}_-$} (m-1-2);
  \path[right hook-latex]
   (m-2-1) edge node[left] {$j_-$} (m-1-1);
  \path[->]
   (m-2-1) edge node[below,yshift=1ex,xshift=5ex]{${\mathcal H}$} (m-1-2);
  \end{tikzpicture}
 \end{equation}
commutes.
\end{Lemma}
We can then form a double cone of the maps $\widetilde{\iota}_-$ and $\widetilde{\lambda}$, which is 
$$\widehat{CFK}_-(S, \widetilde{\bm{a}}, \vf(\widetilde{\bm{a}}), z, 1-g) \oplus CF^\sharp(\vf) \oplus CF^\sharp(\vf \circ \tau_L^{-1})$$
with differental
$$\left ( \begin{matrix}
\partial & 0 & 0 \\ \widetilde{\iota}_- & \partial & 0 \\
\widetilde{\mathcal H}_- & \widetilde{\lambda} & \partial 
\end{matrix} \right ).$$
\begin{Lemma}
The double cone of the maps $\widetilde{\iota}_-$ and $\widetilde{\lambda}$ is acyclic.
\end{Lemma}
\begin{proof}
By Lemmas \ref{tilde or not tilde: iota}, \ref{tilde or not tilde: lambda} and \ref{tilde or not tilde: H} the double cone of $\iota$ and $\lambda$ defined in Section \ref{Subsection: The exact triangle in symplectic homology} is a subcomplex of the double cone of $\widetilde{\iota}_-$ and $\widetilde{\lambda}$. Double cones are naurally filtered complexes because the differential is a lower triangular matrix. Moreover the inclusion induces an isomorphisms on the homology of the associated graded complexes by Lemma \ref{Lemma: The subcomplexes C and overline C}. This implies that the inclusion induces an 
isomorphisms between the homologies of the double cones by a standard algebraic trick. The double cone of $\iota$ and $\lambda$ is acyclic, and therefore the double cone of $\widetilde{\iota}_-$ and $\widetilde{\lambda}$ is also acyclic.
\end{proof}
\subsubsection{The first square}
Let now $(E_{\Upsilon},\pi_{\Upsilon})$ be the trivial symplectic fibration with basis $B_{\Upsilon} = \D \setminus \{1,-i,-1\}$ and fiber $S$. Endow $(E_{\Upsilon},\pi_{\Upsilon})$ with the Lagrangian boundary condition $Q_{\Upsilon}$ given by the symplectic parallel transport over $(0,\pi)$, $(\pi,\frac{3\pi}{2})$ and, respectively, $(\frac{3\pi}{2},2\pi)$ of copies of $\widetilde{\bm{a}}$, $\vf(\widetilde{\bm{a}})$ and, respectively, $\vf(\bm{a'})$.
 \begin{figure} [ht!] 
  \begin{center}
   \includegraphics[scale = 1]{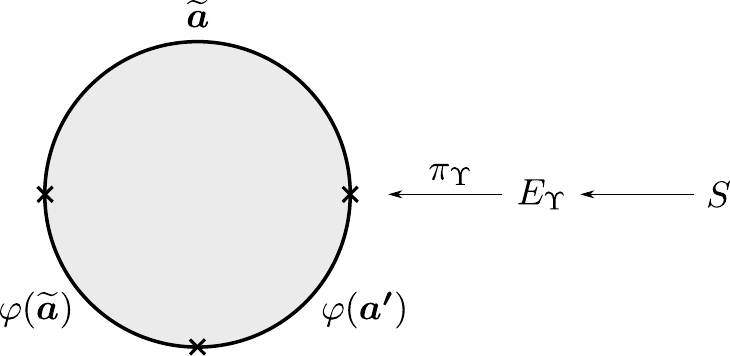}
  \end{center}
  \caption{The fibration $(E_{\Upsilon},\pi_{\Upsilon})$ with its Lagrangian boundary condition.}
  \label{Figure: The fibration Xi}
\end{figure}

Let
\[\Upsilon \colon \widehat{CFK}(S,\widetilde{\bm{a}},\vf(\bm{a'}),z,-g+1) \longrightarrow  \widehat{CFK}(S,\widetilde{\bm{a}},\vf(\widetilde{\bm{a}}),z,-g+1)\]
be the chain map defined on the generators $\mathbf{x}$ by
 \[  \Upsilon(\mathbf{x}) =  \sum_{\mathbf{y}}\#_2\mathcal{M}_0^{\Upsilon}(\mathbf{x},\mathbf{y}) \mathbf{y} \]
 where the sum is taken over all the generators $\mathbf{y}$ of $\widehat{CFK}(S,\widetilde{\bm{a}},\vf(\widetilde{\bm{a}}),z, -g+1)$ and $\mathcal{M}_0^{\Upsilon}(\mathbf{x},\mathbf{y})$ is the moduli space of index $0$, embedded degree $2g$ holomorphic multisections $u$ of $(E_{\Upsilon},\pi_{\Upsilon})$ with Lagrangian boundary condition $Q_{\Upsilon}$, asymptotic to $\mathbf{x}$ at $1$, to $\vf(\bm{\Theta}_{a',\widetilde{a}})$ at $-i$ and to $\mathbf{y}$ at $-1$ and such that $\Imm(u) \cap (B_{\Upsilon} \times \{z \}) = \emptyset$.
We denote by $\Upsilon_\pm \colon  \widehat{CFK}(S, \widetilde{\bm{a}}, \vf(\bm{a}'), z, -g+1) \to \widehat{CFK}_\pm (S, \widetilde{\bm{a}}, \vf(\widetilde{\bm{a}}) ,z, -g+1)$ the components of $\Upsilon$.

\begin{Lemma} \label{The map Upsilon}
In homology $\Upsilon_-$ induces an isomorphism and $\Upsilon_+$ the trivia map.
\end{Lemma}
\begin{proof}
Let $\overline{\Sigma} = S \cup_{\partial} \overline{S}$ be the double of $S$. We complete $\widetilde{\bm{a}}$, $\vf(\bm{a}')$ and $\vf(\widetilde{\bm{a}})$ in $S$ to collections of curves $\widetilde{\bm{\alpha}}$, $\bm{\beta}'$ and $\widetilde{\bm{\beta}}$ as explained in Subsection \ref{Subsection: Heegaard Floer homology for open books}, with the caveat that $\widetilde{\alpha}_1= \widetilde{a}_1=L$ and $\widetilde{\beta}_1= \vf(\widetilde{a}_1)$, and only the arcs $\widetilde{\alpha}_i$ and $\vf(\widetilde{a}_i)$
for $i=2, \ldots, 2g$ are completed in $\overline{S}$.

$(\overline{\Sigma}, \widetilde{\bm{\alpha}}, \bm{\beta}')$ is a Heegaard diagram for 
$\overline{Y}_0$: in fact we can slide $\widetilde{\alpha}_2$ over $\widetilde{\alpha}_1$ and isotope the resulting curve so that it intersects $\beta_1'$ in only one point, and intersects no other $\beta'$-curve. Then we can destabilise the diagram and remove those two curves.
In the resulting diagram $\widetilde{\beta}_2'$ becomes isotopic to $\vf(L)$. One can  check that this is a Heegaard diagram for $Y_0$ compatible with a broken fibration over $S^1$ with one critical point and vanishing cycle $L$ (cf. Lekili \cite{Lek} for a similar construction).

$(\overline{\Sigma}, \widetilde{\bm{\alpha}}, \widetilde{\bm{\beta}})$ is a Heegaard diagram for $\overline{Y}_0 \# (S^2 \times S^1)$, where the connected sum is performed away from the knot $\overline{K}_0$. In fact one can handleslide $\widetilde{\alpha}_2$ over $\widetilde{\alpha}_1$ and $\widetilde{\beta}_2$ over $\widetilde{\beta}_1$ and, by condition (3) of Definition \ref{Definition: Basis of arcs compatible with a H curve}, the resulting curves are both isotopic to copies of $L$ in $\overline{S}$. They give a copy of $S^2 \times S^1$ and, after removing it, we remain with the same Heegaard diagram we obtained above after the stabilisation. Finally, $(\overline{\Sigma}, \bm{\beta}', \widetilde{\bm{\beta}})$ is a Heegaard diagram for $(S^2 \times S^1)^{\# (2g-1)}$.

Then $(\overline{\Sigma}, \widetilde{\bm{\alpha}}, \bm{\beta}', \widetilde{\bm{\beta}})$ is the triple Heegaard diagram for a cobordism from $\overline{Y}_0$ to $\overline{Y}_0 \# (S^2 \times S^1)$ obtained by a single $2$-handle addition because $\beta_i'$ is isotopic to $\widetilde{\beta}_i$ for $i=2, \ldots, 2g$. Moreover, the knot to which the handle is attached is an unknot which is geometrically unlinked with $\overline{K}_0$. In fact, the longitude is $\widetilde{\beta}_1$, which, in the Heegaard diagram for $\overline{Y}_0$ obtained after destabilising $(\overline{\Sigma}, \widetilde{\bm{\alpha}}, \bm{\beta}')$, is isotopic to $\beta_2'$, which bounds in $\overline{Y}_0$ a disc disjoint from $\overline{K}_0$.

The same arguments which prove Theorem \ref{Theorem: HF on a page} also show that the map $\Upsilon$ coincides with the map $F \colon \widehat{HFK}(\overline{Y}_0, \overline{K}_0, 1-g) \to  \widehat{HFK}(\overline{Y}_0 \# (S^2 \times S^1), \overline{K}_0, 1-g) \cong \widehat{HFK}(\overline{Y}_0, \overline{K}_0, 1-g) \otimes \widehat{HF}(S^2 \times S^1)$ defined by the doubly pointed triple Heegaard diagram $(\overline{\Sigma}, \widetilde{\bm{\alpha}}, \bm{\beta}', \overline{\bm{\beta}}, w,z)$. The group $\widehat{HF}(S^2 \times S^1)$ is generated by homogeneous elements $\theta_{top}$ and
$\theta_{bot}$, and and a direct computation (on a simpler triple Heegaard diagram for the same cobordism, using the invariance of triangle maps) shows that $F(\mathbf{x})=\mathbf{x} \otimes \theta_{bot}$.

To prove the lemma it remains only to identify $\widehat{HFK}_-(S, \widetilde{\bm{a}}, \vf(\widetilde{\bm{a}}), z, 1-g)$ with $\widehat{HFK}(\overline{Y}_0, \overline{K}_0, 1-g)
\otimes \langle \theta_{bot} \rangle$. For that we use Heegaard Floer homology with twisted coefficients. We recall that $\widehat{\underline{HF}}(S^2 \times S^1) \cong \Z /2\Z$, the universal coefficient map $\widehat{\underline{HF}}(S^2 \times S^1) \to \widehat{HF}(S^2 \times S^1)$ is injective and its image is $\langle \theta_{bot} \rangle$. Then, if $\Z/2\Z[t, t^{-1}] \cong \Z/2\Z[H^1(S^2 \times S^1)]$ is endowed with the $H^1(\overline{Y}_0 \# (S^2 \times S^1))$--module structure induced by the map $H^1(\overline{Y}_0 \# (S^2 \times S^1)) \to H^1(S^2 \times S^1)$, K\"unneth's formula gives that
$$\underline{\widehat{HFK}}(\overline{Y}_0 \# (S^2 \times S^1)), \overline{K}_0, 1-g; \Z/2\Z[t, t^{-1}] ) \cong \widehat{HFK}(\overline{Y}_0, \overline{K}_0, 1-g) \otimes \underline{HF}(S^2 \times S^1)$$
and the image of the universal coefficient map 
$$\underline{\widehat{HFK}}(\overline{Y}_0 \# (S^2 \times S^1)), \overline{K}_0, 1-g; \Z/2\Z[t, t^{-1}] ) \to \widehat{HFK}(\overline{Y}_0 \# (S^2 \times S^1)), \overline{K}_0, 1-g)$$
is $\widehat{HFK}(\overline{Y}_0, \overline{K}_0, 1-g) \otimes \langle \theta_{bot} \rangle$.
Now it is easy to see that for this twisted coefficient system, all holomorphic curves contributing to the differential of  $\underline{\widehat{CFK}}(S, \widetilde{\bm{a}},  \vf(\widetilde{\bm{a}}), z, 1-g; \Z/2\Z[t, t^-1])$ are counted with coefficient $1$, except for the two curves in Figure \ref{Figure: The two disks for the tilde curves}, one of which is counted with coefficient $1$ and the other with coefficient $t$. Then $\underline{\widehat{HFK}}(S, \widetilde{\bm{a}},  \vf(\widetilde{\bm{a}}), z, 1-g; \Z/2\Z[t, t^-1]) \cong HFK_-\widehat{HFK}(S, \widetilde{\bm{a}},  \vf(\widetilde{\bm{a}}), z, 1-g)$. This concludes the proof.
\end{proof}

\begin{Lemma}\label{Lemma: First square}
There exists a chain homotopy 
$$\mathcal{R} \colon \widehat{CFK}(S,\widetilde{\bm{a}},\vf(\bm{a'});-g+1) \longrightarrow CF^{\sharp}(\vf)$$ 
between $\widetilde{\iota} \circ \Upsilon$ and $\Phi^\sharp \circ i$ that makes the following diagram commute in homology:
\begin{equation} \label{Diagram: First main square}
\begin{tikzpicture}[->,auto,node distance=1.5cm,>=latex',baseline=(current  bounding  box.center)]
\matrix (m) [matrix of nodes, row sep=1cm,column sep=2cm]
{ $\widehat{CFK}(S,\widetilde{\bm{a}},\vf(\bm{a'}), z, -g+1)$ & $\widehat{CFK}(S,\bm{a'},\vf(\bm{a'}), z, -g+1)$  \\
  $\widehat{CFK}(S, \widetilde{\bm{a}},\vf(\widetilde{\bm{a}}), z, -g+1)$ & $CF^{\sharp}(\vf)$ \\ };
 \path[->]
  (m-1-1) edge node[right] {$\Upsilon$} (m-2-1)
          edge node[above] {$i$} (m-1-2)
          edge[dashed] node[above] {$\mathcal{R}$} (m-2-2);
 \path[->]
  (m-1-2) edge node[right] {$\Phi^\sharp$} (m-2-2);
 \path[->]
  (m-2-1) edge node[below] {$\widetilde{\iota}$} (m-2-2);
\end{tikzpicture}
\end{equation}
\end{Lemma}
\begin{proof}
Let $\{\mathfrak{w}_t\}_{t \in (-1,1)}$ be the path in  $\D \setminus \{1,i\}$ defined by
$\mathfrak{w}_t = t e^{i \pi/4}$
for $t \in (-1,1)$. Consider the one-parameter family of punctured Riemann surfaces $B_t^{\mathcal{R}} \coloneqq \D \setminus \{1, i, \mathfrak{w}_t\}$  and let $\pi_t^{\mathcal{R}} \colon E_t^{\mathcal{R}} \to B_t^{\mathcal{R}} $ be the symplectic fibration over $B_t^{\mathcal{R}}$ with fiber $\widehat{S}$ and monodromy $\widehat{\vf}$ around  $\mathfrak{w}_t$. We endow these  fibrations with Lagrangian boundary conditions $Q_t^{\mathcal{R}}$ induced by the one represented in the picture in the middle of Figure \ref{Figure: The homotopy R}.
The homotopy $\mathcal{R}$ is defined by counting pairs $(t,u)$ where $t \in (-1,1)$ and $u$ is a degree $2g$ multisection of $\pi_t^{\mathcal{R}} \colon E_t^{\mathcal{R}} \to B_t^{\mathcal{R}} $ of index $-1$ with boundary on $Q_t^{\mathcal{R}}$  which are asymptotic to 
 a generator of $\widehat{CFK}(S,\widetilde{\bm{a}},\vf(\bm{a'});-g+1)$ at $1$, to a generator of $CF^\sharp(\vf)$ at $\mathfrak{w}_t$  and to $\bm{\Theta}_{\bm{a'},\widetilde{\bm{a}}}$ or $\vf(\bm{\Theta}_{\bm{a'},\widetilde{\bm{a}}})$ at $i$, depending on the choice of a trivialization of the fibration that determines the action of $\wf$ (in the picture in the middle of Figure \ref{Figure: The homotopy R} we represented, in the usual way, the trivialization that induces limits at $\bm{\Theta}_{\bm{a'},\widetilde{\bm{a}}}$).

\begin{figure} [ht!] 
  \begin{center}
  \includegraphics[scale = 1]{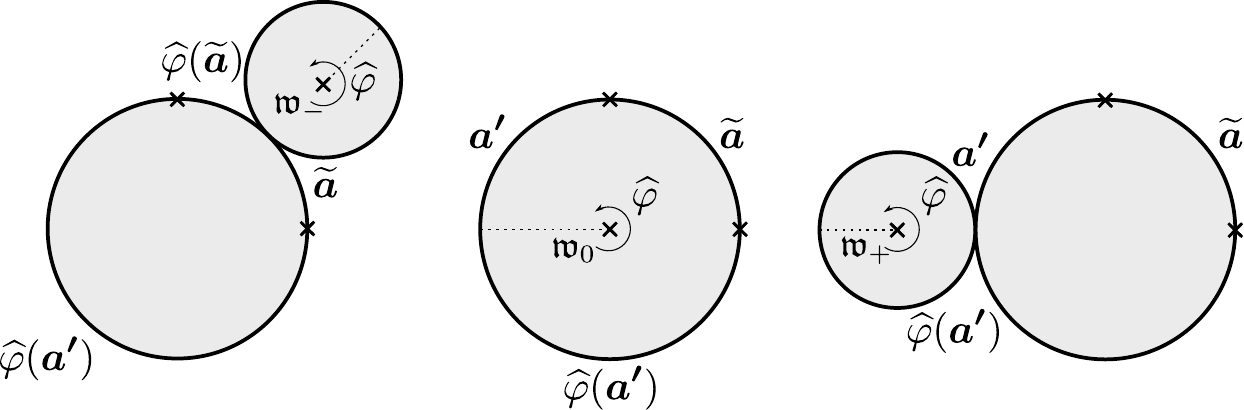}
  \end{center}
  \caption{The homotopy from $B^{\mathcal{R}}_+$ (left) to $B^{\mathcal{R}}_-$ (right), with $B^{\mathcal{R}}_0$ pictured in the middle.  Each boundary component is labeled with the set of curves in the corresponding Lagrangian boundary condition $Q^{\mathcal{R}}_t$.}
  \label{Figure: The homotopy R}
\end{figure}

The family $B_t^{\mathcal{R}}$ can be compactified by adding two-level buildings
\[B_+^{\mathcal{R}} = ( \D \setminus \{ 1, i, e^{i \pi/4} \} ) \sqcup  (\D \setminus \{0, -e^{i \pi/4} \} )\] 
as $t \to +1$ and 
\[B_+^{\mathcal{R}} = (\D \setminus \{1, i, - e^{i \pi /4} \}) \sqcup (\D \setminus \{ e^{i \pi/4}, 0\})\] 
as $t \to -1$. These degenerations correspond, by pull back, to degenerations $\pi_\pm^{\mathcal{R}} \colon E_\pm^{\mathcal{R}} \to B_\pm^{\mathcal{R}}$ of $\pi_t^{\mathcal{R}} \colon E_t^{\mathcal{R}} \to B_t^{\mathcal{R}}$. Thus, after reparametrisation, $(E_+^{\mathcal{R}}, \pi_+^{\mathcal{R}}) = (E_{\Upsilon},\pi_{\Upsilon}) \sqcup (E^{\vf}_{\iota},\pi_{\iota})$ and the corresponding count of multisections gives $\widetilde{\iota} \circ \Upsilon$. Similarly, $(E_-^{\mathcal{R}}, \pi_-^{\mathcal{R}}) = (E_i,\pi_i) \sqcup (W_+,\pi_{B_+})$ and the corresponding count of multisections gives $\Phi^{\sharp} \circ i$.
\end{proof}

\subsubsection{The second square} \label{Subsubsection: Second square} 
\begin{Lemma}
 There exists a chain homotopy 
 \[\mathcal{S} \colon \widehat{CFK}(S,\bm{a'},\vf(\bm{a'}), z, -g+1) \longrightarrow CF^{\sharp}(\vf\circ \tau_L^{-1}) \]
from $\widetilde{\lambda} \circ \Phi^{\sharp}$ to $\Phi^{\sharp}\circ l$ that makes the following diagram commute in homology: 
\begin{equation} \label{Diagram: Second main square}
\begin{tikzpicture}[->,auto,node distance=1.5cm,>=latex',baseline=(current  bounding  box.center)]
\matrix (m) [matrix of nodes, row sep=1cm,column sep=2cm]
{ $\widehat{CFK}(S,\bm{a'},\vf(\bm{a'});-g+1)$ & $\widehat{CFK}(S,\bm{a},\vf(\bm{a'});-g+1)$  \\
  $CF^{\sharp}(\vf)$ & $CF^{\sharp}(\vf\circ \tau_L^{-1}).$ \\ };
 \path[->]
  (m-1-1) edge node[right] {$\Phi^{\sharp}$}   (m-2-1)
          edge node[above] {$l$} (m-1-2)
          edge[dashed] node[above] {$\mathcal{S}$} (m-2-2);
 \path[->]
  (m-1-2) edge node[right] {$\Phi^{\sharp}$} (m-2-2);
 \path[->]
  (m-2-1) edge node[below] {$\widetilde{\lambda}$} (m-2-2);
\end{tikzpicture}
\end{equation}
\end{Lemma}
\begin{proof}
 Let $\{\mathfrak{w}_t\}_{t\in (-1,0)}$ be the path in $\D$ defined by $\mathfrak{w}_t = 
te^{i \pi/4}$.
 Consider the one-parameter family of punctured Riemann surfaces $B^{\mathcal{S}}_t = \D\setminus \{1, \mathfrak{w}_t\}$ endowed with the marked point $\mathfrak{z}=0$. Let $\pi_t^{\mathcal{S}} \colon E_t^{\mathcal{S}} \to E_t^{\mathcal{S}}$ be the Lefschetz fibration over $B_t^{\mathcal{S}}$ with regular fiber $\widehat{S}$, one critical point over $\mathfrak{z}$ with vanishing cycle $L$ and monodromy $\widehat{\vf}\circ \tau_L^{-1}$ around $\mathfrak{w}_t$. We endow these Lefschetz fibrations with the Lagrangian boundary conditions $Q_t^{\mathcal{S}}$ induced by the one represented in the picture in the middle of Figure \ref{Figure: The homotopy S}.
 
 The homotopy $\mathcal{S}$ is defined by counting pairs $(t,u)$ where $t \in (-1,0)$ and $u$ is a degree $2g$ multisection of $\pi_t^{\mathcal{S}} \colon E_t^{\mathcal{S}} \to B_t^{\mathcal{S}}$ of index $-1$, with boundary on $Q_t^{\mathcal{S}}$,  and which are asymptotic to a generator of  $\widehat{CFK}(S,\bm{a'},\vf(\bm{a'});-g+1)$ at $1$ and to a generator of $CF^\sharp(\vf \circ \tau_L^{-1})$ at $\mathfrak{w}_t$.

 \begin{figure} [ht!] 
  \begin{center}
  \includegraphics[scale = 1]{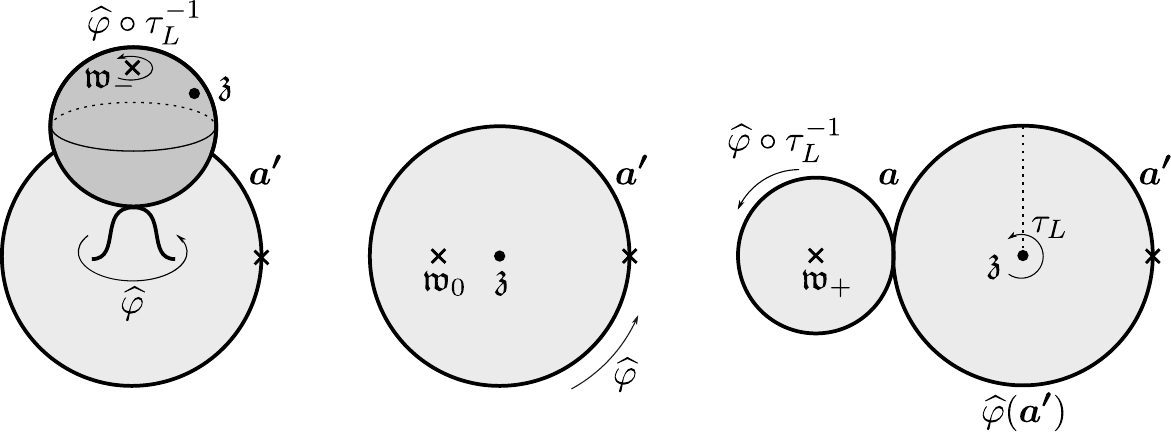}
  \end{center}
  \caption{The homotopy from $B^{\mathcal{S}}_{0}$ (left) to $B^{\mathcal{S}}_-$ (right), with $B^{\mathcal{S}}_{-\frac{1}{2}}$ pictured in the middle.  Each boundary component is labeled with the set of curves in the corresponding Lagrangian boundary condition $Q^{\mathcal{S}}_t$.}
  \label{Figure: The homotopy S}
\end{figure}
 
 The family $B_t^{\mathcal{S}}$ can be compactified by adding two-level buildings
 \[B_0^{\mathcal{S}} =  B_{\iota} \sqcup (S^2 \setminus \{\mathrm{N},\mathrm{S}\})\]
 with $\mathfrak{z} \in S^2 \setminus \{\mathrm{N},\mathrm{S} \}$ as $t \to 0$, and (after reparametrisation)
 \[B_-^{\mathcal{S}} = B_l \sqcup B_{\iota}\] 
 with $\mathfrak{z} \in B_l$, as $t \to -1$. Again, these degenerations correspond, by pull-back, to degenerations $\pi_0^{\mathcal{S}} \colon E_0^{\mathcal{S}} \to B_0^{\mathcal{S}}$ and $\pi_-^{\mathcal{S}} \colon E_-^{\mathcal{S}} \to B_-^{\mathcal{S}}$ of $\pi_t^{\mathcal{S}} \colon E_t^{\mathcal{S}} \to B_t^{\mathcal{S}}$. It is not difficult to see that $(E_-^{\mathcal{S}}, \pi_-^{\mathcal{S}}) \cong (W_+,\pi_{B_+}) \sqcup (E^{\vf}_{\lambda},\pi_{\lambda})$ and the corresponding count of multisections gives $\widetilde{\lambda} \circ \Phi^{\sharp}$. Similarly, $(E_0^{\mathcal{S}}, \pi_0^{\mathcal{S}}) = (E_l,\pi_l) \sqcup (W_+,\pi_{B_+})$ and the corresponding count of multisections gives $\Phi^{\sharp} \circ l$.
\end{proof}

\subsubsection{The homotopy of homotopies}
\begin{Lemma}\label{lemma: homotopy of homotopies}
 There exists a homotopy 
 \[\mathcal{T} \colon \widehat{CFK}(S,\widetilde{\bm{a}},\vf(\bm{a'});-g+1) \longrightarrow CF^{\sharp}(\vf\circ \tau_L^{-1}) \]  
 from $(\mathcal{H}'\circ\Upsilon + \widetilde{\lambda} \circ \mathcal{R} )$ to $(\Phi^{\sharp} \circ \mathcal{H} + \mathcal{S} \circ i)$. 
\end{Lemma}
\begin{proof}
Let $\Delta = \{ (s,t) \in \R^2 : 0 < s <1, 0 < t <s \}$ be an open two-dimensional symplex. 
For $(s, t) \in \Delta$ we choose points $\mathfrak{z}_s = se^{i \pi/4}$ and $\mathfrak{w}_t = te^{i \pi/4}$ and define the two-parameter familty of punctured Riemann surfaces 
\[ B^{\mathcal{T}}_{s,t}= \D \setminus \{1, i, \mathfrak{w}_t\} \]
endowed with the marked point $\mathfrak{z}_s$.

 Let $\pi^{\mathcal T}_{s,t} \colon E_{s,t}^{\mathcal{T}} \to B^{\mathcal{T}}_{s,t}$ be the Lefschetz fibration with fibre $\widehat{S}$, one critical point over $\mathfrak{z}_s$ with fanishing cycle $\tau_L$ and monodromy $\wf \circ \tau_L^{-1}$ around $\mathfrak{w}_t$.
We endow $E_{s,t}^{\mathcal{T}}$ with the Lagrangian boundary condition $Q_{s,t}^{\mathcal{T}}$ induced by the labelling in Figure \ref{Figure: The homotopy T}.

The map ${\mathcal T}$ is defined by counting triples $(s,t,u)$ where $(s,t) \in \Delta$ and
$u$ is a degree $2g$ holomorphic multisection of $\pi_{s,t}^{\mathcal T} \colon E_{s,t}^{\mathcal T} \to B_{s,t}^{\mathcal T}$ of index $-2$, with boundary on $Q_{s,t}^{\mathcal{T}}$, and which are asymptotic to a generator of $\widehat{CFK}(S, \widetilde{\bm{a}}, \vf(\bm{a}'), z, 1-g)$ at $1$, to $\bm{\Theta}_{\bm{a'},\widetilde{\bm{a}}}$ at $i$ and to a generator of $CF^\sharp(\vf \circ \tau_L^{-1})$ at $\mathfrak{w}_t$.
\begin{figure} [ht!] 
  \begin{center}
  \includegraphics[scale = 1.12]{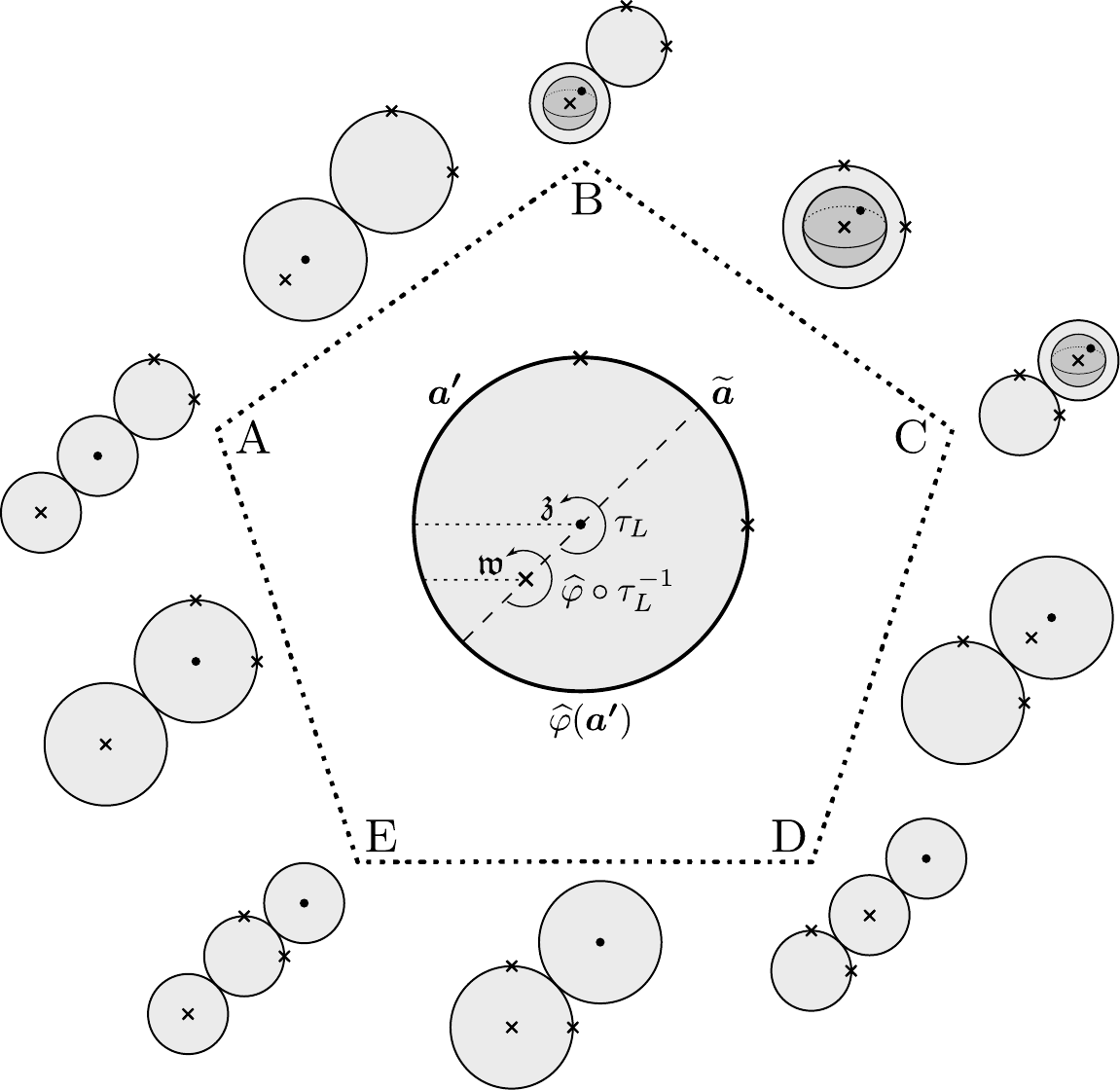}
  \end{center}
  \caption{The compactification of the family $\{ B_{s,t}^{\mathcal{T}} \}_{(s,t) \in \Delta}$. The points $\mathfrak{w}_t$ and $\mathfrak{z}_s$ are free to move over the leaning dashed line.} 
  \label{Figure: The homotopy T}
\end{figure}

The Deligne--Mumford style compactification of the family $\{ B_{s,t}^{\mathcal T} \}_{(s,t) \in \Delta}$ can be represented as a pentagon and the edges and the vertexes of its boundary are associated with the degenerations represented in Figure \ref{Figure: The homotopy T}:
\begin{itemize}
 \item[(AB)] These degenerations occur when $s, t \to -1$. The limit configurations, 
after reparametrisation, are  $B_t^{\mathcal S} \sqcup B_i$ 
and therefore yield the term ${\mathcal S} \circ i$. 
 \item[(BC)] These degenerations occur when $s-t \to 0$. The limit configurations are
$B_\lambda \sqcup B^{\mathcal R}_t$ and therefore they yield the term $\widetilde{\lambda} \circ {\mathcal R}$.
 \item[(CD)] These degenerations occur when $s,t \to 1$. The limit configuration, after reparametrisation, is $B^{\mathcal H}_t \sqcup B_{\Upsilon}$ and therefore yields the term $\widetilde{\mathcal H} \circ \Upsilon$.
 \item[(DE)] These degenerations occur when $s \to 1$ but $t$ remains in $(0,1)$. One of the components of the limit configurations is the base of the fibration of Example \ref{Example: G of tau L with condition L is 0}, and therefore the contribution of these degenerations is zero.
  \item[(EA)] These degenerations occur when $t \to -1$ and $s$ remains in $(0,1)$. The limit configurations, after reparametrisation, are $B_+ \sqcup B^{{\mathcal H}'}_t$ and therefore they yield the term $\Phi^\sharp \circ {\mathcal H}'$.
\end{itemize}
\end{proof}

From Lemma \ref{lemma: homotopy of homotopies}, Lemma \ref{Come mettere a posto le cose} and Lemma \ref{Lemma: Algebraic lemma for the third square} we get the main ingredient for the proof of Theorem \ref{Theorem: Main theorem in introduction}.
\begin{Thm}
 The diagram \eqref{Diagram: Main diagram between HF and HF without maps} commutes.
\end{Thm}

\subsection{The induction}
As recalled in Subsection \ref{Subsection: Heegaard Diagrams}, the Dehn-Lickorish-Humphries theorem states that, up to isotopy, any diffeomorphism $\vf$ of $S$ can be decomposed as
\[\vf = \tau_n \circ \dots \circ \tau_1 \]
for some $n\in \mathbb{N}$, where $\tau_i$ is a positive or negative Dehn twist along some nonseparating curve $L_i$. In this subsection we use Diagram \eqref{Diagram: Main diagram between HF and HF without maps} to finish the proof of Theorem \ref{Theorem: Main theorem in introduction}. We will proceed by induction on $n$.

\subsubsection{Initialization} \label{Subsubsection: Initialization}
The base case for the induction is when $\vf$ is Hamiltonian isotopic to the identity. We construct $\wf$ as the time-one flow of a Hamiltonian function $\widehat{H} \colon \widehat{S} \to \R$. We fix a function $H \colon S \to \R$ satisfying the following properties.
\begin{itemize}
\item In $A$ it depends only on the coordinate $y$, and moreover:
\begin{itemize}
\item[$\circ$] $\frac{\partial H}{\partial y}>0$ in $\{ y \in (2, 3] \}$,
\item[$\circ$] $\{y=0 \} = \partial S$ is a Morse-Bott circle of minima for $H$,
\item[$\circ$] $\frac{\partial H}{\partial y} \le 0$ and $\frac{\partial^2 H}{\partial y^2} \le 0$ in $\{ y \in [0, 2)  \}$,
\item[$\circ$] $\{ y=1 \}$ is a Morse-Bott circle of minima for $\frac{\partial H}{\partial y}$, and
\item[$\circ$] $\frac{\partial H}{\partial y}$ is small in absolute value near $y=0$.
\end{itemize}
\item In $\widehat{S} \setminus A$ it is a Morse function with a unique maximum $e_+$ and 
$2g$ saddles $h_1, \ldots, h_{2g}$, and its differential is small.
\end{itemize}
The conditions on $A$ are chosen so that the Hamiltonian flow of $H$ produces the finger move of Figure \ref{fig:trivial region}. The Function $\widehat{H}$ is obtained by perturbing the Morse-Bott circle of minima at $y=0$ into a minimum $e$ and a saddle $h$.

The chain complex $CF^\sharp(\vf)$ is generated by $e_+, h_1, \ldots, h_2g$ and holomorphic cylinders contributing to the differential of $CF^\sharp(\vf)$ correspond, by Morse-Bott theory, to negative gradient flow lines between generators. This is a fairly elementary instance of Morse-Bott perturbation, where the correspondence between holomorphic cylinders and Morse flow-lines can be worked out explicitly. In fact the projection of a holomorphic cylinder in $\R \times \widehat{T}_\vf$ to $\widehat{S}$ satisfied the Floer equation, and for a suitably small Hamiltonian function (and in absence of $J$-holomorphic spheres, which is the case for $\widehat{S}$) the solutions of the Floer equation are in bijection with the flow lines between critical points. Then the differential in $CF^\sharp(\vf)$ is
\begin{equation}\label{differential for phi=id}
\partial^\sharp e_+=h, \quad \partial^\sharp h = \partial^\sharp h_1 = \ldots = \partial^\sharp h_{2g}=0.
\end{equation}

Next we choose a convenient basis of arcs $a_1, \ldots, a_{2g}$ of $\widehat{S}$. First we extend $A=[0,3] \times S^1$ inside $\widehat{S}$ to an annulus $\widehat{A}= [-1, 3] \times S^1$ so that no critical point of $\widehat{H}$ is contained in $[-1, 0] \times S^1$.  We define the arcs such that 
\begin{itemize}
\item $a_i \cap (\widehat{S} \setminus \widehat{A})$ is the unstable manifold of the critical point $h_i$ for all $i=1, \ldots, 2g$, 
\item in $\widehat{A} \setminus A$ the arcs come close together, 
\item  $a_i \cap A = \{ \theta= \theta_i\}$ for some $\theta_i \in S^1$.
\end{itemize}
We also assume that the distance between the arcs in $A$ is smaller than the size of the finger move and that all intersection points between $\bm{a}$ and $\wf{\bm{a}}$ are contained in $A$ except for $h_1, \ldots, h_{2g}$.  


We denote by $\bm{h}_i$ the equivalence class of $(c_1, \ldots, c_{i-1}, h_i, c_{i+1}, \ldots, c_{2g})$ for $i=1, \ldots, 2g$, which are evidently generators of $\widehat{HFK}(S, \bm{a}, \vf(\bm{a}), z, 1-g)$, and by $\mathbf{H}$ the subspace they generate. We denote also by $\mathbf{C}_0$ the subspace of $\widehat{HFK}(S, \bm{a}, \vf(\bm{a}), z, 1-g)$ generated by all other generators.
\begin{Lemma}\label{convenient splitting}
$\mathbf{H}$ and $\mathbf{C}_0$ are subcomplexes of $\widehat{CFK}(S, \bm{a}, \vf(\bm{a}), z, 1-g)$. The differential on $\mathbf{H}$ is trivial and $\mathbf{C}_0$ is acyclic.
\end{Lemma}
\begin{proof}
The intersection points $c_j$ and $c_j'$ cannot appear at the positive end of a nontrivial component of a holomorphic curve contributing to $\partial \bm{h}_i$. This implies that any such curve must consist of a nontrivial section with boundary on $a_i$ and $\vf(a_i)$, which however must pass through the point $z$. This shows that $\partial \bm{h}_i=0$, and therefore $\mathbf{H}$ is a subcomplex on which the differential is trivial.

To show that $\mathbf{C}_0$ is also a subcomplex, we observe that no holomorphic curve contributing to the differential in $\widehat{CFK}(S, \bm{a}, \vf(\bm{a}), z, 1-g)$ can have a nontrivial component with a negative end at $h_i$ because the projection of that component to $\widehat{S}$ should cover a region which intersect $\partial \widehat{S}$. Then the only possibility left is the union of a trivial section over $h_i$ with a multisection with negative ends at intersection points of the form $c_j$ or $c_j'$ and whose projection to $\widehat{S}$ is contained in $\widehat{A}$. One can see that such a curve is either the union of trivial sections, or must cross the basepoint $z$: the portion of the diagram in $\widehat{A}$ is nice (in the sense of Sarkar and Wang \cite{combinatorial HF}) and therefore the $J$-holomorphic curves contributing to the differential correspond to empty bigons and rectangles in the diagram. One can readily check that bigons must cross $z$ and rectangles cannot have two diagonally opposed vertices in $\{ c_1, c_1', \ldots, c_{2g}, c_{2g}')$. 

When $g=1$ the diagram $(S, \bm{a}, \vf(\bm{a}), z)$ describes the fibred knot in $(S^2 \times S^1) \#  (S^2 \times S^1)$ which is denoted by $B(0,0)$ in \cite[ Section 9]{OS3} so, in general, it describes $B(0,0)^{\# g} \subset  (S^2 \times S^1)^{\# 2g}$. By  Proposition 9.2 \cite[Proposition 9.2]{OS3} and the Künneth formula for knot Floer homology \cite[Theorem 7.1]{OS3} then $\dim \widehat{HFK}(S, \bm{a}, \vf(\bm{a}), z, 1-g)=2g= \dim \mathbf{H}$, from which we deduce that $\mathbf{C}_0$ is acyclic. 
\end{proof}

Let $\Phi^\sharp_0$ be the ``low energy part'' or $\Phi^\sharp$. The low energy 
$J$-holomorphic multisections contributing to $\Phi^\sharp$ which have $(c_1, \ldots, c_{i-1}, h_i, c_{i+1}, \ldots, c_{2g})$ at the positive end consist the union of the horizontal section over $h_i$ and $2g-1$ sections from $c_i$ to $e$. By construction $h_i$ is a positive intrsection point between $a_i$ and $\wf(a_i)$, and therefore the trivial section over $h_i$ has index zero by Lemma \ref{caldo} and is regular by Lemma 2.27 of \cite{Se5}. The low energy sections from $c_i$ to $e$ are obtained by Morse-Bott perturbation of horizontal sections over $c_i$ of the fibration with monodromy $\vf$.

Then $\Phi^\sharp_0(\mathbf{h}_i)=h_i$, which implies, by Gauss elimination, that the composition 
$$\mathbf{H} \to \widehat{CFK}(\bm{a}, \wf(\bm(a)), z, 1-g) \xrightarrow{\Phi} CF^\sharp \to \langle h_1, \ldots, h_{2g} \rangle$$
is an isomorphism, where the first map is the inclusion and the last map is the projection.  
The first and the last maps induce isomorphisms in homology by Lemma \ref{convenient splitting} and Equation \eqref{differential for phi=id}. This proves that $\Phi^\sharp$ is an isomorphism when $\vf$ is Hamiltonian isotopic to the identity.

\subsubsection{The inductive step}
Now we assume that $\vf= \tau_{L_1}^{s_1} \circ \ldots \circ \tau_{L_n}^{s_n}$ where $L_i$ is 
a nonseparating simple closed curve in $S$ and $s_i \in \{ +1, -1 \}$. Denote $\psi= \tau_{L_1}^{s_1} \circ \ldots \circ \tau_{L_{n-1}}^{s_{n-1}}$. The inductive hypothesis is that
$$\Psi^\sharp_* \colon \widehat{HFK}(S, \bm{a}, \psi(\bm{a}), z, 1-g) \to HF^\sharp (\psi)$$
is an isomorphism. If $s_n=-1$, then $\vf= \psi \circ \tau_L^{-1}$, and therefore we have the commutative diagram of exact sequences

 \begin{equation*}
 \begin{tikzpicture}[->,auto,node distance=1.5cm,>=latex',baseline=(current  bounding  box.center)]
 \matrix (m) [matrix of nodes, row sep=1cm,column sep=1cm] 
 {${HF}(\vf(L),L)$ & $\widehat{HFK}(S, \bm{a}, \psi(\bm{a}), a, 1-g)$ &  $\widehat{HFK}(S, \bm{a}, \psi(\bm{a}), a, 1-g)$ \\
  ${HF}(\vf(L),L)$ & $HF^\sharp(\vf) $ & $HF^\sharp(\psi).$ \\};
 \path[->]
  (m-1-1) edge node[above] {$i_*$} (m-1-2)
          edge node[right] {$(\Upsilon_-)_*$} (m-2-1);
 \path[->]
  (m-1-2) edge node[right] {$\Phi^\sharp_*$} (m-2-2)
          edge node[above] {$l_*$} (m-1-3);
 \path[->]
  (m-1-3) edge node[right] {$\Phi^\sharp_*$}  (m-2-3)
   edge[bend right=25] node[above] {$d$} (m-1-1);
 \path[->]
  (m-2-1) edge node[below] {$\iota_*$} (m-2-2);
 \path[->]
  (m-2-2) edge node[below] {$\lambda_*$} (m-2-3);
 \path[->]
  (m-2-3) edge[bend left=25] node[below] {$\delta$} (m-2-1);
 \end{tikzpicture}
\end{equation*}
The first vertical map in the diagram is an isomorphism by Lemma \ref{The map Upsilon} and the third vertical map is an isomorphism by the inductive hypothesis, and therefore by the five
lemma the middle vertical map is also an isomorphism. If $s_n=+1$, then $\psi= \vf \circ \tau_{L_n}^{-1}$ and the argument is similar.
\section{Applications}

Knot Floer homology is a powerful invariant of knots and it is an interesting question to understand what kind informations one can extract from it. From this point of view Theorem \ref{Theorem: Main theorem in introduction} has interesting consequences.

Let $\vf \colon S \rightarrow S$ be a diffeomorphism and let $\vf_c$ be the \emph{canonical Nielsen--Thurston representative} of the mapping class $[\vf]$. Recall that, by the Thurston's classification of surfaces homeomorphisms, $\vf_c$ may be \emph{pseudo-Anosov, periodic} or \emph{reducible} (see \cite{T}, \cite{FLP} and \cite{CC} for details). When, as in our case, $\partial S \cong S^1 $ it is convenient to see $\vf_c$ as a homeomorphism of a subsurface $S \setminus A$ of $S$, where $A$ is a collar of $\partial S$, such that $\vf_c$ is isotopic relative to $\partial (S \setminus A)$ to $\vf|_{S \setminus A}$.

The canonical representative $\vf_c$ of $[\vf]$ defines a \emph{canonical rotation angle $\vartheta_c \in [0,2\pi)$} that corresponds to the rotation defined by $\vf_c|_{\partial (S\setminus A)}$.
If $\vf|_{\partial S} = id$, the ``difference'' near the boundary between $\vf$ and $\vf_c$ is then encoded by the \emph{fractional Dehn twist coefficient} of $\vf$.
\begin{Def}
 Let $S$ be a surface with boundary $\partial S \cong S^1$ and $\vf \colon S \rightarrow S$ a surface diffeomorphism such that $\vf|_{\partial S} = id$. Fix a collar $A \cong (0,1] \times S^1$ of $\partial S = \{1\} \times S^1$ and perturb $\vf$ by an isotopy relative to $\partial S$ in a way that $\vf|_{(S \setminus A)} = \vf_c$. The fractional Dehn twist coefficient $t_c(\vf)$ of $\vf$ is the winding number of the arc $\vf(\{(1-y,1)|y\in [0,1]\})$.
\end{Def}

Observe in particular that
\begin{equation} \label{Equation: t_c function of theta_c and k_vf}
 2\pi t_c(\vf) = \vartheta_c + 2\pi k_{\vf}
\end{equation}
where $k_{\vf} \in \Z$ is the sum with signs of the full boundary parallel Dehn twists of $\vf$.

For the first application of Theorem \ref{Theorem: Main theorem in introduction} we are interested in fibered knots whose monodromy $\vf$ has canonical representative $\vf_c$ that is irreducible. 

\begin{Thm} \label{Theorem: HFK and the minimal number of fixed points ireducible}
 Let $(K,S,\vf)$ be an open book decomposition of $Y$ with $\vf_c$ irreducible.
 Denote $\mathcal{F}_{min}({[\vf]})$ the minimum number of fixed points that an area-preserving non-degenerate representative of $[\vf]$ may have. Then:
 \begin{equation*}
  \dim\left(\widehat{HFK}(Y,K;-g+1)\right) = \left\{ \begin{array}{ll}
                                                                             \mathcal{F}_{min}({[\vf]}) - 1 & \mbox{if } \vf_c = id \mbox{ and } t_c(\vf) = 0;\\
                                                                             \mathcal{F}_{min}({[\vf]}) + 1 & \mbox{otherwise}.
                                                                         \end{array} \right. 
 \end{equation*} 
\end{Thm}

Before giving the proof we need to recall some of the ideas behind an analogous result for symplectic homology.
\begin{Lemma} \label{Lemma: F_min = dim HF}
If the canonical representative of $\vf$ is irreducible and $k_{\vf}=0$, then: 
 \begin{equation} \label{Equation: F_min = dim HF}
  \dim\left(HF(\vf)\right)=\mathcal{F}_{min}({[\vf]}).
\end{equation}
\end{Lemma}
This result was essentially proven by Gautschi (in the periodic case) and Cotton-Clay (in the pseudo-Anosov case). We refer the reader to \cite{Ga} and \cite{CC} for the details and for some definitions that we will not recall. Observe that we did not specify the version $\pm$ of the Floer homology: this is because the behavior of $\vf|_{\partial S}$ is assumed to coincide with the one prescribed by $t_c(\vf)$ (cf. Remark \ref{Remark: If vf_c trivial near the boundary then you can choose vf_sm near the boundary} below).

We also recall the following result (see Cotton-Clay \cite[Theorem 1.1]{CC2}).
\begin{Lemma}\label{Lemma: F_min = sum of indexes}
If the canonical representative of $\vf$ is irreducible and $k_{\vf}=0$, then: 
\begin{equation}\label{Equation: F_min = sum of indexes}
 \mathcal{F}_{min}({[\vf]}) = \sum_{\eta} |ind(\eta)|
\end{equation}
where the sum is over the set of \emph{Nielsen classes} of $\vf$ (i.e. the set of free homotopy classes of oriented paths $S^1 \hookrightarrow T_{\vf}$), $ind(\eta) \coloneqq \sum_{\{x \in \mathrm{Fix}(\vf)|\gamma_x \in \eta\}}\varepsilon(x)$ and $\varepsilon(x) \coloneqq \mathrm{sign}(\det(\mathbbm{1} - d_x\vf))$ denotes the \emph{Lefschetz sign} of $x$. 
\end{Lemma}
We remark that lemmas \ref{Lemma: F_min = dim HF} and \ref{Lemma: F_min = sum of indexes} can be generalized to the irreducible case (see Theorem 4.16 of \cite{CC} and Theorem 1.1 of \cite{CC2}) but to prove Theorem \ref{Theorem: HFK and the minimal number of fixed points ireducible} we will only use the statements for irreducible maps plus some partial results of \cite{CC} about boundary parallel Dehn twists. In the rest of this section, if $C \subset S$ is a set of fix points all in the same Nielsen class, we will refer to this class as to the Nielsen class of $C$. 

The two main ingredients to show \eqref{Equation: F_min = dim HF} and \eqref{Equation: F_min = sum of indexes} are the following:
\begin{enumerate}[leftmargin=*,label=(I.\arabic*)]
 \item  \label{Key ingredient 1} the canonical representative of an irreducible diffeomorphism minimizes the number of fixed points in its mapping class;
 \item  \label{Key ingredient 2} if $\langle\partial x,y \rangle \neq 0$, then $x$ and $y$ are in the same Nielsen class and $\varepsilon(x) \neq \varepsilon(y)$.
\end{enumerate}
The first is a classical result in fixed point theory (see Jang--Guo \cite{JG}), while the second is an obvious consequence of the definitions and the fact that the Floer differential inverts Lefschetz signs (see for example \cite[§2]{Se2}).

If $[\vf] \neq [id]$ is periodic then in \cite[§3]{Ga} Gautschi shows that the canonical representative $\vf_c$ of $[\vf]$ is symplectic and \emph{monotone}: the monotonicity condition was introduced by Seidel in \cite{Se2} to ensure the compactness of the moduli spaces of holomorphic curves that come into play in the definition of $HF$ (even if we will not recall the definition, we underline that condition \eqref{exactness of phi} implies monotonicity). Moreover Gautschi shows that 
all the fixed points of $\vf_c$ are in different Nielsen classes and have Lefschetz signs $+1$ (see \cite[Proposition 9]{Ga}): \ref{Key ingredient 2} implies then that the Floer differential of $CF(\vf)$ vanishes and \ref{Key ingredient 1} gives equations \eqref{Equation: F_min = dim HF} and \eqref{Equation: F_min = sum of indexes}.

If $[\vf]$ is pseudo-Anosov the argument is a bit more complicate because in general $\vf_c$ is not smooth (and in particular not symplectic). On the other hand in \cite[§3]{CC} Cotton--Clay shows that if $\vf$ is symplectic and non-degenerate then is also \emph{weakly monotone}: the last condition was defined by Cotton Clay as a weaker version of Seidel's monotonicity, but which is still sufficient to ensure the good definition of $HF(\vf)$. In particular, if $\vf_{min} \in [\vf]$ is any symplectic non-degenerate representative with exactly $\mathcal{F}_{min}({[\vf]})$ fixed points, then $HF(\vf_{min})$ is well defined and, by invariance of $HF$:
\begin{equation}\label{Equation: dim of HF leq f_min}
  \dim\left(HF(\vf)\right) = \dim\left(HF(\vf_{min})\right) \leq \mathcal{F}_{min}({[\vf]}).
\end{equation}
His strategy to show \eqref{Equation: F_min = dim HF} consists then in smoothing the singularities of $\vf_c$ to obtain a special symplectic representative $\vf_{sm}$ of $[\vf]$ for which the differential of $CF(\vf_{sm})$ vanishes. This gives
\begin{equation}\label{Equation: dim of HF geq f_min}
 \dim\left(HF(\vf)\right) = \dim\left(HF(\vf_{sm})\right) = \#\mathrm{Fix}(\vf_{sm}) \geq \mathcal{F}_{min}({[\vf]}),
\end{equation}
which, in combination with \eqref{Equation: dim of HF leq f_min}, implies Equation \eqref{Equation: F_min = dim HF}.

The special symplectic representative $\vf_{sm}$ (which, by what we just said, has exactly $\mathcal{F}_{min}({[\vf]})$ fixed points) is obtained as follows. By \cite{BK} and \cite{JG}, it is known that any two fixed (smooth or singular) points of the (canonical) pseudo-Anosov homeomorphism $\vf_c$ are in different Nielsen classes. In order to get $\vf_{sm}$, in \cite[§3]{CC} Cotton--Clay perturbs $\vf_c$ in a neighborhood of all the singular fixed point. If $\overline{x}$ is one of these singularities and has $p$ \emph{prongs} (cf \cite[§3.2]{CC}), the perturbation for $\vf_{sm}$ produces either:
\begin{itemize}
 \item[(i)\phantom{i}] one elliptic fixed point if $\overline{x}$ is a \emph{rotated singularity}, i.e. if the prongs are cyclically permuted by a non trivial permutation;
 \item[(ii)] $p-1$ positive hyperbolic fixed points if $\overline{x}$ is an \emph{unrotated singularity}. i.e. if the permutation of the prongs is trivial.
\end{itemize}
It is important to underline the fact that all the fixed points of $\vf_{sm}$ described in (i) and (ii) are in the same Nielsen class of $\overline{x}$. 

To deal with the boundary (that we assume connected here) of the surface, Cotton--Clay applies in \cite[§4.2]{CC} the ideas of \cite[§2.1]{JG} and proceeds as follows. One starts with a closed surface $\overline{S}$ and a pseudo-Anosov homeomorphism $\overline{\vf}_c$ of $\overline{S}$ with a (hyperbolic) fixed point $\overline{x}$ such that $\overline{S} \setminus \{\overline{x}\} \cong \mathrm{int}(S)$ and $\overline{\vf}_c|_{\overline{S} \setminus \{\overline{x}\}}$ is isotopic to $\vf_c|_{\mathrm{int}(S)}$. In order to get $\vf_{sm}$ one smooths all the singular fixed points of $\overline{\vf}_c$ to obtain a symplectomorphism $\overline{\vf}_{sm} \sim \overline{\vf}_c$: the smoothing is as above for all the singular fixed points different from $\overline{x}$, while the latter is smoothed in a slightly different way that gives:
\begin{itemize}
 \item[(i')\phantom{i}] one elliptic fixed point $\overline{x}_e$ if $\overline{x}$ was rotated;
 \item[(ii')] one elliptic fixed point $\overline{x}_e$ plus $p$ positive hyperbolic fixed points if $\overline{x}$ was unrotated with $p$ prongs.
\end{itemize}
In either case one removes a $\overline{\vf}_{sm}$-invariant open neighborhood $\mathcal{N}(\overline{x}_e)$ of $\overline{x}_e$ (not containing other fixed points) and obtains the special symplectic representative $\vf_{sm}\coloneqq \overline{\vf}_{sm}|_{\overline{S} \setminus \mathcal{N}(\overline{x}_e)}$ that we were looking for. It is again important to remark that all the fixed points in both (i') and (ii') are in the same Nielsen class of $\overline{x}$, which in turn is different from the Nielsen class of any other fixed point of $\overline{\vf}_{sm}$.

Now, by the separation of the Nielsen classes of the fixed points of $\vf_c$, the construction above implies that two fixed points $x$ and $y$ of $\vf_{sm}$ are in the same Nielsen class if and only if they come from the perturbation of type (ii) (or (ii')) of one unrotated singularity (or, respectively, puncture): in this case $x$ and $y$ are both positive hyperbolic, so that $\varepsilon(x) = \varepsilon(y) = -1$. Then \ref{Key ingredient 2} implies that the differential of $CF(\vf_{sm})$ is $0$, so that
\[\dim\left(HF(\vf)\right) = \dim(HF(\vf_{sm})= \#\mathrm{Fix}(\vf_{sm}) \]
and
\[\mathcal{F}_{min}({[\vf]}) = \mathrm{Fix}(\vf_{sm}) = \sum_{\eta}\bigg|\sum_{\{x \in \mathrm{Fix}(\vf_{sm})|\gamma_x \in \eta\}}\varepsilon(x)\bigg| =  \sum_{\eta} |ind(\eta)| \]
giving Equation \eqref{Equation: F_min = sum of indexes}.

\begin{Rmk}
 For simplicity, in the rest of the paper we will talk about special symplectic representatives also referring to periodic diffeomorphisms by setting simply $\vf_{sm} \coloneqq \vf_c$. This makes sense since by Gautschi $\vf_c$ is symplectic, monotone and, if $\vf_c \neq id$, also non-degenerate.
\end{Rmk}

\begin{Rmk} \label{Remark: If vf_c trivial near the boundary then you can choose vf_sm near the boundary}
 As said before, $\vf_c$ uniquely defines a canonical rotation angle $\vartheta_c$ for $\partial S$. On the other hand we do not have a similar well defined rotation angle $\vartheta_{sm}$ for $\vf_{sm}$, since it depends on the choice of a Hamiltonian isotopy near $\partial S$. As remarked by Cotton Clay in Section 4.2 of \cite{CC}, if $\vartheta_c \neq 0$ (which occurs either when $\vf_c\neq id$ is periodic or when $\vf_c$ is pseudo-Anosov with a rotated puncture) then a natural choice is $\vartheta_{sm} = \vartheta_c$. On the other hand if $\vartheta_c = 0$, in order to ensure non-degeneracy we have to choose a (positive or negative) small non-trivial rotation angle $\vartheta_{sm}$ and consequently perturb $\vf_c$ to get a special representative $\vf_{sm}$ that satisfies $\vf_{sm}|_{\partial S}=\vartheta_{sm}$. We then follow Cotton-Clay's conventions and choose $-1 \ll \vartheta_{sm}< 0$ if $t_c(\vf)>0$ and $0<\vartheta_{sm}\ll 1$ if $t_c(\vf)\leq 0$. As we will see later, this choice minimizes the number of fixed points of $\vf_{sm}$ near the boundary. Observe that, once these conventions fixed, we do not need to specify a $\pm$--version for  $HF(\vf_{sm})$.
\end{Rmk}

Now we briefly recall what happens when $\vf_c$ is reducible. In this case there exists a $\vf_c$--invariant collection $C$ of essential circles such that $S\setminus C = \{A_1,\ldots,A_l\}$ where, if $k_i$ is the smallest integer for which $\vf^{k_i}_c$ maps $A_i$ to itself, then $\vf^{k_i}_c|_{A_i}$ is either periodic or pseudo-Anosov. We will refer to $\{A_1,\ldots,A_l\}$ as to \emph{the canonical decomposition of $S$ induced by $\vf_c$}, which is well defined only up to isotopy.
Cotton-Clay defines then component-wise the special symplectic representative $\vf_{sm}$ of $\vf$, in a way that, if $k_i=1$, then $\vf_{sm}|_{A_i}$ coincides with the special representative $(\vf_c|_{A_i})_{sm}$ of the corresponding type (periodic or pseudo-Anosov) as defined above. In particular $\vartheta_{sm}$ depends only on the component of $S\setminus C$ containing $\partial S$.

\begin{Rmk} \label{Remark: Monotonicity for reducible maps}
 The reason for which the argument to show \eqref{Equation: F_min = dim HF} can not be directly used when $\vf_c$ is reducible is that in this case a non-degenerate symplectic representative of $[\vf_c]$ is not necessarily weakly monotone: in particular we can not a priori ensure that there exists a representative $\vf_{min}$ realizing $\mathcal{F}_{min}([\vf])$ for which $HF(\vf_{min})$ is well defined. 
\end{Rmk}

\begin{Def} Given a surface diffeomorphism $\vf \colon S \rightarrow S$, a subset $F\subset S$ is called \emph{$\vf$--invariant} if $\vf(F) \subset F$.
 Two $\vf$--invariant sets $F_0$ and $F_1$ are \emph{$\vf$--related} if there exists a path $c\colon ([0,1],0,1) \rightarrow (S,F_0,F_1)$ such that $\vf \circ c \simeq c$ through maps $([0,1],0,1) \rightarrow (S,F_0,F_1)$.
\end{Def}
Observe that two fixed points of $\vf$ are in the same Nielsen class if and only if they are \emph{$\vf$--related}.


\begin{Lemma} \label{Lemma: No fixed point is related to the boundary}
  Let $x \in \Fix(\vf_{sm})$ and $A \subset S$ be the component of the canonical decomposition of $S$ induced by $\vf_c$ that contains $\partial S$. Assume that $k_{\vf}=0$. Then $x$ is $\vf_{sm}$--related to $\partial S$ if and only if either:
  \begin{enumerate}
   \item $\vf_c|_A = id$ and $x$ is in the same Nielsen class of $\Fix(\vf_{sm}|_A)$ or
   \item $\vf_c|_A$ is pseudo-Anosov and $x$ is one of the positive hyperbolic fixed points of case (ii').
  \end{enumerate}
\end{Lemma}
\begin{proof}
 This comes directly from Cotton-Clay's adaptation to $\vf_{sm}$ of the work of Jiang and Guo \cite{JG} about $\vf_c$. We prove the lemma for $\vf_c$ reducible: the irreducible case is a direct consequence. 
 
 Let $A'\neq A$ be some component of the canonical decomposition of $S$ induced by $\vf_c$ having has some boundary component $C$ in common with $A$. The discussion in Section 4.3 of \cite{CC} implies that a fixed point of $\vf_{sm}|_{A'}$ is $\vf_{sm}$--related to some point of $A$ if and only if $\vf_c|_{A} = id$, $\vf_c|_{A'}$ is pseudo-Anosov, the boundary component $C$ of $A'$ comes from an unrotated puncture and $x$ is one of the hyperbolic fixed points in (ii'). This is in particular the only case in which a point $x \in \Fix(\vf_{sm}|_{S\setminus A})$ can be $\vf_{sm}$--related to $\partial S$. On the other hand the condition $\vf_c|_{A} = id$ implies that $x$ is $\vf_{sm}$--related to the whole $A$, which in turn is obviously $\vf_{sm}$--related to $\partial S$. This gives the first case. 
 
 Assume now that $x \in A$ and $\vf_c|_{A} \neq id$. If $\vf_c|_{A}$ and is periodic then Lemma 1.2 of \cite{JG} implies that no fixed point in $\Int(A)$ is $\vf_{sm}$--related to $\partial S$. If $\vf_c|_{A}$ is pseudo-Anosov and $\partial S$ comes from a rotated puncture then it is easy to see that the hyperbolic fixed points of case (ii') are $\vf_{sm}$--related to $\partial S$: on the other hand Lemma 2.2 of \cite{JG} implies that these are the only fixed points for which the statement can be true (the reader should be aware of the fact that in \cite{JG} the hyperbolic fixed points of case (ii') are assumed to belong to $\partial S$).
\end{proof}

The main step to prove Theorem \ref{Theorem: HFK and the minimal number of fixed points ireducible} is the following result.

\begin{Prop} \label{Proposition: dim HF and dim HF^sharp}
 Let $(K,S,\vf)$ be an open book of genus $g>0$. Let $A \subset S$ be the component of the canonical decomposition of $S$ induced by $\vf_c$ that contains $\partial S$. Then:
 \begin{equation*}
  \dim \left(HF^{\sharp}(\vf)\right) = \left\{ \begin{array}{ll}
                                            \dim \left(HF(\vf)\right) - 1 & \mbox{if } A=S,\ \vf_c = id \mbox{ and } t_c(\vf) = 0;\\
                                            \dim \left(HF(\vf)\right) + 1 & \mbox{otherwise},
                                           \end{array} \right. 
 \end{equation*} 
 where the $\pm$--version of $HF(\vf)$ is the one specified by $\vf_{sm}$ as in Remark \ref{Remark: If vf_c trivial near the boundary then you can choose vf_sm near the boundary}.
\end{Prop}

\begin{proof}
 Observe first Equation \eqref{Equation: F_min = dim HF} implies that in general
\begin{equation} \label{Equation: dim HF^sharp = dim HF pm 1}
 \dim(HF^{\sharp}(\vf)) = \dim(HF(\vf)) \pm 1
\end{equation}
where the contribution $\pm 1$ is given by the generator $p_h \in \partial S$. To prove the result we need to study the interaction under the Floer differential between $p_h$ and the other generators of $CF^{\sharp}(\vf)$. 

Fix a neighborhood $N \cong (0,2] \times S^1$ of $\partial S =\{2\} \times S^1$ on which we fix the usual coordinates $(y,\vartheta)$. Set $N_0 \coloneqq \{y\in (0,1]\}$ and $N_1\coloneqq \{y\in (1,2]\}$. We can then assume that 
\[\vf = \tau_{\vf} \circ \vf_{sm} \] 
where:
\begin{itemize}[leftmargin=*]
 \item the only fixed points of $\vf_{sm}|_{N}$ are the (possible) hyperbolic fixed points obtained as in the smoothing of type (ii') above, in which case they are all contained in $\mathrm{int}(N_0)$ (observe that these fixed points can occur only if $\vartheta_c=0$);
 \item $\vf_{sm}(y,\vartheta) = (y,\vartheta+\vartheta_{sm})$ for every $y \in [0,1]$ where, according to Remark \ref{Remark: If vf_c trivial near the boundary then you can choose vf_sm near the boundary}, we set:
 \begin{enumerate}[leftmargin=*,label=\textbf{(C.\arabic*)}]
  \item if $\vartheta_c \neq 0$ then $\vartheta_{sm} = \vartheta_c$; 
  \item if $\vartheta_c = 0$ and $t_c(\vf) > 0$ then $-1 \ll \vartheta_{sm} < 0$;
  \item if $\vartheta_c = 0$ and $t_c(\vf) \leq 0$ then $0 < \vartheta_{sm} \ll 1$;
 \end{enumerate}
 \item $\tau_{\vf}$ is a symplectomorphism with support in $N_1$ that interpolates between $\vf_{sm}|_{\partial (S \setminus N_1)}$ and $id_{\partial S}$ and realizes $t_c(\vf)$: more precisely, $\tau_{\vf}|_{\overline{N}_1}$ is of the form $\tau_{\vf}(y,\vartheta) = (y,\vartheta - g(\vartheta))$ for some $g \in \mathcal{C}^{\infty}([1,2])$ with $g(1)= g'(1) = 0$ and 
 \[g(2) = 2\pi t_c(\vf) + (\vartheta_{sm} - \vartheta_c) = 2\pi k_{\vf} + \vartheta_{sm}\]
 where $k_{\vf}$ is as in Equation \eqref{Equation: t_c function of theta_c and k_vf};
 \item in order to minimize the number of Dehn twists in $N_1$ and, at the same time, make $\partial T_{\vf}$ into a negative Morse--Bott torus, we require that:
 \begin{enumerate}[leftmargin=*]
  \item[\textbf{(a)}] if $t_c(\vf) \geq 0$ then $g$ is monotone increasing;
  \item[\textbf{(b)}] if $t_c(\vf) < 0$ there exists $0 < \nu\ll 1$ such that $g$ is monotone decreasing in $(1,2-\nu)$ and monotone increasing in $(2-\nu,2]$, $g(2-\nu) = 2\pi k_{\vf}$ and $g(y) \leq 2\pi k_{\vf}+  \vartheta_{sm}$ for all $y \in (2-\nu,2]$ (observe that if $t_c(\vf) < 0$ then $k_{\vf} \lneq 0$ and $\vartheta_{sm} \gneq 0$, so that the function $\vartheta \mapsto \vartheta_{sm} - g(\vartheta)$ has a maximum in $2-\nu$ of value $\vartheta_{sm} + 2\pi |k_{\vf}| $ ).
 \end{enumerate}
\end{itemize}


As explained in Subsection \ref{Subsection: Review of symplectic Floer homology}, the result $\wf$ of a Morse--Bott perturbation of $\vf$ near $\partial_S$ gives two fixed points $p_e$ and $p_h$ that are in the same Nielsen class of $\vf|_{\partial S} = id$ (here we identify the Nielsen classes of $T_{\vf}$ and $\widehat{T}_{\wf}$ via the obvious diffeomorphism). By the property (I.2) above, since $\varepsilon(p_h)=-1$, to prove the proposition it is enough to show that $\partial S$ is not $\vf$--related to any other $x \in \Fix(\vf)$ with $\varepsilon(x)=+1$.
We will do that distinguishing the different cases given by (C.1)--(C.3) and (a) and (b).

\begin{description}
 \item[(C.1-a)] The set of fixed points of $\vf|_{N}$ can be decomposed into a collection of boundary parallel circles
 \begin{equation*} 
  \mathrm{Fix}(\vf|_{N}) = \partial S \sqcup_{i=0}^{k_{\vf}-1} C_{i}
 \end{equation*}
 where $C_{i} \coloneqq \{y = y_i\}$ with $y_i \in (1,2)$ such that $g(y_i) = \vartheta_{sm} + 2i\pi $. Assume that there exists $x \in \mathrm{Fix}(\vf|_{\Int(S)})$ which is $\vf$--related to $\partial S$ via a path $c \colon ([0,1],x,\partial S) \rightarrow (S,x,\partial S)$. 
 
 If $k_{\vf}=0$ then $\Fix(\vf|_N) = \partial S$, so that $x \in S \setminus N$. Since $\vf|_N \sim id$, the restriction of $c$ to $S \setminus \Int(N)$ would give a $\vf_{sm}$--relation between $x$ and $\partial (S\setminus \Int(N))$, which would contradict Lemma \ref{Lemma: No fixed point is related to the boundary} (the condition $\vartheta_c \neq 0$ implies that $\vartheta_c|_A \neq id$).
 
 If $k_{\vf}>0$, the argument to compute the contribution to $\dim(HF(\vf))$ of each $C_i$ is again due to Gautschi and Cotton-Clay and works essentially as follows. Lemma 4.8 and Corollary 4.9 of \cite{CC} imply that each circle of the decomposition above has an its own Nielsen class in $T_{\vf}$ (in \cite{CC} these components are called of type Ib). A Morse--Bott perturbation near $C_i$ gives two new generators of $CF(\vf)$ that, by what we just said, are in their own Nielsen class and, by Morse--Bott theory, there exist exactly two index $1$ holomorphic cylinders between these two fixed points. It follows that their contribution increase by $2$ the total dimensions of $HF(\vf)$ and, by definition, also of $HF^{\sharp}(\vf)$. Summing up over $i$ and adding the contribution of $p_h$ (which also has an its own Nielsen class) we get the result.

 \item[(C.1-b)] The set of fixed points of $\vf|_{N}$ can again be decomposed into a collection of boundary parallel circles
 \begin{equation*} 
  \mathrm{Fix}(\vf|_{N}) = \partial S \sqcup_{i=1}^{-k_{\vf}} C_{i}
 \end{equation*}
 where $C_{i} = \{y = y_i\}$ with $y_i \in (1,2)$ such that $g(y_i) = \vartheta_{sm} - 2i\pi $. 
 Observe that, if $\nu$ is as in the description of the case (b), then $y_{k_\vf} < 2-\nu$.
 
 As in the previous case, each $C_i$ has an its own Nielsen class for $i=1,\ldots, k_{\vf}-1$ and each of these circles contributes by $2$ to both $\dim(HF(\vf))$ and $\dim(HF^{\sharp}(\vf))$. We need then to show that the contribution of $C_{k_{\vf}} \sqcup \partial S$ to $\dim(HF^{\sharp}(\vf))$ is $+1$. First we observe that $C_{k_{\vf}}$ is $\vf$--related to $\partial S$ via a path of the form $\{(y_{k_{\vf}} + t(2-y_{k_{\vf}}),\vartheta_0)|t\in [0,1]\}$ for some $\vartheta_0$. A Morse--Bott perturbation of $C_{k_{\vf}}$ gives two fixed points $p_h'$ and $p_e'$ that are hyperbolic and, respectively, elliptic. The situation here is the same as the one considered in Colin--Ghiggini--Honda \cite[§8-9]{CGH1} and Wendl \cite[§4.2]{W}: their computations directly apply here and we get that $\partial p_e' = p_h$ and $\partial p_h'=0$ (here we use the fact that $p_e$ is not a generator of $CF^{\sharp}(\vf)$). Since $p_h'$ is not $\vf$--related to other fixed points it contributes by $+1$ to $\dim(HF^{\sharp}(\vf))$.

 \item[(C.2)] The conditions $\vartheta_c=0$ and $t_c(\vf)>0$ imply $k_{\vf}\geq 1$. Now we can decompose:
 \begin{equation*} 
  \mathrm{Fix}(\vf|_{N}) = P \sqcup \partial S \sqcup_{i=1}^{k_{\vf}-1} C_{i}
 \end{equation*}
 where $P = \{x_1,\ldots,x_p\}$ is the set of (possible) positive hyperbolic fixed points given by the perturbation of type (ii') above (here $p$ is the number of prongs of the pre-smoothed puncture) and $C_{i} = \{y = y_i\}$ is such that $g(y_i) = \vartheta_{sm} + 2i\pi $. Again each circle of the decomposition has an its own Nielsen class and the proof of the result works as above.

 \item[(C.3-a)] The conditions $\vartheta_c=0$ and $t_c(\vf)=0$ imply $k_{\vf}= 0$, so that
  \begin{equation*} 
  \mathrm{Fix}(\vf|_{N}) = P \sqcup \partial S
 \end{equation*}
 where $P$ is as above. We are then in the situation of Lemma \ref{Lemma: No fixed point is related to the boundary}, which gives only two cases in which $\partial S$ can be $\vf$--related to other fixed points. 
 
 The first possibility is when $\vf_c|_A = id$, so that $P= \emptyset$. For this case we reason as in Lemma 4.14 of \cite{CC} and we assume that $\vf|_A$ is the time $1$ Hamiltonian flow of a small Morse--Smale function on $A$. If $A = S$ we can choose a Morse--Smale function with only one maximum and no minima: this situation has been treated in Section \ref{Subsubsection: Initialization} and gives the special case of the statement.
 If $A \neq S$ we can choose a Morse--Smale function with no maxima or minima, so that the fixed points of $\vf_{sm}|_A$ are all positive hyperbolic. On the other hand, by Lemma \ref{Lemma: No fixed point is related to the boundary}, the only fixed points of $\vf_c$ in the same Nielsen class of $\partial S$ are the fixed points in the interior of $A$ and the (possible) positive hyperbolic fixed points of case (ii') coming from a pseudo-Anosov component $A' \neq A$ that abuts $A$. In any case all these fixed points are positive hyperbolic and by (I.2) they can not interact with $p_h$ under the Floer differential.
 
 The second possibility given by Lemma \ref{Lemma: No fixed point is related to the boundary} is that $\vf|_A$ is pseudo-Anosov, in which case $\partial S$ is in the same Nielsen class of the points of $P$ (if $P \neq \emptyset$). Again these points are all positive hyperbolic and can not interact with $p_h$ under the Floer differential. 
 
 
 \item[(C.3-b)] The conditions $\vartheta_c=0$ and $t_c(\vf)<0$ imply $k_{\vf} \lneq 0$, so that
  \begin{equation*} 
  \mathrm{Fix}(\vf|_{N}) = P \sqcup \partial S \sqcup_{i=1}^{-k_{\vf}} C_{i}
 \end{equation*}
 where $C_{i} = \{y = y_i\}$ is such that $g(y_i) = \vartheta_{sm} - 2i\pi $. Again $P$ and each $C_i$ have a their own Nielsen class and the proof of the result goes as in case (C.1-b).
\end{description}
\end{proof}

\proof[Proof of Theorem \ref{Theorem: HFK and the minimal number of fixed points ireducible}]
 We show that the result holds for $HF^{\sharp}(\vf)$: Theorem \ref{Theorem: Main theorem in introduction} and the fact that $\widehat{HFK}(Y,K,-g+1) \cong \widehat{HFK}(\overline{Y},\overline{K},-g+1)$ will imply the theorem. By last proposition and Lemma \ref{Lemma: F_min = dim HF}, it is enough to check that $p_h$ and the boundary parallel Dehn twists equally contribute to $\mathcal{F}_{min}([\vf])$ and $\dim(HF^{\sharp}(\vf))$. This can be easily done in each of the different cases considered in the proof of Proposition \ref{Proposition: dim HF and dim HF^sharp}, using Lemma \ref{Lemma: F_min = sum of indexes} and the fact that if a circle of fixed points has an its own Nielsen class then the Poincar\'e--Birkhoff fixed point theorem implies that it contributes by $+2$ to $\mathcal{F}_{min}([\vf])$.
\endproof

As explained in Remark \ref{Remark: Monotonicity for reducible maps}, we can not directly generalize the argument above to reducible case. 
Still, in some case the last theorem can be generalized to reducible maps. 

\begin{Thm} \label{Theorem: HFK and the minimal number of fixed points homology 3-sphere}
 Let $Y$ be a rational homology sphere and $(K,S,\vf)$ an open book decomposition of $Y$ with $g(S)\geq 1$.
 Then:
 \begin{equation*}
  \dim\left(\widehat{HFK}(Y,K;-g+1)\right) =  \mathcal{F}_{min}({[\vf]}) + 1.
 \end{equation*} 
\end{Thm}
\begin{proof}
 If $Y$ is a homology $3$--sphere any symplectic and non-degenerate representative of $[\vf]$ is also monotone: this follows directly either from the definition of monotonicity given in \cite{Se2} or from Lemma 3.2 of \cite{CC2}. In particular a symplectomorphism $\vf_{min}$ realizing $\mathcal{F}_{min}([\vf])$ is monotone and, by definition, Equation \eqref{Equation: dim of HF leq f_min} holds also in this case. 
 
 On the other hand Proposition 4.9 of \cite{CC2} implies that there exists a representative $\vf_{sm}$ of $[\vf]$ such that $\#\Fix(\vf_{sm})$ equals the rank of a \emph{twisted version} $HF(\vf_{sm};Q(\Z/(2\Z)[M]))$ of the symplectic homology of $\vf_{sm}$. Here the twisted coefficients are taken over the field of fractions of the group ring $\Z/(2\Z)[M]$ where $M$ is a given subgroup of $\ker(\vf_* - id) \subset H_1(T_{\vf},\R)$ (see \cite[§3.1]{CC2} for the details). If $Y$ is a homology sphere then $\ker(\vf_* - id) = \{0\}$, $HF(\vf_{sm};Q(\Z/(2\Z)[M]))$ reduces to the standard $HF(\vf_{sm})$ and Equations \eqref{Equation: dim of HF geq f_min} and \eqref{Equation: F_min = dim HF} also hold.  
 
 The proof of the theorem is then analogue to that of Theorem \ref{Theorem: HFK and the minimal number of fixed points ireducible}. Observe that we do not need to care about the special case of Theorem \ref{Theorem: HFK and the minimal number of fixed points ireducible} because the condition $\ker(\vf_* - id) = \{0\}$ implies that if $A \subset S$ is a non-contractible component on which the canonical representative of $[\vf]$ is the identity map then $A$ is a boundary-parallel annulus, on which the monodromy can be perturbed to give only a circle of fixed points that can be treated as in the proof of the aforementioned theorem.   
 \end{proof}

A direct consequence of Proposition \ref{Proposition: dim HF and dim HF^sharp} and the last two theorems is the following result. 
\begin{Cor}
 If $(K,S,\vf)$ is an open book decomposition of a $3$--manifold $Y$ with $g(S)>0$ then 
 \begin{equation} \label{Equation: dim HFK_-g+1 geq 1}
  \dim\left(\widehat{HFK}(Y,K;-g+1)\right) \geq 1.
 \end{equation} 
 Moreover, if $[\vf]$ is irreducible or $Y$ is a rational homology sphere then the equality holds if and only if $[\vf]$ admits a symplectic non-degenerate representative with no fixed points in $\Int(S)$. 
\end{Cor}

We remark that Equation \eqref{Equation: dim HFK_-g+1 geq 1} has recently been proved by Baldwin and Shea Vela-Vick in \cite{BS}.

Last corollary has an interesting consequence about the topology the of $L$--space knots. We recall that an $L$--space is a rational homology sphere $Y$ such that the rank of $\widehat{HF}(Y)$ coincides with the number of elements of $H_1(Y,\Z)$. An $L$--space knot is a knot in $S^3$ that admits a non-trivial surgery to an $L$--space. 


In \cite{OS7}, Ozsv\'{a}th and Szab\'{o} proved that if $K$ is an $L$--space knot then 
\[\dim\left(\widehat{HFK}(S^3,K,i)\right) \leq 1\]
for every $i$. In particular if $K$ is an $L$--space knot the inequality \eqref{Equation: dim HFK_-g+1 geq 1} is sharp and the following holds.
\begin{Cor}
 Let $K \subset S^3$ be an $L$-space knot whose complement fibers with fiber $S$ and monodromy $\vf$. Then $[\vf]$ admits a symplectic non-degenerate representative with no fixed points. 
\end{Cor}

Theorem \ref{Theorem: Main theorem in introduction} has also interesting applications in the study of the geometric and topological informations that one can extract from the family of knot Floer homologies $\widehat{HFK}(Y^n(K),K^n)$, where $K^n$ is the branched locus of the $n$-th branched cover $Y^n(K)$ of $Y$ over $K$. If $(K,S,\vf)$ is an open book decomposition of $Y$ then $(K^n,S,\vf^n)$ is an open book decomposition of $Y^n(K)$, so that, by Theorem \ref{Theorem: Main theorem in introduction}, we have
\begin{equation}\label{Equation: HFK iso to HF for each iterated}
 \widehat{HFK}(Y^n(K),K^n;-g+1) \cong HF^{\sharp}(\vf^n).
\end{equation}
We remark that, while Heegaard Floer homology of double branched covers has been studied in different situations (like in Manolescu--Owens \cite{MO} and Ozsv\'{a}th--Szab\'{o} \cite{OS6}), the case of higher branched covers appear less often in the literature. 


From Theorem 1.1 of Fel'shtyn \cite{Fel}, we can recover the following result, already proven by Lipshitz, Ozsv\'{a}th and Thurston using bordered Floer homology techniques (cf. Corollary 4.2 and Proposition 3.18 of \cite{LOT2} and Theorem 14 of \cite{LOT1}).
\begin{Cor}\label{Corollary: HFK detects the stretching factor}
 Given a collection of vector spaces $\{A_n\}_{n=1}^{\infty}$, define the \emph{growth rate} of (the dimensions of) $A_n$ by
\[\mathcal{GR}(A_n) \coloneqq  \limsup_{n\rightarrow\infty} \big(dim(A_n)\big)^{\frac{1}{n}}.\]
 If $(K,S,\vf)$ is an open book decomposition of $Y$, then  
 \[ \mathcal{GR}\left(\widehat{HFK}(Y^n(K),K^n;-g+1)\right) = \lambda_{[\vf]}\]
 where $\lambda_{[\vf]}$ is the largest dilatation factor among all the pseudo--Anosov components of the canonical representative of the mapping class $[\vf]$.
\end{Cor}

By classical results about $3$--dimensional geometry and topology (see for example \cite{FLP}) we get the following nice relation between Heegaard Floer homology and $3$--dimensional geometry (cf. Cotton-Clay \cite[Corollary 1.7]{CC}).
\begin{Cor}
 If $K \subset Y$ is a fibered knot then $ \dim\big(\widehat{HFK}(K^n,Y^n(K);-g+1)\big)$ grows exponentially if and only the JSJ decomposition of $Y \setminus K$ has some hyperbolic component. 
\end{Cor}

Another interesting consequence of Theorem \ref{Theorem: Main theorem in introduction} and Proposition \ref{Proposition: dim HF and dim HF^sharp} is about algebraic knots and comes from an analogous result of McLean about fixed point Floer homology (see \cite[Corollary 1.4]{McLean}).
\begin{Cor}\label{Corollary: HFK detects the multiplicity}
 Let $K \subset S^3$ be the $1$--component link of an isolated singularity of an irreducible complex polynomial $f$ with two variables. Then
 \[\min \left\{n>0\ |\ \dim(\widehat{HFK}(Y^n(K),K^n;-g+1)) \neq 1 \right\} = \mathfrak{m}(f) \]
 where $\mathfrak{m}(f)$ is the multiplicity of $f$.
\end{Cor}

\end{document}